\documentclass{article}
\usepackage[
style=alphabetic,
mincitenames=4,
maxcitenames=99,
maxbibnames=99,
maxalphanames=99,
doi=false,isbn=false,url=false,eprint=false
]{biblatex}
\usepackage{pdflscape}
\usepackage{graphicx}
\usepackage{amsmath,amsthm,amsfonts,amscd}
\usepackage{eucal} 	 	
\usepackage{verbatim}      	
\usepackage{algorithm}
\usepackage{algorithmic}
\usepackage{xcolor}
\usepackage{markdown}
\usepackage{booktabs}
\usepackage{url}

\usepackage{multirow}
\usepackage{makecell}
\usepackage{tabularx}
\usepackage{bbm}
\usepackage{mathtools}
\usepackage{hyperref}
\hypersetup{
    colorlinks,
    linkcolor={black},
    citecolor={black},
    urlcolor={black}
}
\usepackage{tikz-cd}

\newtheorem{theorem}{Theorem}
\newtheorem{lemma}{Lemma}[section]

\theoremstyle{definition}
\newtheorem{definition}[lemma]{Definition}

\theoremstyle{remark}
\newtheorem*{remark}{Remark}

\numberwithin{figure}{section}

\DeclareMathOperator{\C}{\mathbb{C}}
\DeclareMathOperator{\R}{\mathbb{R}}

\DeclareMathOperator{\Z}{\mathbb{Z}}

\DeclareMathOperator{\Hom}{Hom}
\DeclareMathOperator{\Rep}{Rep}

\DeclareMathOperator{\Hit}{Hit}
\DeclareMathOperator{\hol}{hol}
\DeclareMathOperator{\supp}{supp}

\DeclareMathOperator{\SL}{SL}
\DeclareMathOperator{\PSL}{PSL}
\DeclareMathOperator{\del}{\partial}

\DeclareMathOperator{\im}{Im}
\DeclareMathOperator{\tr}{tr}

\title{Boundary Currents of Hitchin Components}
\author{Charles Reid}

\begin{document}
\maketitle

\begin{abstract}
The space of Hitchin representations of the fundamental group of a closed surface $S$ into $\text{SL}_n\mathbb{R}$ embeds naturally in the space of projective oriented geodesic currents on $S$. We find that currents in the boundary have combinatorial restrictions on self-intersection which depend on $n$. We define a notion of dual space to an oriented geodesic current, and show that the dual space of a discrete boundary current of the $\text{SL}_n\mathbb{R}$ Hitchin component is a polyhedral complex of dimension at most $n-1$. 
For endpoints of cubic differential rays in the $\text{SL}_3\mathbb{R}$ Hitchin component, the dual space is the universal cover of $S$, equipped with an asymmetric Finsler metric which records growth rates of trace functions. 

\end{abstract}

\tableofcontents
\pagebreak

\section{Introduction}
Let $S$ be a closed oriented surface of genus at least $2$, and let $\mathcal{T}(S)$ be the Teichm\"uller space of hyperbolic metrics on $S$. Thurston defined a compactification \cite{Thurston88, FLP79} which makes the open ball $\mathcal{T}(S)$ into a closed ball whose boundary points parametrize measured laminations on $S$ \cite{Bonahon88}. 
Dually, boundary points give metric $\R$-trees with $\Gamma$ action, and every $\R$-tree with $\Gamma$ action, with virtually cyclic edge stabilizers, arises this way \cite{MorganShalen91, Skora96}. 

Let $\Gamma:=\pi_1 (S)$. A discrete, faithful representation $\Gamma\to \PSL_2\!\R$ gives rise to an oriented hyperbolic surface $\Gamma\backslash\mathbb{H}^2$ marked by $S$. In this way, $\mathcal{T}(S)$ is identified with a connected component of $\Rep(\Gamma,\PSL_2\!\R)$. 

One might wonder whether other spaces of representations of finitely generated groups $\Gamma$ into non-compact Lie groups $G$ have compactifications analogous to Thurston's, and whether boundary points of these compactifications parameterize geometric objects. In the case where $\Gamma$ is an arbitrary finitely generated group, and $G$ is a rank one lie group, one still finds trees with $\Gamma$ action at the boundary \cite{CullerShalen83, Bestvina88, Paulin88}. 

When $G$ is a semisimple Lie group with real rank $r > 2$, there is a an analogous compactification called the Weyl length compactification \cite{Parreau12} whose boundary points can be interpreted as $\Gamma$ actions on $r$-dimensional polyhedral complexes called affine buildings. The caveat is that there is a huge amount of choice involved in representing a boundary point as a building with $\Gamma$ action, so while actions on buildings do give geometric insight into the Weyl length compactification, this perspective does not realize the Weyl length compactification as a moduli space. 

In the last decade there has been an effort, for example \cite{parreau_invariant_2015, Le16, BIPP21}, to show that Weyl length boundaries of specific components of specific character varieties parameterize more tractable, canonically defined objects. 
The present paper is part of this effort, though in contrast we do not work directly with buildings and instead construct spaces from scratch, generalizing Morgan and Shalen's construction of $\R$-trees dual to measured laminations.

This paper concerns $\PSL_n\!\R$ Hitchin components: For $n\geq 2$, composition with the unique irreducible representation 
$\PSL_2\!\R \to \PSL_n\!\R$, gives an embedding 
\[\mathcal{T}(S)\to \Rep(\Gamma,\PSL_n\!\R)\] 
and $\Hit^n\!(S)$ is defined to be the component of $\Rep(\Gamma,\PSL_n\!\R)$ containing this copy of $\mathcal{T}(S)$. Hitchin \cite{Hitchin92} discovered, using the nonabelian Hodge correspondence, that this component has trivial topology:
\[\Hit^n(S)\simeq \R^{(2g-2)(n^2 - 1)}\]
Hitchin components are the archetypal examples of ``higher Teichm\"uller spaces": components of character varieties of surfaces consisting entirely of discrete and faithful representations \cite{fock_moduli_2006, Labourie_anosov}.  We will study a generalization of Thurston's compactification to $\Hit^n(S)$ called the spectral radius compactification. The spectral radius compactification coincides with the Weyl chamber length compactification for $n=3$ but is in general coarser.

\subsection{Spectral radius compactification}
The spectral radius of a matrix $M$, is $|\lambda_1(M)|$ where $\lambda_1(M)$ is the greatest magnitude eigenvalue of $M$. Sending a representation of $\Gamma$ to its complete list of spectral radii gives a map from $\Hit^n(S)$ to functions on the set $[\Gamma]$ of conjugacy classes of $\Gamma$. 
\[\Hit^n(S)\to \R^{[\Gamma]}\]
\[\rho \mapsto (l_\rho: \gamma\mapsto \log|\lambda_1(\rho(\gamma))|)\]
We call the class function $l_\rho$ the marked length spectrum of $\rho$, because for $n=2$, $2\log|\lambda_1(\rho(\gamma))|$ is the hyperbolic length of the geodesic in the homotopy class specified by $[\gamma]\in [\Gamma]$. Let $[l_\rho]$ be the projectivized marked length spectrum: the image of $l_\rho$ in $\mathbb{P}(\R^{[\Gamma]})$.
The map $\rho\mapsto [l_\rho]$ is an embedding of $\Hit^n(S)$ into $\mathbb{P}(\R^{[\Gamma]})$ and has compact closure. This closure is the spectral radius compactification, and we denote its boundary by $\del_{\lambda_1}\! \Hit^n(S)$.

Note that we get the same compactification if we use trace functions $\tr(\rho(\gamma))$ in place of spectral radii $\lambda_1(\rho(\gamma))$. This is sometimes desirable, as trace functions are algebraic, but in our context spectral radii are more natural.

For $n=2$, $\del_{\lambda_1}\! \Hit^2(S)$ is the Thurston boundary of Teichm\"uller space which consists of translation length spectra of actions on trees. For $n>2$, these length spectra corresponding to trees comprise a small part of the boundary, but the rest of the boundary length spectra are less understood. 

\subsection{Geodesic currents at infinity}
There is an embedding of $\Hit^n(S)$ into a different infinite dimensional space, namely the space of projective geodesic currents, which will give a slightly coarser compactification, but with more geometric understanding of boundary points. Bonahon introduced geodesic currents for exactly this purpose in the case $n=2$. In that case the compactifications exactly coincide. In \cite{Bonahon88}, Bonahon worked with unoriented currents whereas we will use oriented currents throughout. 

A geodesic current is an invariant measure on the space of geodesics in the universal cover of $S$. To make sense of ``geodesics" without choosing a metric on $S$, one uses the Gromov boundary $\del\Gamma$. Recall that $\del\Gamma$ is homeomorphic to a circle. A hyperbolic structure on $S$ gives an identification $\xi: \del\Gamma \to \R \mathbb{P}^1$ and any two such identifications differ by a H\"older homeomorphism. A hyperbolic structure also gives rise to an identification of the space of geodesics in $\tilde{S}$ with \[\mathcal{G} := \del\Gamma \times \del\Gamma \backslash \Delta,\] the space of pairs of distinct points in the Gromov boundary.

\begin{definition}
    A geodesic current on $S$ is a locally finite, $\Gamma$-invariant Borel measure on $\mathcal{G}$. The space of geodesic currents on $S$ with the weak topology is denoted $\mathcal{C}(S)$.
\end{definition}
Geodesic currents generalize closed curves. For any $\gamma\in \Gamma$ there is an attracting fixed point $\gamma^+\in \del\Gamma$ and a repelling fixed point $\gamma^-\in \del\Gamma$. The point $(\gamma^-,\gamma^+)\in \mathcal{G}$ is called the axis of $\gamma$. Let $\delta_\gamma$ denote the delta measure at $(\gamma^-,\gamma^+)$. If $C\in [\Gamma]$ is the conjugacy class of a primative element, define
\[\delta_{C} = \sum_{\gamma\in C} \delta_\gamma.\]
Delta currents are defined for non-primitive classes by requiring $\delta_{[\gamma^n]} = n\delta_{[\gamma]}$. These currents span a dense subspace of $\mathcal{C}(S)$.

Various natural geometric structures on $S$, most notably negatively curved metrics, and Anosov representations, give rise to geodesic currents. The space of projective geodesic currents $\mathbb{P}(\mathcal{C}(S))$ is compact, thus giving compactifications of moduli spaces of structures on $S$.

A natural embedding $\Hit^n(S)\to \mathcal{C}(S)$, introduced in \cite{Labourie07}, is defined using limit maps. A Hitchin representation $\rho:\Gamma\to \SL(V)$, with $V\simeq \R^n$ is in particular projective Anosov, meaning that there are continuous equivariant limit maps
\[\xi:\del\Gamma \to \mathbb{P}(V)\]
\[\xi^*:\del\Gamma\to \mathbb{P}(V^*)\]
such that $\xi(a^+)$ is the eigenline of top eigenvalue of $\rho(a)$, and $\xi^*(a^+)$ is the eigenline of top eigenvalue of $\rho(a^{-1})^*$, and $ \xi^*(x)$ contains $\xi(y)$ if and only if $x=y$. These limit maps define a measure $\mu_\rho\in \mathcal{C}(S)$ by
\[\mu_\rho([x_1,x_2]\times [y_1,y_2]) = \log|\frac{\langle \xi^*(x_1),\xi(y_1)\rangle\langle \xi^*(x_2),\xi(y_2)\rangle}{\langle \xi^*(x_1),\xi(y_2)\rangle\langle \xi^*(x_2),\xi(y_1)\rangle}|\] 
where $x_1,x_2,y_1,y_2\in \del\Gamma$ are cyclically ordered.  For $n=2$, $\xi$ and $\xi^*$ are both the usual identification of $\del\Gamma$ with the boundary of the hyperbolic plane, and $\mu_\rho$ is the Lioville measure for the hyperbolic metric. 



The closure of $\mathcal{T}(S)$ in $\mathbb{P}(\mathcal{C}(S))$ is also the Thurston compactification. Bonahon showed this using his intersection pairing on currents
\[i:\mathcal{C}(S)\times \mathcal{C}(S)\to \R\]
which extends geometric intersection number of unoriented closed curves. An essential formula is that intersecting a hyperbolic structure with a closed curve gives length:
\[i(\mu_\rho, \delta_{[\gamma]}) = \text{Len}_\rho(\gamma)\]
In other words, the marked length spectrum map 
$\mathcal{T}(S)\to \R^{[\Gamma]}$ 
factors through the embedding $\mathcal{C}(S)^{\Z/2}\to \R^{[\Gamma]}$ sending a symmetric current $\mu$ to the intersection function $i(\mu_\rho, -)$. 
\begin{center}
\begin{tikzpicture}

    \node (Q1) at (0,1.6) {$\mathcal{C}(S)^{\Z_2}$};
    \node (Q2) at (-1.6,0) {$\mathcal{T}(S)$};
    \node (Q3) at (1.6,0) {$\R^{[\Gamma]}$};

    \draw [->] (Q2)--(Q1);
    \draw [->] (Q1)--(Q3);
    \draw [->] (Q2)--(Q3);

    \end{tikzpicture}
\end{center}
Since the closure of $\mathcal{T}(S)$ is already compact in $\mathbb{P}(\mathcal{C}(S))$, one gets the same closure in $\R^{[\Gamma]}$. 

This strategy doesn't work beyond $n=2$ because $i(\mu_\rho, -)$ gives the symmetrized length length spectrum of $\rho$, so the diagram does not commute. There is however a natural embedding $\mathcal{C}(S) \to \R^{[\Gamma]}/\Hom(\Gamma,\R)$ which makes the following diagram commute.
\begin{center}
\begin{tikzpicture}

    \node (Q1) at (0,1.6) {$\mathcal{C}(S)$};
    \node (Q2) at (-1.6,0) {$\Hit^n(S)$};
    \node (Q3) at (1.6,0) {$\R^{[\Gamma]}/\Hom(\Gamma,\R)$};

    \draw [->] (Q2)--(Q1);
    \draw [->] (Q1)--(Q3);
    \draw [->] (Q2)--(Q3);

    \end{tikzpicture}
\end{center}
The map $\mathcal{C}(S)\to \R^{[\Gamma]}/\Hom(\Gamma,\R)$ will be defined in section \ref{equivariant bundles and geodesic currents} and will come from a correspondence
\begin{center}
\begin{tikzpicture}

    \node (Q1) at (0,1.6) {$\mathcal{A}(S)$};
    \node (Q2) at (-1.6,0) {$\mathcal{C}(S)$};
    \node (Q3) at (1.6,0) {$\R^{[\Gamma]}$};

    \draw [->] (Q1)--(Q2);
    \draw [->] (Q1)--(Q3);

\end{tikzpicture}
\end{center}
where we will define $\mathcal{A}(S)$ to be the set of $\Gamma$ equivariant principal $\R$ bundles with ``connection" on $\mathcal{G}$ of non-negative curvature. The left map takes a bundle to its curvature, which is a geodesic current, and the right map takes a bundle to its period spectrum. The key to showing this correspondence gives an embedding $\mathcal{C}(S)\to \R^{[\Gamma]}/\Hom(\Gamma,\R)$ is Lemma \ref{abc lemma} which gives an asymmetric counterpart to the classic formula \ref{hilbert length box} expressing symmetrized periods as a cross ratio. 

Unfortunately, some of the spectral radius compactification is collapsed under the quotient $\mathbb{P}(\R^{[\Gamma]}) \to \mathbb{P}(\R^{[\Gamma]}/\Hom(\Gamma,\R))$, so the compactification of $\Hit^n(S)$ in projective currents is a non-trivial quotient of the spectral radius compactification. However, we will see that this quotient map is bijective when restricted to cubic differential ray limit points in $\del_{\lambda_1}\Hit^3(S)$. Despite this difference, studying the current compactification is a helpful stepping stone to understanding the spectral radius compactification.



Currents in the boundary of $\mathcal{T}(S)$ are measured laminations, i.e. symmetric currents with no self intersection. This self intersection condition has a direct generalzation for $\SL_n\!\R$.
\begin{theorem}
\label{INTRO no n intersection}
    If $[\mu]\in \mathbb{P}(\mathcal{C}(S))$ is in the boundary of $\Hit^n(S)$, then there can not be $(x_1,y_1),...,(x_n,y_n)\in \supp(\mu)$, with $x_1 < \dots < x_n < y_1 <\dots  < y_n$.
\end{theorem}

\tikzset{every picture/.style={line width=0.75pt}} 

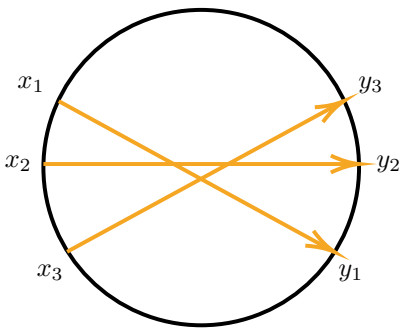
\begin{figure}
\centering

\tikzset{every picture/.style={line width=0.75pt}} 

\tikzset{every picture/.style={line width=0.75pt}} 

\tikzset{every picture/.style={line width=0.75pt}} 

\begin{tikzpicture}[x=0.75pt,y=0.75pt,yscale=-1,xscale=1]

\draw  [line width=1.5]  (37.51,172.4) .. controls (37.51,128.44) and (73.14,92.81) .. (117.1,92.81) .. controls (161.06,92.81) and (196.7,128.44) .. (196.7,172.4) .. controls (196.7,216.36) and (161.06,252) .. (117.1,252) .. controls (73.14,252) and (37.51,216.36) .. (37.51,172.4) -- cycle ;
\draw [color={rgb, 255:red, 245; green, 166; blue, 35 }  ,draw opacity=1 ][line width=1.5]    (49.47,214.89) -- (186.28,140.32) ;
\draw [shift={(188.92,138.89)}, rotate = 151.41] [color={rgb, 255:red, 245; green, 166; blue, 35 }  ,draw opacity=1 ][line width=1.5]    (14.21,-4.28) .. controls (9.04,-1.82) and (4.3,-0.39) .. (0,0) .. controls (4.3,0.39) and (9.04,1.82) .. (14.21,4.28)   ;
\draw [color={rgb, 255:red, 245; green, 166; blue, 35 }  ,draw opacity=1 ][line width=1.5]    (37.51,170.61) -- (193.7,170.61) ;
\draw [shift={(196.7,170.61)}, rotate = 180] [color={rgb, 255:red, 245; green, 166; blue, 35 }  ,draw opacity=1 ][line width=1.5]    (14.21,-4.28) .. controls (9.04,-1.82) and (4.3,-0.39) .. (0,0) .. controls (4.3,0.39) and (9.04,1.82) .. (14.21,4.28)   ;
\draw [color={rgb, 255:red, 245; green, 166; blue, 35 }  ,draw opacity=1 ][line width=1.5]    (45.29,138.89) -- (181.5,213.45) ;
\draw [shift={(184.13,214.89)}, rotate = 208.7] [color={rgb, 255:red, 245; green, 166; blue, 35 }  ,draw opacity=1 ][line width=1.5]    (14.21,-4.28) .. controls (9.04,-1.82) and (4.3,-0.39) .. (0,0) .. controls (4.3,0.39) and (9.04,1.82) .. (14.21,4.28)   ;

\draw (23,125) node [anchor=north west][inner sep=0.75pt]   [align=left] {$\displaystyle x_{1}$};
\draw (16.48,165) node [anchor=north west][inner sep=0.75pt]   [align=left] {$\displaystyle x_{2}$};
\draw (32.64,220) node [anchor=north west][inner sep=0.75pt]   [align=left] {$\displaystyle x_{3}$};
\draw (184.8,220) node [anchor=north west][inner sep=0.75pt]   [align=left] {$\displaystyle y_{1}$};
\draw (204,165) node [anchor=north west][inner sep=0.75pt]   [align=left] {$\displaystyle y_{2}$};
\draw (195,125) node [anchor=north west][inner sep=0.75pt]   [align=left] {$\displaystyle y_{3}$};

\end{tikzpicture}

\caption{A forbidden configuration for a current in the boundary of $\Hit^3(S)$}
\end{figure}

This will be proved in Lemma \ref{no n-intersection} using the tropicalization of Labourie's determinant relations \cite{Labourie07}. 


\subsection{Geometry at infinity}
There is another perspective on the Thurston compactification: every point in the Thurston boundary is dual to an $\R$-tree with $\Gamma$ action. Rational points correspond to genuine trees while irrational points are more exotic. The translation lengths of $\Gamma$ acting on this tree are renormalized limits of hyperbolic lengths. Generalizing this perspective to $\SL_n\!\R$ is the main goal of this paper. 

From a tropical rank $n$ geodesic current $\mu$ we define (Definition \ref{X definition}) a metric space $X_\mu$ with $\Gamma$ action, such that the translation length of $\gamma\in \Gamma$ is $l(\gamma) + l(\gamma^{-1})$, where $l\in \R^{[\Gamma]}$ is the length spectrum corresponding to $\mu$. We call $X_\mu$ the \textbf{universal asymmetric dual space} to $\mu$.

Of course, it would be better to have $l(\gamma)$ instead of $l(\gamma) + l(\gamma^{-1})$. We find a slightly more technical way to encode the actual asymptotic $\lambda_1$-spectrum geometrically. From a boundary point $l\in \R^{[\Gamma]}$ of the spectral radius compactification which maps to the tropical rank $n$ current $\mu$, we construct a principal $\R$-bundle $L$ over $X_\mu$ together with a ``reletive metric" $d:L\times L\to \R$ (Definition \ref{definition relative metric}), such that the translation length of $\gamma\in \Gamma$ acting on $L$ is $l(\gamma)$.

In fact, $X_\mu$ is defined for any geodesic current which is ``nullhomologous", a large class of currents containing symmetric currents, though $X_\mu$ is exceptionally nice when $\mu$ is a tropical rank $n$ current. First we state what we can prove about $X_\mu$, then comment on its definition.


It is still largely a mystery what $X_\mu$ can look like, but we know a few things. A current $\mu$ is called tropical rank $n$ (Definition \ref{tropical rank n current}) if it satisfies the tropicalization of Labourie's determinant relations \cite{Labourie07}. Currents in the boundary of $\Hit^n(S)$ are tropical rank $n$, but we do not know the converse.
\begin{theorem} [Lemma \ref{Xmu is dimension n-1}]
    If $\mu$ is a discrete tropical rank $n$ geodesic current, then $X_\mu$ is a polyhedral complex of dimension at most $n-1$.
\end{theorem}
From the definition of $X_\mu$ this is far from obvious. The definition of $X_\mu$ makes sense for a broad class of geodesic currents, including those coming from negatively curved metrics and from Anosov representations, but usually $X_\mu$ will be infinite dimensional. We expect that for any tropical rank $n$ current, $X_\mu$ has dimension at most $n-1$, but we only know how to formulate and prove this for discrete currents.

For $n=2$ we recover the well known story that Thurston boundary points parametrize $\R$-trees with $\Gamma$-action. 
\begin{theorem} 
    Tropical rank $2$ currents are measured laminations, and if $\mu$ is a measured lamination $X_\mu$ is an $\R$-tree.
\end{theorem}
Lemma \ref{no n-intersection} shows that a tropical rank $2$ current $\mu$ has no self-intersection, Lemma \ref{vanishing triple ratios implies symmetry} shows $\mu$ is symmetric, thus a measured lamination, and Lemma \ref{measured lamination implies R tree} shows that $X_\mu$ is an $\R$-tree.

For $n=3$, there are certain boundary points for which $X_\mu$ turns out to be naturally in bijection with $\tilde{S}$, namely endpoints of cubic differential rays. If we equip $S$ with a complex structure and a non-zero holomorphic cubic differential $\alpha$ we can define a certain Higgs bundle $(E,\phi_\alpha)$.
This Higgs bundle lies in the Hitchin section, thus by solving Hitchin's equation we get a representation in $\Hit^3(S)$. In fact, $\Hit^3(S)$ is parametrized in this way by pairs of complex structure and cubic differential on $S$ \cite{labourie_flat_2006, Loftin01}. Paths in $\Hit^3(S)$ coming from a ray of cubic differentials $\{R\alpha : R>0 \}$ are called cubic differential rays. 
\begin{theorem}
\label{cubic ray endpoint}
    Let $\alpha$ be a non-zero cubic differential on $S$. For $R>0$, let $\mu(R)$ be the geodesic current coming from the Hitchin representation corresponding to the Higgs bundle $(E,\phi_{R\alpha})$. Then as $R$ goes to infinity, $\mu(R)/R$ converges to the current of real trajectories of $\alpha$.
\end{theorem}
Theorem \ref{cubic ray endpoint} is mostly a corollary of Theorem B from \cite{Reid2023}, (see also \cite{loftin_limits_2022}) where it was shown that the endpoint in the spectral radius compactification of the ray specified by $\alpha$ is the length spectrum of an explicit Finsler metric $F^\Delta_\alpha$ with triangular unit balls, which is flat on the compliment of the zeros of $\alpha$. 
In Lemma \ref{cubic current corridor} and Lemma \ref{cubic current support} it is shown that the Lioville current of $F^\Delta_\alpha$ is the current of real trajectories of $\alpha$.  Finally, we show that $X_\mu$ is very nice when $\mu$ is a cubic differential current.
\begin{theorem}
    If $\mu$ is the current of real trajectories of $\alpha$, then there is a $\Gamma$ equivariant map $\tilde{S}\to X_\mu$ which is an isometry for the symmetrized metric $F^\Delta_\alpha + F^\Delta_{-\alpha}$.
\end{theorem}
The map is shown to be a bijection in Lemma \ref{holonomy zero means comes from point}, and subsequently shown to be an isometry.

The idea that convex projective structures on $S$ degenerate to flat structures on $S$ with singularities has some history. 
Parreau used Fock-Goncharov coordinates to show that certain paths to infinity limit to spaces that have 1 and 2 dimensional parts \cite{parreau_invariant_2015}. 
In \cite{OuyangTamburelli21} it is shown that, along cubic differential rays, the Blaschke metric limits to a flat Riemannian metric with cone singularities. 
In fact, they compactify $\Hit^3(S)$ by embedding into $\mathbb{P}(\R^{[\Gamma]})$ by taking length spectra of Blaschke metrics. 
They show that cone metrics comprise an open dense subset of the boundary, and describe the rest of the boundary as certain mixed structures. 

In \cite{Reid2023} we showed that length spectra of triangular Finsler metrics arise in the spectral radius compactification. Marked length spectrum rigidity is less understood for of Finsler metrics than for Riemannian metrics, so the possibility remained that different triangular Finsler metrics have the same length spectra. The current paper eliminates this possibility, thus showing that a subset of the spectral radius boundary of $\Hit^3(S)$ is the space of triangular Finsler metrics up to scale. By analogy with \cite{OuyangTamburelli21}, we conjecture that this subset is open and dense.   

\subsection{Heuristic description of $X_\mu$}
Now we give an impressionistic definition of $X_\mu$; see Definition \ref{X definition} for the actual definition. The definition is similar in spirit to the definition of the dual space to a geodesic current in \cite{BIPP21} and further studied in \cite{LucaDidac2023dualspacesgeodesiccurrents} but importantly $X_\mu$ does not depend on a choice of background hyperbolic metric. To start, suppose  $\mu$ is symmetric. Choosing a hyperbolic metric $g$ on $S$, each support point of $\mu$ becomes an unoriented geodesic in $\tilde{S}$. Let $C_\mu(g)$ be the set of complementary regions of the union of all such geodesics. Let the distance between two regions be the measure of the set of geodesics separating them. With this metric, $C_\mu(g)$ is more or less the same as the dual space from \cite{BIPP21}. 

\begin{figure}[h]
    \centering
    \includegraphics[width=0.8\linewidth]{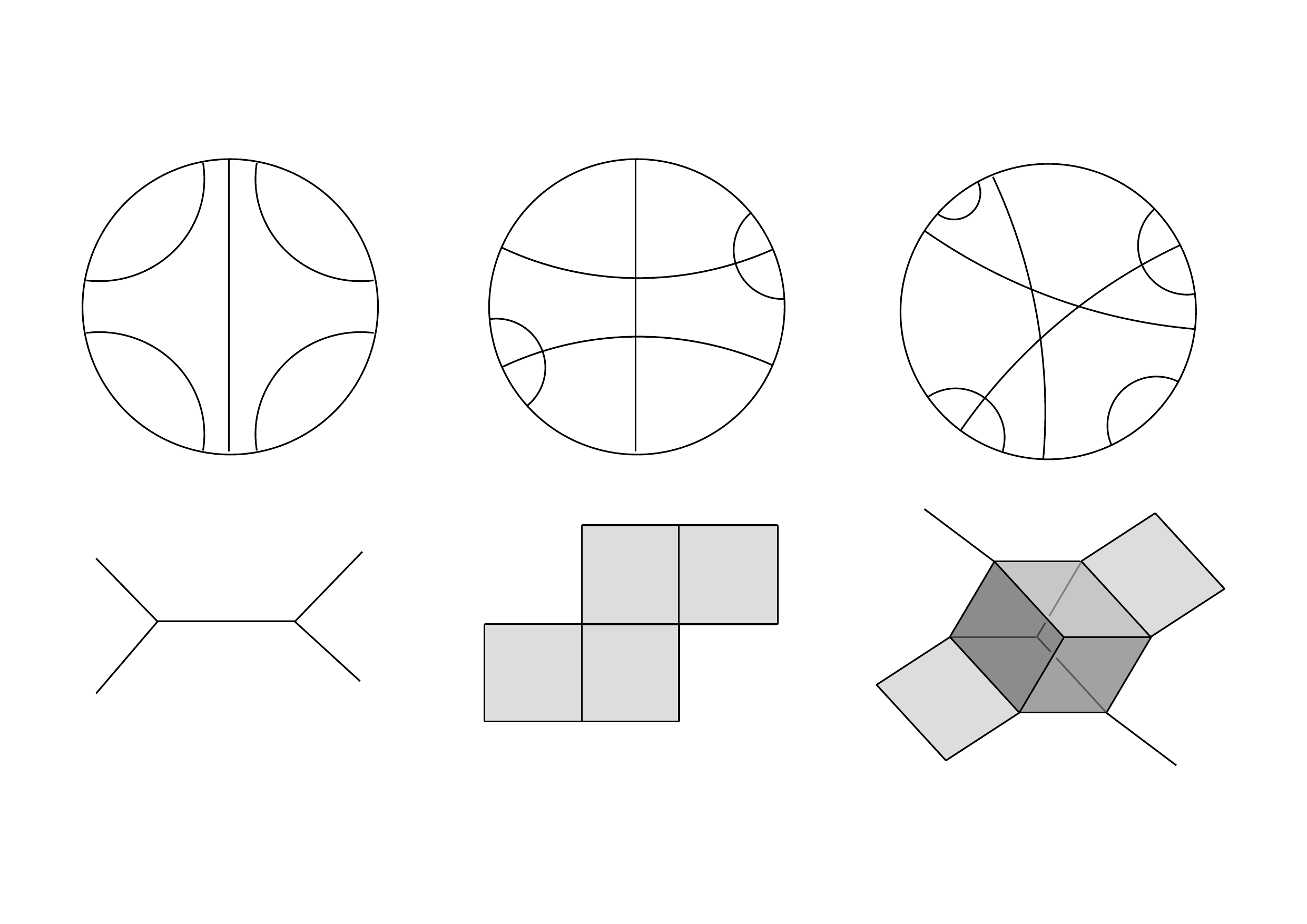}
    \caption{Examples of universal symmetric dual spaces to some unoriented geodesic currents on the disk. The first two examples are no different from the dual spaces defined in \cite{BIPP21}, but there is a difference in the third example. }
\end{figure}

As $g$ varies, complementary regions can appear and disappear. Let $C_\mu$ denote the set of all complementary regions which could possibly appear after drawing an unoriented chord for each support point of $\mu$ without any two chords bounding a bigon. It still makes sense to measure how many geodesics pass between two elements of $C_\mu(g)$. Now, allow even more ``complementary regions" by allowing chords to be arbitrarily split, i.e. a chord carrying measure $a+b$ can be split into parallel chords carrying measures $a$ and $b$ with the same endpoints, call this metric space of complementary regions $X_\mu^{sym}$. For symmetric currents, $X_\mu^{sym}$ would be a good alternative to $X_\mu$. We call $X_\mu^{sym}$ the \textbf{universal symmetric dual space} to $\mu$. The dual spaces of \cite{BIPP21}, for various hyperbolic metrics, are various isometrically embedded slices of $X_\mu^{sym}$. 

It would be quite interesting to look at $X_\mu^{sym}$ for symmetric currents arising from limits of representations in $SO(n,n)$ and $Sp(2n,\R)$. In particular, Theorem \ref{INTRO no n intersection} implies that if $\mu$ is a discrete boundary current of the $Sp(4,\R)$ Hitchin component, then $X^{sym}_\mu$ is a 2 dimensional cube complex, and in fact coincides with the BIPP dual space which in this special case doesn't depend on the choice of hyperbolic metric.

Now allow $\mu$ to be asymmetric, and imagine drawing an oriented chord for every support point of $\mu$, allowing for splittings, and only forbidding bigons which have both sides oriented parallelly, and let $X'_\mu$ denote the set of all complementary regions which can arise. If it is possible to find a $\Gamma$-invariant function $h:X'_\mu\to \R$ such that the change in $h$ from one region to the next is the signed measure of chords passing between them, then $\mu$ is called nullhomologous, and $h$ is called a holonomy function. From a holonomy function, one can construct a class in $H^2(S,\R)$ and there is a unique $h$ for which this class is zero. Let $X_\mu$ be the subset of $X'_\mu$ on which $h=0$. This $X_\mu$ is the universal asymmetric dual space to $\mu$. 

\begin{figure}[h]
    \centering
    \includegraphics[width=0.8\linewidth]{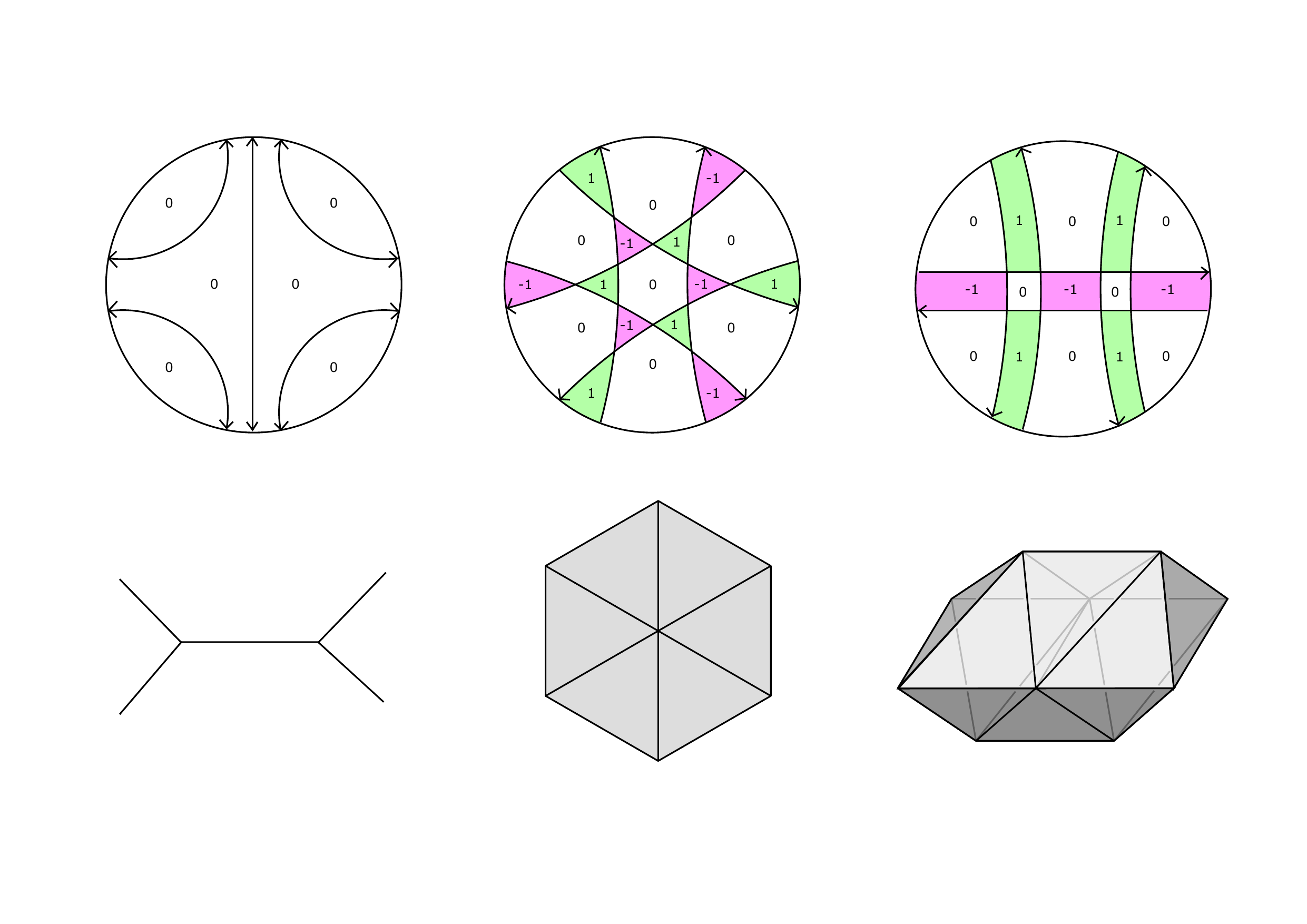}
    \caption{Examples of $X_\mu$ for some oriented geodesic currents on the disk, with holonomy functions. Generally, when $\mu$ is symmetric, $X_\mu$ is bigger than $X_\mu^{sym}$, but when $\mu$ is a measured lamination, $X_\mu = X_\mu^{sym}$.}
    
\end{figure}

\subsection{Future Directions}
This paper raises at least as many questions as it answers. Here are a few such questions.
\begin{itemize}
\item Do all tropical rank $n$ currents arise in the boundary of $\Hit^n(S)$?
\item Is there a nice, complete characterization of the multicurves on $S$ which are tropical rank $n$ currents, or (if there is a difference) $\Hit^n(S)$ boundary currents? 
\item What does $X_\mu$ look like in general, for $\mu$ a tropical rank $n$ current? What about specifically for $n=3$? Can we classify local models?
\item Is there a version of $X_\mu$ corresponding to the Weyl chamber length compactification of $\Hit^n(S)$, i.e. a version which keeps track of the ordered list of all eigenvalues $(\log |\lambda_1|, ..., \log|\lambda_n|)$?
\item Is there a version of $X_\mu$ for $G\neq \SL_n\!\R$?
\item If $B$ is a building with $\Gamma$ action lifting a boundary current $\mu$, what is the relationship between $X_\mu$ and $B$?
\end{itemize}

\subsection{Acknowledgments}
I am especially grateful to Jeff Danciger for many mathematical discussions and for thorough feedback during many stages of this project. This project benefited from discussions with many people, in particular Fran\c{c}ois Labourie, Anne Parreau, Beatrice Pozzetti, Xenia Flamm, Max Riestenberg, and Dan Freed. I am grateful also to the stimulating research environments of UT Austin, Harvard University, and the MPI for Mathematics in the Sciences. This research was supported in part by NSF grants DMS-1937215, and DMS-1945493.

\section{Equivariant bundles and geodesic currents}
\label{equivariant bundles and geodesic currents}
Geodesic currents which arise in nature tend to be curvatures of equivariant principal $\R$ bundles on $\mathcal{G}$. The central example is that the Lioville current of a negatively curved metric on $S$ is the curvature of a natural connection on the unit tangent bundle $T^1\tilde{S}$, viewed as a principal $\R$ bundle, where $\R$ acts by geodesic flow. We will see that currents coming from Anosov representations are also curvatures of bundles. One might ask whether any geodesic current can be realized as the curvature of an equivariant bundle. An initial obstruction is that a geodesic currents can assign non-zero measure to a fixed point $(\gamma^-,\gamma^+)\in \mathcal{G}$ meaning that the connection must have a singularity at that point. For this reason we will work with bundles with connection on the complement of all fixed points.
\[\mathcal{G}^\circ := \mathcal{G} \backslash\{ (\gamma^-,\gamma^+) : \gamma\in \Gamma\backslash\{1\}\}\]

We will factor the curvature map into two steps, and investigate the relationships between three objects: geodesic currents, positive holonomy functions, and equivariant bundles with positive taxi connection on $\mathcal{G}^\circ$. There are forgetful maps relating these three objects.
\begin{equation}
 \parbox{11em}{\centering equivariant bundles with \\ positive taxi connection} \longrightarrow 
 \parbox{10em}{\centering positive holonomy functions} \longrightarrow 
 \parbox{10em}{\centering geodesic currents}
\end{equation}
Denote the sets of (isomorphism classes of) these objects by $\mathcal{A}(S)$, $\mathcal{H}(S)$, and $\mathcal{C}(S)$ respectively. Geodesic currents were introduced in \cite{Bonahon88} to give a more efficient treatment of Thurston's compactification, and have been used extensively since. The other two objects are variations on existing notions. Positive holonomy functions generalize cross ratios \cite{Labourie07}, while equivariant bundles with positive taxi-connection are similar to reparametrizations of geodesic flow \cite{Sambarino14, BCLS2018}. Both have been used to study Anosov representations.


The three infinite dimensional cones $\mathcal{A}(S)$, $\mathcal{H}(S)$, and $\mathcal{C}(S)$ play similar roles. Moduli spaces of interest, for example $\Hit^n(S)$, have natural embeddings into these infinite dimensional spaces, providing new structure on said moduli spaces. Neither of the maps in the sequence $\mathcal{A}(S)\to \mathcal{H}(S)\to \mathcal{C}(S)$ is injective or surjective, but the discrepancies can be well understood.
In particular, any geodesic current which maps to zero in $H^1(S,\R)$ (in a sense we will define) is the curvature of some equivariant bundle, and any two equivariant bundles with the same curvature only differ by a change in equivariance.

\subsection{Taxi connections}
We will define a notion of $\Gamma$ equivariant bundles with connection on $\mathcal{G}^\circ$. The curvature of such a bundle will be a geodesic current.

\begin{definition}
    A \textbf{segment} of $\mathcal{G}$ is a subset of the form $s_{x,x';y} := [x,x']\times \{y\}$, where $y < x\leq x' < y$, or $s_{x;y,y'} := \{x\}\times [y,y']$ where $x < y \leq y' < x$. 
\end{definition}

\begin{definition} Let $A$ be a group. A \textbf{taxi connection} $F$ on a principal $A$ bundle $P$ over $\mathcal{G}^\circ$ is a $G$-orbit $F(s)$ of sections over every segment $s\subset\mathcal{G}^\circ$ which is compatible with restriction to subsegments. We refer to $F(s)$ as the flat sections over $s$.
\end{definition}

This definition gives a notion of parallel transport along ``taxi-paths" i.e. concatenations of horizontal and vertical segments.  In this paper, $A$ will always be an Abelian Lie group, usually $\R$ or $\R^*$. Recall that principal bundles with Abelian structure group can be tensored:
If $A$ is an Abelian group, and $P$ and $Q$ are principal $A$ bundles on a space $X$, their tensor product $P\otimes Q$ is the quotient of the fiber product $P\times_X Q$ by the equivalence relation $(p\cdot a,q) \sim (p,q\cdot a)$ for $a\in A$. If $P$ and $Q$ are principal $A$ bundles with taxi connection on $\mathcal{G}$ or $\mathcal{G}^\circ$ then $P\otimes Q$ inherits a taxi connection, thus isomorphism classes of $A$ bundles with taxi connection form an Abelian group. Just as the deRahm complex is useful in studying smooth connections, the following chain complex will be useful for taxi connections. 

\begin{definition} Define the chain complex $T_2(\mathcal{G}) 
\overset{\del}{\longrightarrow} T_1(\mathcal{G}) \overset{\del}{\longrightarrow} T_0(\mathcal{G})$ as follows.
\begin{itemize} 
    \item Let $T_0(\mathcal{G})$ denote the free Abelian group on points of $\mathcal{G}^\circ$.
    \item Let $T_1(\mathcal{G})$ denote the free Abelian group on the set of segments in $\mathcal{G}^\circ$, modulo the subgroup generated by 
    $s_{x;y,y'} + s_{x;y',y''} - s_{x;y,y''}$ for all $x < x' < x'' < y$, and $s_{x,x';y} + s_{x',x'';y} - s_{x,x'';y}$  for all $y < y' < y'' < x$.
    \item Let $T_2(\mathcal{G})$ denote the free Abelian group generated by boxes $r_{x,x';y,y'} := [x,x']\times[y,y']\subset \mathcal{G}$ with boundary in $\mathcal{G}^\circ$ modulo the subgroup generated by
    $r_{x,x';y,y'} + r_{x',x'';y,y'} - r_{x,x'';y,y'}$ for $x < x' < x'' < y < y'$ and 
    $r_{x,x';y,y'} + r_{x,x';y',y''} - r_{x,x';y,y''}$
    for $ y < y' < y'' < x < x'$.
    \item Define $\del:T_1(\mathcal{G})\to T_0(\mathcal{G})$ by
    \[\del s_{x,x';y} := (x',y) - (x,y)\]
    \[\del s_{x;y,y'} := (x,y') - (x,y)\]
    \item Define $\del:T_2(\mathcal{G})\to T_1(\mathcal{G})$ by
    \[\del r_{x,x';y,y'} = s_{x;y,y'} + s_{x,x';y'} - s_{x';y,y'} -  s_{x,x';y}\]
\end{itemize}
\end{definition}
One checks that $\del^2 = 0$ by evaluating it on an arbitrary rectangle $r_{x,x';y,y'}\in T^2(\mathcal{G})$.
We abuse notation and denote a segment and the corresponding generator in $T^1(\mathcal{G})$ the same way. 
It is an easy consequence of the definition of $T_1(\mathcal{G})$ that $s_{x,x;y} = 0$, and $s_{x,x';y} + s_{x',x;y} = 0$. The corresponding statements hold for vertical segments and for rectangles. 

Note that $\mathcal{G}$ is homotopy equivalent to a circle. The next lemma shows that the homology of $T_*(\mathcal{G})$ is the homology of a circle.

\begin{lemma}
\label{homology}
    The homology of $T_*(\mathcal{G})$ is $\Z,\Z,0$ in degrees $0,1,2$.
\end{lemma}
\begin{proof}

Fix a point $p\in \mathcal{G}^0$. Any point $q\in \mathcal{G}^\circ$ is homologous to $p$ because $\mathcal{G}^\circ$ is taxi-path connected. This proves that $H_0(T_*(\mathcal{G}))\simeq \Z$. 

Fix a taxi loop $l$ in $\mathcal{G}^\circ$ based at $p$ which projects monotonically and with degree $1$ to both $\del \Gamma\simeq S^1$ factors. For instance, let $l = s_{x,x';y} + s_{x';y,y'} + s_{x',x;y'} + s_{x;y',y}$ where $x < y < x' < y'$ are cyclically ordered points in $\del \Gamma^\circ$. Any horizontal segment $s_{x,x';y}\subset \mathcal{G}^\circ$ is homologous in $T_1(\mathcal{G})$ to a sum of vertical segments, and horizontal segments contained in $l$. A closed 1-cycle comprised only of vertical segments and segments in $l$ will be equal in $T_1(\mathcal{G})$ to a cycle consisting only of segments in $l$. Such a cycle must be an integral multiple of $l$, so $H^0(T_*(\mathcal{G}))\simeq \Z$.

Suppose $c$ is a 2-chain with $\del c = 0$. By definition, $c = \sum a_{x,x';y,y'} r_{x,x';y,y'}$ with $a_{x,x';y,y'}\in \Z$, and all but finitely many $a_{x,x';y,y'}$ are zero. Let $(x_i)$ and $(y_j)$ be cyclically ordered lists of coordinates which appear in coefficients $a_{x,x';y,y'}$ which are non-zero. We can rewrite $c$ as $\sum c_{ij}r_{i(i+1);j(j+1)}$. The only way for $\del c$ to be zero is for $c_{ij}$ to be all equal, but the rectangle $r_{i(i+1);j(j+1)}$ only exists when $x_i < x_{i+1} < y_j < y_{j+1}$ which is not true for all $i,j$. We conclude that all $c_{ij}$ are zero, so $c=0$. 
\end{proof}

\begin{definition}
For an Abelian group $A$, let $T^i(\mathcal{G},A) := \Hom(T_i(\mathcal{G}), A)$ and let $d$ be the dual differential.
\end{definition}

A taxi connection $F$ on the trivial bundle $\mathcal{G}^\circ\times A$ gives a cochain $t_F\in T^1(\mathcal{G})$. The connection $F$ is a choice of $A$ torsor $F(s)$ of functions $s\to A$ over each segment $s$, compatible with restriction. Define
\[t_F(s) = f(\del_+ s) - f(\del_- s)\]
where $f\in F(s)$, and $\del_{\pm} s$ are the front and back endpoints of $s$.

\begin{lemma} The map $F\mapsto t_F$ from taxi connections on the trivial bundle to $T^1(\mathcal{G},A)$ is a bijection.
\end{lemma}
\begin{proof}
    We give an inverse. Let $t\in T^1(\mathcal{G},A)$.
    For a horizontal segment $s_{x,x';y}$, define $F(s_{x,x';y})$ to be the set of functions $f:s\to A$ such that for any two points $x\leq p < q \leq x'$ we have $f((q,y))-f((p,y))=t(s_{p,q;y})$. The function $f((p,y)) = t(s_{x,p;y})$ will satisfy this property, so $F(s)$ is non-empty. It is easy to see that two elements of $F(s)$ must differ by a constant, and that $F$ is compatible with restriction. Define $F(s)$ for vertical segments in the same way.
\end{proof}

The cochain $t_F$ corresponding to a connection $F$ on the trivial bundle can be thought of as parallel transport. If $c\in T^1(\mathcal{G})$ is a taxi path i.e. a sum of segments $s_1 + ... + s_k$ with $\del_+ s_i = \del_- s_{i+1}$ for $i=1,...,k-1$, then $t_F(c)$ is the parallel transport along this path. If $c$ is a loop, i.e. $\del_+ s_k = \del_- s_1$, then $t_F(c)$ is the holonomy of this loop. 

Let $Z_1(\mathcal{G}) \subset T_1(\mathcal{G})$ denote the kernel of $\del$. Its elements will be referred to as taxi-cycles.

\begin{lemma}
\label{bundles classified by holonomy}
    Let $A$ be an Abelian group. Principal $A$-bundles with taxi connection on $\mathcal{G}^\circ$ are classified up to isomorphism by $\Hom(Z_1(\mathcal{G}),A)$.
\end{lemma}
\begin{proof}
    Any $A$-bundle admits a section, (recall that we are not considering $A$ bundles with taxi-connection as having topology) and any two trivializations differ by addition of a function $f:\mathcal{G}^\circ \to A$. This means we can just study connections on the trivial bundle, quotiented by the action of functions. 

Let $F$ be a taxi connection on the trivial bundle, and let $t_F\in T^1(\mathcal{G}, A)$ be the corresponding cochain. If $f:\mathcal{G}^\circ\to A$ is a function, and $F + f$ is the taxi connection obtained by change of trivialization, then
\[t_{F+f} = t_F + df.\]
Thus, $A$ bundles with taxi connection are classified by
$T^1(\mathcal{G},A)/d T^0(\mathcal{G},A)$. 
Applying $\Hom(-,A)$ to the exact sequence \[0\to Z_1(\mathcal{G}) \to T_1(\mathcal{G}) \to T_0(\mathcal{G}),\]
we see $T^1(\mathcal{G},A)/d T^0(\mathcal{G},A) = \Hom(Z_1(\mathcal{G}),A)$.
\end{proof}
We refer to the element of $\Hom(Z_1(\mathcal{G}),A)$ corresponding to $(P,F)$ as the holonomy function $h_{(P,F)}$ of $(P,F)$. We have shown that taxi-connections are determined by holonomy, and every holonomy function is realizable.

We use the following notion of curvature:
\begin{definition} 
The curvature of a taxi-connection $F$ on an $A$-bundle $P$ 
is the 2-cocycle $curv(F)\in T^2(\mathcal{G},A)$ which assigns $h_F(\del r)$ 
to every rectangle $r$ with boundary in $\mathcal{G}^\circ$. 
\end{definition}

If we trivialize $P$, then $F$ is equivalent to the cochain $t_F\in T^1(\mathcal{G}, A)$, and $curv(F)$ is simply $d t_F$. This means that bundles with zero curvature are classified by $H^1(T^*(\mathcal{G},A))\cong A$ and all curvatures are realized by bundles because $H^2(T^*(\mathcal{G},A)) = 0$. 

For various arguments, we will want curvature to be a measure on $\mathcal{G}$ which when integrated over a rectangle gives holonomy around the boundary. For this we will assume curvature is positive. Let $\mathcal{C}(S)$ denote the space of geodesic currents on $S$. 

\begin{lemma} \label{positive cocycles are measures}
    Evaluating on rectangles with boundary in $\mathcal{G}^\circ$ gives a map $\mathcal{C}(S)\to T^2(\mathcal{G},\R)$ which is a bijection onto cocycles which take positive values on all rectangles, and are $\Gamma$ invariant.
\end{lemma} 
\begin{proof}
    Let $c$ be a $\Gamma$ invariant, positive cocycle. We can define a content $\mu$ on the semiring of subsets of $\mathcal{G}$ of the form $(x_1,x_2]\times(y_1,y_2]$ with boundary in $\mathcal{G}$ by defining $\mu((x_1,x_2]\times(y_1,y_2]) = c(r_{x_1,x_2;y_1,y_2})$. If $\mu$ is $\sigma$-additive, then it defines a unique measure. To prove $\sigma$-additivity, it suffices to show that $c(r_{x_1,x_2;y_1,y_2})$ is continuous on the space of rectangles with boundary in $\mathcal{G}^\circ$. If it is discontinuous, then we can find a segment $s$ in $\mathcal{G}^\circ$ such that $c$ evaluates to at least $\epsilon$ on any rectangle containing $s$. This contradicts $\Gamma$ invariance of $c$.
\end{proof}

To define positivity for $A\neq \R$ we must assume $A$ is endowed with a partial order. We let $\mathcal{C}^2(\mathcal{G},A)_+ \subset T^2(\mathcal{G},A)$ denote the space of cocycles such that $c(r)\in A_{\geq 0}$ for all rectangles $r$ with boundary in $\mathcal{G}^\circ$, and call elements of $\mathcal{T}^2(\mathcal{G},A)_+$ positive cocycles. We are using the word ``positive" to mean positive or zero, by analogy with positive measures. 

In order to have Lemma \ref{positive cocycles are measures} for $A$, we must assume that $A$ is bounded complete, that is every bounded subset has a least upper bound. The space of invariant, positive cocycles, $(T^2(\mathcal{G},A)_+)^\Gamma$, is then identified with the space of geodesic currents valued in $A_{\geq 0}$ which we denote $\mathcal{C}(S,A)$, or $\mathcal{C}(S,A_{\geq 0})$ to emphasize that it really depends on the partial order. In this paper $A$ will be $\R$, $\R^*$ or $\Z$.

\begin{definition}
    Let $\mathcal{A}(S,A)$ denote the space of equivariant $A$ bundles with taxi-connection on $\mathcal{G}^\circ$ with positive curvature.
\end{definition}
\begin{definition}
Let $\mathcal{H}(S,A)$ denote the cone in $\Hom(Z_1(\mathcal{G}),A)^\Gamma$ of invariant holonomy functions which are positive or zero on boundaries of rectangles, and call elements of $\mathcal{H}(S,A)$ \textbf{positive holonomy functions}.
\end{definition}
The relation between positive holonomy functions and Labourie's cross ratios is as follows. For $x,x',y,y'\in \del \Gamma^\circ$ with $x,x'$ both distinct from $y,y'$, let 
\[[x,x';y,y'] := s_{x;y,y'} + s_{x,x';y'} + s_{x';y',y} + s_{x',x;y}\in T_1(\mathcal{G}).\]
Note that $[x,x';y,y']$ is the boundary of a rectangle when the points are ordered correctly, but not always.
If $h\in H(S,\R^*)$, and $h([x,x';y,y'])$ extends to a H\"older function 
\[B:\{x,x',y,y'\in \del\Gamma | x\neq y', y\neq x'\} \to \R\]
which vanishes along $x=y$ and $x'=y'$, then $B$ is a cross ratio in the sense of Labourie.

By convention, $\mathcal{A}(S)$ and $\mathcal{H}(S)$ denote the versions with $A=\R$. The sequence
\[\mathcal{A}(S,A)\to \mathcal{H}(S,A)\to \mathcal{C}(S,A)\]
is now defined, and we can compute kernels and cokernels.

\subsection{Geodesic currents and holonomy functions}
We will show that the map from positive holonomy functions to geodesic currents
\[c: \mathcal{H}(S,A) \to \mathcal{C}(S, A)\]
has kernel $A$ and cokernel $\Hom(\Gamma,A_0)\simeq (A_0)^{2g}$ where $A_0$ is the identity component of $A$. The kernel is the group of holonomy functions which vanish on all boundaries of rectangles. Call these flat holonomy functions.

\begin{lemma} \label{flat holonomy functions}
    Evaluation on a taxi-loop which generates $H_1(T_*(\mathcal{G}))$ gives an isomorphism from the group of flat holonomy functions to $A$.
\end{lemma}
\begin{proof}
The curvature map $c$ is the restriction of $d:\Hom(Z_1(\mathcal{G}), A) \to T^2(\mathcal{G},A)$ which has kernel $H^1(T^*(\mathcal{G},A))$ which is the same as $\Hom(H_1(T_*(\mathcal{G})),A)$ which is identified with $A$ by evaluation on a generating taxi-loop. We must check that this copy of $A$ is contained in $\mathcal{H}(S,A)$. Since $S$ is oriented, $\Gamma$ acts trivially on first homology of $\mathcal{G}$, so the kernel of $d$ is contained in $\mathcal{H}(S,A)$. 
\end{proof}

The cokernel of $c$ is a bit more tricky to see. Every geodesic current can be lifted to $\Hom(Z_1(\mathcal{G}),A)$, but not always to a $\Gamma$-invariant element. We will construct a map $\mathcal{C}(S,A_{\geq 0})\to \Hom(\Gamma,A)$ whose kernel is the image of $c$. This map sends a current to its Poincare dual cohomology class. We use the signed variant of Bonahon's intersection product to make this precise. 

\begin{definition}
Let $I:\mathcal{G}\times \mathcal{G}\to \{0,1\}$ be $1$ on pairs of geodesics which intersect and zero on pairs which don't intersect. Let $I_{sgn}:\mathcal{G}\times \mathcal{G} \to \{-1,0,1\}$ be $1$ on pairs which intersect with positive orientation, $-1$ on pairs which intersect with negative orientation, and $0$ on pairs which don't intersect. If two geodesics share one or both endpoints, then they are considered not to intersect.
Bonahon's intersection product of two geodesic currents $\mu_1,\mu_2$ is
\[i(\mu_1,\mu_2) := \int_{\mathcal{G}\times\mathcal{G}/\Gamma} I(g_1,g_2) \mu_1\otimes \mu_2\]
whereas the signed intersection product is
\[i_{sgn}(\mu_1,\mu_2) := \int_{\mathcal{G}\times\mathcal{G}/\Gamma} I_{sgn}(g_1,g_2) \mu_1\otimes \mu_2\]
\end{definition}
Note that $i$ is symmetric, whereas $i_{sgn}$ is antisymmetric. In the case where one of the arguments is $\delta_{[\gamma]}$ for $\gamma\in \Gamma$ primitive, the formula is a bit simpler. 
\[i(\delta_{[\gamma]},\mu) := \int_{\mathcal{G}/\gamma} I((\gamma^-,\gamma^+),g) \mu\]
\[i_{sgn}(\delta_{[\gamma]},\mu) := \int_{\mathcal{G}/\gamma} I_{sgn}((\gamma^-,\gamma^+),g) \mu\]
When we write $\mu_1 \otimes \mu_2$ we use the multiplication $\R \times \R \to \R$. We can similarly define the intersection, or signed intersection of an $A$ valued current with a $\Z$ valued current with the multiplication map $\Z\times A \to A$. The signed intersection product is rarely used because it vanishes on most geodesic currents of interest, for example those coming from Anosov representations or negatively curved metrics. The next lemma explains why this is the case. 

\begin{lemma}
\label{balence}
The map $\gamma\mapsto i_{sgn}(\delta_{[\gamma]},\mu)$ is a homomorphism from $\Gamma$ to $A$. For $h\in \Hom(Z_1(\mathcal{G}),A)$ lifting $\mu$ to be invariant, $\mu$ must satisfy $i_{sgn}(\delta_{[\gamma]},\mu) = 0$ for all $\gamma\in \Gamma$. 
\end{lemma}
\begin{proof}
    Fix an arbitrary geodesic current $\mu$. Let
    \[T := \{h\in \Hom(Z_1(\mathcal{G}),A) : dh = \mu\}\]
    be the set of lifts of $\mu$ to a function of all cycles. The set $T$ is a coset for the group of flat holonomy functions, which is $A$ by Lemma \ref{flat holonomy functions}.  Since $\mu$ is invariant, $\Gamma$ acts on $T$, giving a homomorphism $\phi_\mu: \Gamma\to A$. Since $A$ is Abelian, $\phi_\mu$ descends to a map $H_1(S) \to A$. The geodesic current $\mu$ can be upgraded to an invariant holonomy function only when $\phi_\mu = 0$.

    Now we show that $\phi_\mu(\gamma) = i_{sgn}( \delta_{[\gamma]},\mu)$. To compute $\phi_\mu(\gamma)$ it is sufficient to compute $h(\gamma\cdot z) - h(z)$ for any $h\in T$, and any taxi cycle $z\in T_1(\mathcal{G})$ which generates the homology of $\mathcal{G}$. Figure \ref{holonomy change} shows an example choice of $z$. The difference in holonomy is the signed measure between the two cycles. The two big rectangles will contribute $i_{sgn}(\delta_{[\gamma]},\mu)$. The small rectangles are necessary to avoid $(\gamma^-,\gamma^+)$ and $(\gamma^+,\gamma^-)$ and stay in $\mathcal{G}^\circ$ but contribute nothing, as will be explained in Lemma \ref{period}.
\begin{figure}[h]
	\label{holonomy change}
    \centering
    \includegraphics[width=5cm]{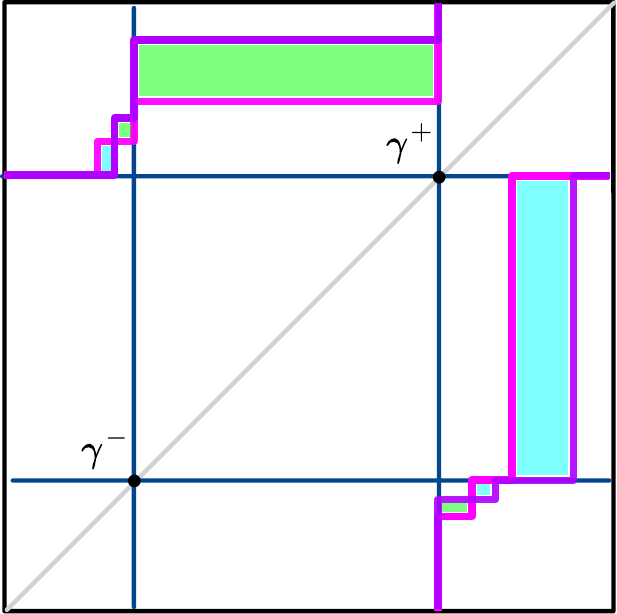}
    \caption{Area between a taxi-loop winding around $\mathcal{G}$ and its $\gamma$ translate}
    
\end{figure}
\end{proof}
The homomorphism $\gamma \to i_{sgn}(\delta_{[\gamma]},\mu)$ is actually valued in the subgroup $A_{\pm} = A_{\geq 0} + A_{\leq 0}$. The resulting map $\mathcal{C}(S,A) \to \Hom(\Gamma,A_{\pm})$ is a surjection, because any cohomology class is poincare dual to an oriented multicurve with $A_{\geq 0}$ weights. The cokernel of the curvature map $\mathcal{H}(S,A)\to \mathcal{C}(S,A)$ is thus $(A_{\pm})^{2g}$.

\subsection{Holonomy functions and equivariant bundles}
Now we show that the forgetful map from equivariant bundles with taxi connection to holonomy functions has kernel $A^{2g}$ and cokernel $A$. 

\begin{lemma}
    There is an exact sequence.
    \[0\to \Hom(\Gamma,A) \to \mathcal{A}(S,A) \to \mathcal{H}(S,A) \to H^2(\Gamma,A)\to 0\]
\end{lemma}

In the rest of this section it will be helpful use lemma \ref{bundles classified by holonomy} to equate holonomy functions with isomorphism classes of non-equivariant bundles with taxi-connection with $\Gamma$-invariant, positive holonomy. The question becomes ``when can a bundle be made equivariant, and how many ways are there to do so?"
\begin{definition}
\label{central extension definition}
Let $P\in \mathcal{H}(S,A)$ be a non-equivariant $A$-bundle with positive invariant holonomy. Let $\Gamma_{P}$ be the group consisting of pairs $(\gamma,\phi)$ where $\gamma\in \Gamma$, and $\phi:P\to P$ is a bundle map covering $\gamma$ preserving the taxi connection.
\[
\begin{tikzcd}
    P \arrow{r}{\phi} \arrow{d} & P  \arrow{d} \\
    \mathcal{G}^\circ \arrow{r}{\gamma} & \mathcal{G}^\circ
\end{tikzcd}
\]
\end{definition}

The group $\Gamma_{P}$ is a central extension of $\Gamma$ by $A$, and a splitting $\Gamma\to \Gamma_P$ is precisely the data making $P$ equivariant. 
Isomorphism classes of central extensions of $\Gamma$ by $A$ are classified by $H^2(\Gamma,A)$, which is the same as $H^2(S,A)$, which can be identified with $A$ by pairing with the fundamental class of $S$. 
We denote the map $H^2(\Gamma,A)\to A$ by $\mathcal{\chi}$ because if $B\to S$ is an oriented circle bundle, $\pi_1(B)$ will be a central extension of $\Gamma$ by $\Z$, and $\chi(\pi_1(B))$ will be the Euler class of $B$, which is often denoted $\chi(B)$.

Recall that the group structure on $H^2(G,A)$, where $G$ is a group and $A$ is an Abelian group, corresponds to the operation on central extensions
\[[\tilde{G}_1] + [\tilde{G}_2] = [\tilde{G}_1 \underset{G}{\times} \tilde{G}_2 / A] \]
where $A\subset \tilde{G}_1 \underset{G}{\times} \tilde{G}_2$ is the subgroup $\{(a,a^{-1}):a\in A\}$.

\begin{lemma}
    If $P$ and $Q$ are $A$ bundles on $\mathcal{G}$ with taxi connection, with $\Gamma$-invariant curvature, then 
\[[\Gamma_{P\otimes Q}] = [\Gamma_P] + [\Gamma_Q].\]
\end{lemma}
\begin{proof}
    The fiber product $\Gamma_P \underset{\Gamma}{\times} \Gamma_Q$ is the group of triples $(\gamma,\phi_1,\phi_2)$, where $\gamma\in \Gamma$, and $\phi_1:P\to P$, and $\phi_2:Q\to Q$ both cover the action of $\gamma$ action on $\mathcal{G}$. Quotienting this fiber product by the subgroup $\{(1, a,-a) : a\in A\}$ gives $\Gamma_{P\otimes Q}$. 
\end{proof}

So we have a homomorphism from $\mathcal{H}(S,A)$, to $H^2(S,A) = A$. We next check that this homomorphism is non-trivial on the subgroup of flat bundles. Let $l\in Z_1(\mathcal{G})$ be a taxi loop representing the generator of first homology of $T_*(\mathcal{G})$ which agrees with the orientation of $\del \Gamma$ induced by the orientation of $S$. The subgroup of flat bundles is identified with $A$ via measuring holonomy around $l$. 

\begin{lemma}
\label{holonomy gives extension}
    If $Q\in \mathcal{H}(S,\Z)$ is a flat bundle on $\mathcal{G}$ with holonomy $1$, then $\chi(\Gamma_Q) = 2-2g$.
\end{lemma}
\begin{proof}
    For convenience, fix a negatively curved metric on $S$ so that we may talk about its unit tangent bundle $T^1S$. There is a commutative square 

\[
\begin{tikzcd}
    \widetilde{T^1 S} \arrow{r}\arrow{d} & \tilde{\mathcal{G}} \arrow{d} \\
    T^1 \tilde{S} \arrow{r} & \mathcal{G} \\
\end{tikzcd}
\]
    where the tildes denote universal cover. The horizontal arrows are principal $\R$ bundles (quotients by geodesic flow) and the vertical arrows are principal $\Z$ bundles. 
    The top row of the diagram has an action of $\tilde{\Gamma} := \pi_1(T^1 S)$ which is a central extension of $\Gamma$ with $\chi(\tilde{\Gamma}) = \chi(S) = 2-2g$. In fact $\tilde{\Gamma}$ is precisely the group of bundle maps of $\tilde{\mathcal{G}}$ which cover elements of $\mathcal{G}$ from definition \ref{central extension definition}. Note that $\tilde{\mathcal{G}}\to\mathcal{G}$ is a flat $\Z$ bundle with holonomy $1$.
\end{proof}

More generally, Lemma \ref{holonomy gives extension} implies that a flat $A$ bundle $Q$ on $\mathcal{G}$ with holonomy $a\in A$ has $\chi(\Gamma_Q) = (2-2g)a$, as $Q$ is induced from $\tilde{\mathcal{G}}$ by the homomorphism $\Z\to A$ with $1\mapsto a$. 

\begin{lemma}
    The curvature map $\mathcal{A}(S,A)\to \mathcal{C}(S,A)$ is surjective onto nullhomologous currents.
\end{lemma}
\begin{proof}
    A nullhomologous current $\mu$ may be realized by an invariant holonomy function, which in turn is the holonomy of a non-equivariant $A$ bundle $P$ with taxi-connection. 
    There is a unique flat $A$ bundle $Q$ on $\mathcal{G}$ such that $\chi(\Gamma_Q) = -\chi(\Gamma_P)$, thus $Q\otimes P$ admits $\Gamma$ equivariance. Thus, there exists a $\Gamma$ equivariant bundle with curvature $\mu$. 
\end{proof}

\begin{lemma}
    Any two equivariant singular $A$-bundles with connection with curvature $\mu$ differ by a modification of the $\Gamma$ action by an element of $\Hom(\Gamma,A)$.
\end{lemma}
\begin{proof}
    The action of $\Hom(\Gamma,A)$ on equivariant $A$ bundles with connection can be thought of as tensoring with trivial $A$ bundles with possibly non-trivial equivariance. Let $P$ and $Q$ be equivariant $A$ bundles with the same curvature. The equivariant $A$ bundle $\Hom(P,Q)$ will be flat. By Lemma \ref{holonomy gives extension} it must have trivial holonomy, but it may be acted on non-trivially by $\Gamma$, so is described by an element of $\Hom(\Gamma,A)$. Finally note $Q = P\otimes \Hom(P,Q)$. 
\end{proof}

\section{Geodesic currents and length spectra}
If we have a $\Gamma$ equivariant $A$ bundle $P$ on $\mathcal{G}$, the period $l_P(\gamma)$ of $\gamma\in \Gamma$ is the amount that it translates the fiber over the fixed point $(\gamma^+,\gamma^-)$. 
We will extend this notion to equivariant bundles with taxi-connection over $\mathcal{G}^\circ$. Since $P$ is equivariant, $l_P(\gamma)$ will only depend on the conjugacy class of $\gamma$. We will see that the map


\[
\begin{tikzcd}
    \mathcal{A}(S,A) \arrow{r}  & A^{[\Gamma]} \\
\end{tikzcd}
\]
given by $P\mapsto l_P$ is injective. That is, equivariant bundles with positive taxi-connection are determined by their periods.

\subsection{Defining periods}

By definition, an equivariant bundle with taxi connection has no fibers over fixed points, so we need to come up with an alternative definition for periods. The intuition behind the definition is that one can measure translation lengths of group elements by measuring how horocycles centered at fixed points are acted upon.
\begin{lemma}
\label{period}
    Let $P\in \mathcal{A}(S)$. Let $\gamma \in \Gamma$. Consider the four rays emanating from $(\gamma^{-},\gamma^{+})$. 
    \[R_N = \{ (\gamma^-, y) : \gamma^+ < y < \gamma^-\} \;\;\;\;\;
    R_S = \{ (\gamma^-, y) : \gamma^- < y < \gamma^+\} \]
    \[R_E = \{ (x, \gamma^+) : \gamma^- < x < \gamma^+\} \;\;\;\;\;
    R_W = \{ (x, \gamma^+) : \gamma^+ < x < \gamma^-\} \]
    There is a number $l_P(\gamma)$, which we call the \textbf{period} of $\gamma$, such that if $s$ is a flat section of $P$ over any of these rays, we have $\gamma\cdot s = l_P(\gamma)s$.
\end{lemma}
\begin{proof}
    Since $\gamma$ preserves each of these rays, and must send flat sections to flat sections, we get a period for each ray, but it isn't obvious that these four periods coincide.
    We illustrate why sections over $R_N$ and $R_E$ are shifted the same amount. Let 
    \[p = s_{\gamma^-, x; y} \cup s_{x;y,\gamma^+}\]
    be a two segment taxi path in $\mathcal{G}^\circ$ connecting $R_N$ to $R_E$, (see Figure \ref{period well defined}). Choose a flat section $\sigma$ over $p$. The translate $\gamma \sigma$ will be a flat section over $\gamma p$.
    The holonomy around the figure-8 taxi-cycle obtained by joining $p$ and $\gamma p$ with a segment of $R_N$ and a segment of $R_E$ will be the difference in the periods of $\gamma$ measured over $R_N$ and $R_E$, as one sees by using $\sigma$, and $\gamma \sigma$ to make a lift.
    This holonomy coincides with the difference in measure of two boxes with respect to the curvature of $P$
    \[\mu_P([\gamma ^-, \gamma  x]\times [\gamma ^+, \gamma  y]) - \mu_P([\gamma ^-, x]\times [\gamma ^+, y])\]
    which must vanish by $\Gamma$-invariance of the measure.
\end{proof}
\begin{figure}[h]
\label{period well defined}
    \centering
    \includegraphics[width=6cm]{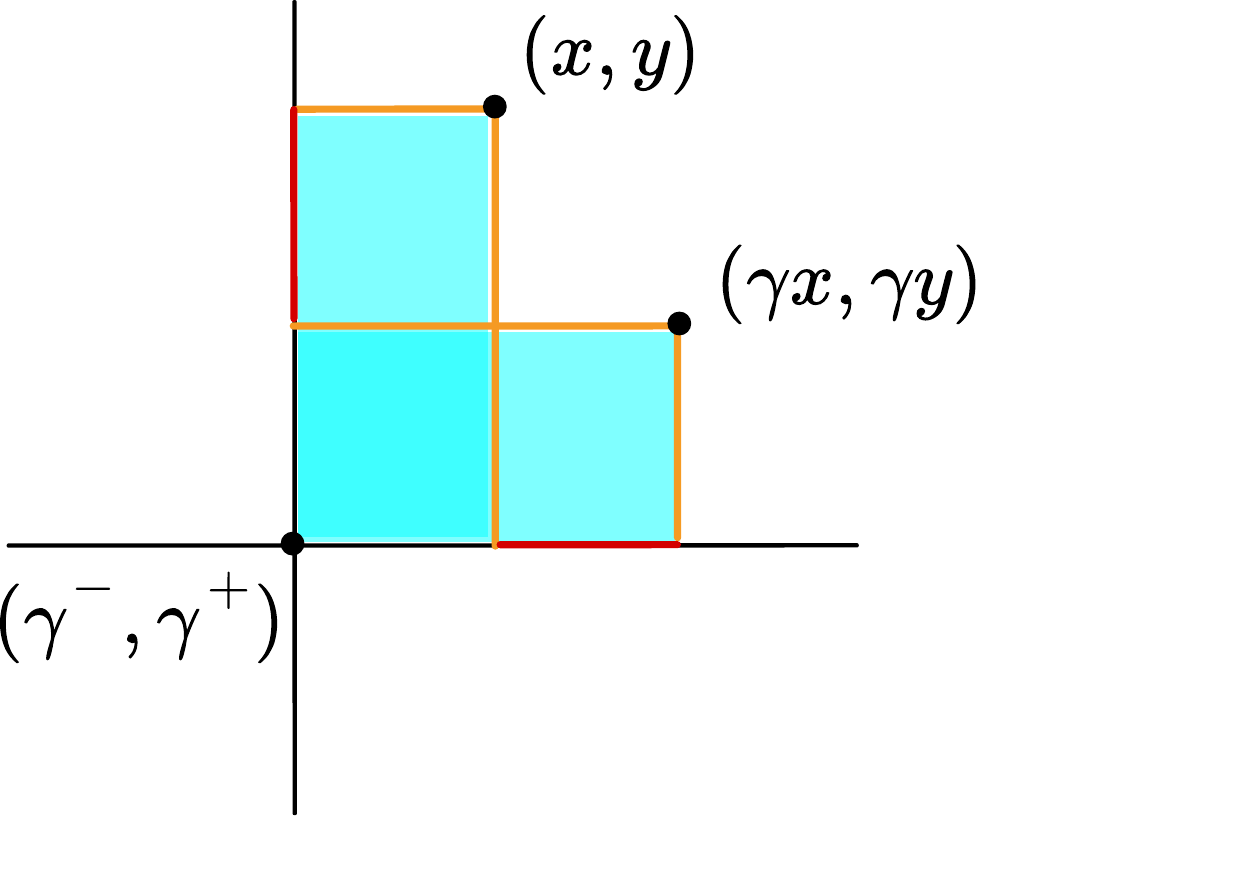}
    \caption{Comparing the two paths connecting $R_N$ to $R_E$}
    
\end{figure}

\subsection{Reduced length spectra and currents}
From an equivariant bundle with taxi connection $P$, one can extract two invariants: the period spectrum $l_P\in A^{[\Gamma]}$, and the curvature $\mu_P\in \mathcal{C}_0(\Gamma,A)$. In this section we derive formulas relating $l_P$ and $\mu_P$. One can change the periods of an equivariant $\R$ bundle without changing its curvature by modifying the $\Gamma$ action by an element of $\Hom(\Gamma,\R)$, motivating the following definition:
\begin{definition}
The reduced period spectrum of $P\in \mathcal{A}(S,A)$ is the image of $l_P$ in $A^{[\Gamma]}/\Hom(\Gamma,A)$. 
\end{definition}
We will use the well known formula \ref{hilbert length box} expressing $l(\alpha) + l(\alpha^{-1})$ as a cross ratio, along with a seemingly new formula \ref{abc lemma} expressing $l_P(\alpha) + l_P(\beta) - l_P(\alpha\beta)$ as a cross ratio, to show that $P\mapsto l_P$ descends to an injection from nullhomologous currents to $A^{[\Gamma]}/\Hom(\Gamma,A)$. It will follow that $P\mapsto l_P$ is also injective, because both vertical arrows in the commutative square
\[
\begin{tikzcd}
    \mathcal{A}(S,A) \arrow{r}\arrow{d}  & A^{[\Gamma]}\arrow{d} \\
    \mathcal{C}_0(\Gamma, A) \arrow{r} & A^{[\Gamma]}/\Hom(\Gamma,A)
\end{tikzcd}
\]
are quotients by $\Hom(\Gamma,A)$.
There is a general strategy for producing formulas relating periods and holonomy which we give now, though we will really only use lemmas \ref{hilbert length box} and \ref{abc lemma}.

\begin{lemma}
\label{relation vector loop}
    Suppose $\gamma_k \cdots \gamma_1 = e$ is a relation in $\Gamma$. Then then there is a taxi loop $r$ in $\mathcal{G}^\circ$ such that 
    \[\sum_{i=1}^k l_P(\gamma_i) = \hol_P(r)\]
    for any equivariant $\R$-bundle with taxi-connection $P$ on $\mathcal{G}^\circ$.
\end{lemma}

\begin{proof}
    Choose $g=(x,y)\in \mathcal{G}^\circ$ with neither $x $ or $y$ fixed by any element of $\Gamma$. 
    Choose a lift $\tilde{g}$ to the $\R$-bundle $P$. 
    Let $g_i = \gamma_i\cdot\gamma_{i-1}\cdots\gamma_1 g$, and $\tilde{g}_i = \gamma_i\cdot\gamma_{i-1}\cdots\gamma_1 \tilde{g}$ for $i=0,...,k$.
    Let $(x_i,y_i) = g_i$.
    For $i=1,...,k$, let $r_i$ be the following concatenation of segments in $\mathcal{G}^\circ$.
    \[(x_{i-1},y_{i-1})\to (x_{i-1}, \gamma_i^+) \to (x_i, \gamma_i^+) \to (x_i, y_i)\]
    We next construct a discontinuous piecewise section $\tilde{r}_i$ of $P$ over $r_i$ with endpoints $\tilde{g}_i$ and $\tilde{g}_{i+1}$. Over $(x_{i-1},y_{i-1})\to (x_{i-1}, \gamma_i^+)$ use the flat section $s$ starting at $\tilde{g}_{i-1}$, extend continuously over $(x_{i-1}, \gamma_i^+) \to (x_i, \gamma_i^+)$, but then use the section $\gamma_i (s^{-1})$ over $(x_i, \gamma_i^+) \to (x_i, y_i)$. The discontinuity of $\tilde{r}_i$ is exactly the period $l_P(\gamma_i)$. 

    Concatenating the sections $\tilde{r}_i$ we obtain a piecewise flat section over the closed loop $r$ in $\mathcal{G}^\circ$ with total discontinuity 
    \[\sum_{i=1}^k l_P(\gamma_i)\]
    which must coincide with the holonomy of the taxi cycle $-r$.
    
\end{proof}

In the proof, we chose $x$, and $y$ not to be fixed points to guarentee that the resulting cycle $r$ lay in $\mathcal{G}^\circ$. In practice, we may choose any starting point we like as long as the resulting cycle lies in $\mathcal{G}^\circ$. Applying the construction of lemma \ref{relation vector loop} to the simplest relation $\gamma ^{-1}\gamma  = e$, recovers the following well known fact.

\begin{figure}[h]
\label{width cross ratio}
    \centering
\includegraphics[width=10cm]{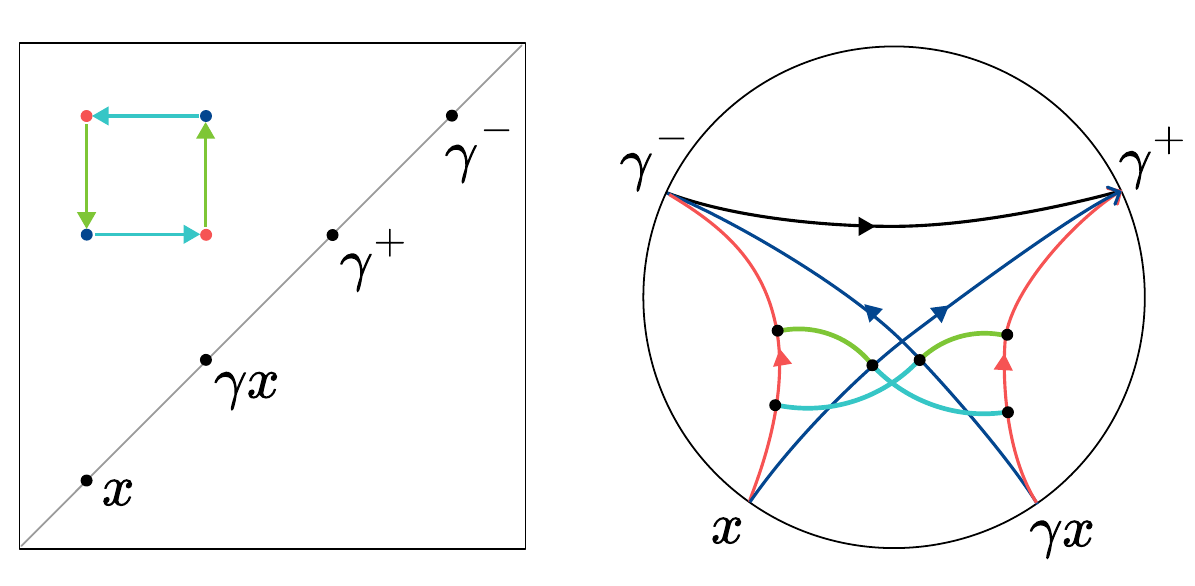}
    \caption{The loop in $P$, drawn in $\mathcal{G}$ and in $\tilde{S}$, which proves Lemma \ref{width cross ratio}}
    
\end{figure}
\begin{lemma}
\label{hilbert length box} $l(\gamma ) + l(\gamma ^{-1}) = h(x,\gamma x;\gamma ^+,\gamma ^-)$ for any nontrivial $\gamma \in \Gamma$, and any $x\in \del\Gamma^\circ$. 
\end{lemma}
\begin{proof}
For the starting point use $g=(x,\gamma ^+)$, and apply the strategy of Lemma \ref{relation vector loop}. Figure \ref{width cross ratio} shows the resulting taxi-path in $\mathcal{G}^\circ$. Identifying $P$ with the unit tangent bundle of $\tilde{S}$ for a hyperbolic metric, we also depict the projection to $\tilde{S}$ of the piecewise flat lift of this path to $P$, which is a union of horocyclic segments. The discontinuities of this lift are seen to be $l(\gamma)$ and $l(\gamma^{-1})$
\end{proof} 

Now we apply the strategy from lemma \ref{relation vector loop} to the relation $ab(ab)^{-1}$ to get a loop whose holonomy is $l(a) + l(b) - l(ab)$, but we deviate slightly from the algorithm so that this loop is just a single cross ratio. 

\begin{lemma}
\label{abc lemma}
    For any $a,b\in \Gamma$ we have the following relation between periods and cross ratios.
    \[l(a) + l(b) - l(ba) = h((ba)^-, b^-; a^+,b\cdot a^+)\]
\end{lemma}
\begin{proof}
    The right hand side is by definition the holonomy around a cycle with four segments. We will check that that each of these segments is in $\mathcal{G}^\circ$ so that this is well defined, but first let us complete the proof assuming this. Let $c=ba$. We subdivide so that the cycle is written with five segments.
    \[(c^-, a^+) \overset{s_1}{\longrightarrow} 
    (ac^-,a^+) \overset{s_2}{\longrightarrow}  
    (b^-, a^+) \overset{s_3}{\longrightarrow}  
    (b^-, ba^+) \overset{s_4}{\longrightarrow}  
    (c^-, ba^+) \overset{s_5}{\longrightarrow}  
    (c^-, a^+) \]
    Choose a point $p\in P$ in the fiber over $(c^-,a^+)$. 
    Let $\tilde{s}_1$ be the flat section of $P$ over $s_1$ starting at $p$. 
    Let $\tilde{s}_2$ be the flat section over $s_2$ starting at $ap$. 
    Let $\tilde{s}_3$ be the flat section over $s_3$ which agrees with $\tilde{s}_2$ at $(b^-,a^+)$. 
    Let $\tilde{s}_4 = b \tilde{s}_2$, and note that its endpoint over $(c^-,ba^+)$ is $cp$. 
    Let $\tilde{s}_5$ be the flat lift of $s_5$ ending at $p$. The lift \[\tilde{s}_1 + \tilde{s}_2 + \tilde{s}_3 + \tilde{s}_4 + \tilde{s}_5\] has jumps by $l(a)$, $l(b)$ and $-l(c)$ at $(ac^-,a^+)$, $(b^-, ba^+)$, and $(c^-,ba^+)$ respectively, so we have
    \[h((ba)^-, b^-;b\cdot a^+, a^+) + l(a) + l(b) - l(c) = 0\]
    which is equivalent to the statement of the lemma.
    
\begin{figure}[h]
    \centering
\includegraphics[width=10cm]{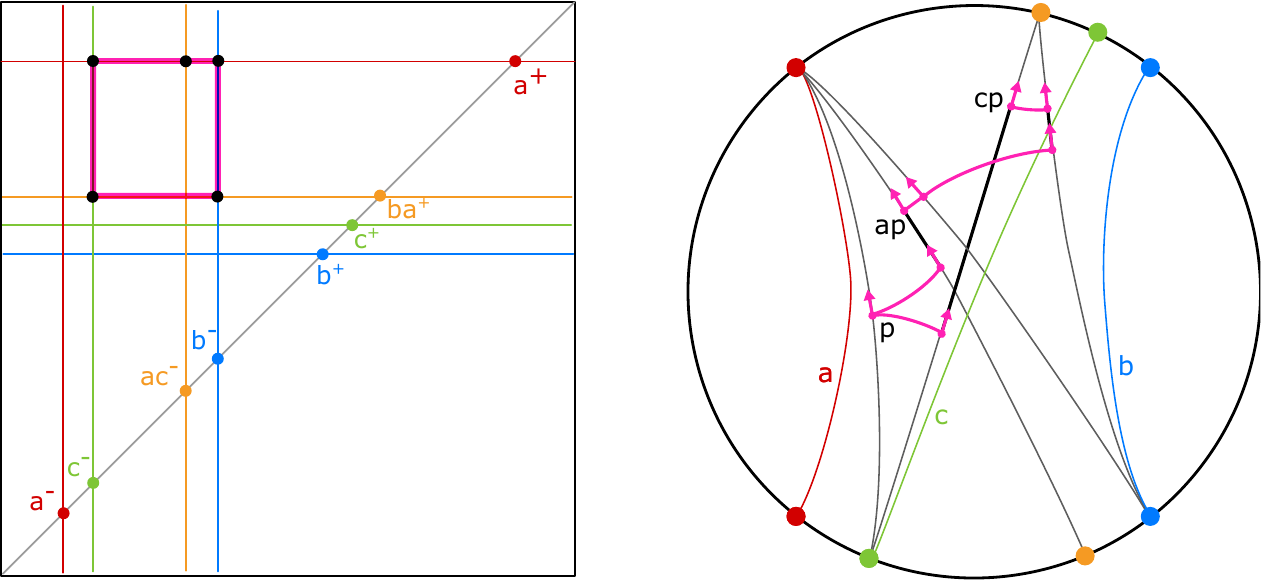}
    \caption{The geodesics and horocyclic segments used in the proof of Lemma \ref{abc lemma}}
    
\end{figure}

Now we return to the issue of showing that the segments invloved in \linebreak $h((ba)^-, b^-; ba^+,a^+)$ are all in $\mathcal{G}^\circ$. There are three topological possibilities for the relative position of the fixed points of $a$ and $b$.

We need to also know where the fixed points of $ba$ and $bab^{-1}$ are on the circle in each of these scenarios. (Note that $ba^+ = (bab^{-1})^+$.) One way to locate these fixed points is to draw the curves on the infinite volume surface $\langle a, b\rangle \backslash\tilde{S}$, which will either be a three-holed sphere, or a one-holed torus.
\begin{figure}[h]
\label{quotient surfaces}
    \centering
    \includegraphics[width=12cm]{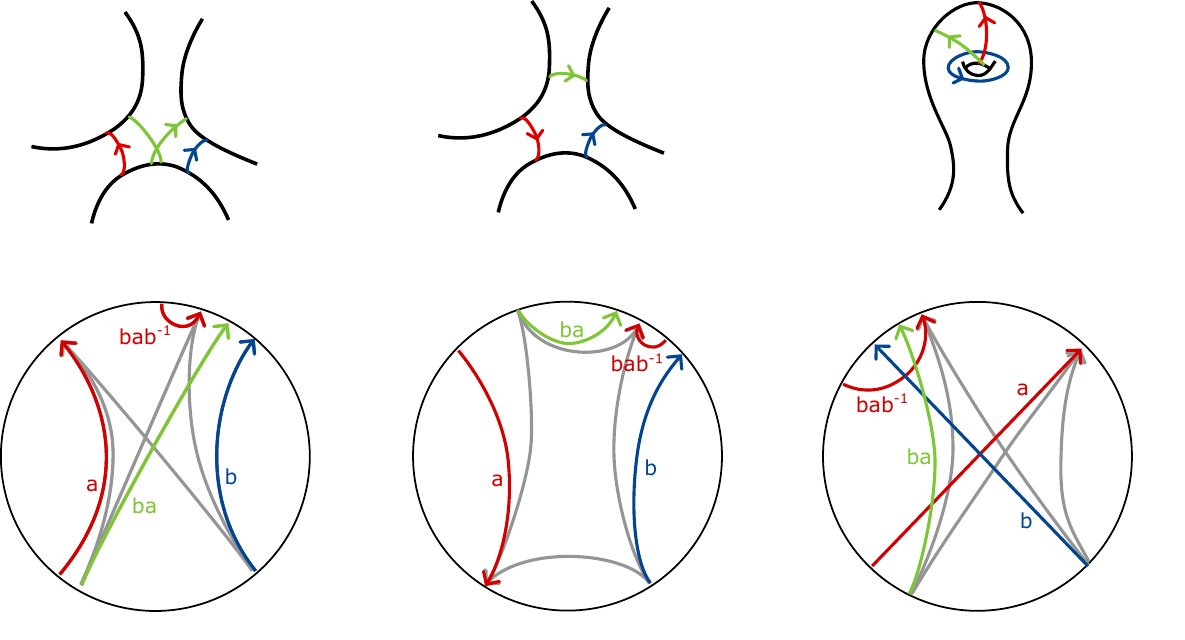}
    \caption{The three possible relative positions for axis of $a$ and $b$}
    
\end{figure}
In Figure \ref{quotient surfaces}, the grey geodesics are the corners of the taxi-path which we hope lies in $\mathcal{G}^\circ$. Since this taxi-path has one coordinate being a fixed point of of $a,b,ba$, or $bab^{-1}$ at all times, the only fixed points it could run into are the four geodesics in the picture. One inspects the picture to make sure this doesn't happen.
\end{proof}

Lemma \ref{abc lemma} lets us express any period as an integral combination of holonomies, and periods of a generating set.
\begin{lemma}
\label{holonomy determines spectrum}
    For any equivariant $A$ bundle with connection on $\mathcal{G}^\circ$, holonomy determines reduced period spectrum.
\end{lemma}
\begin{proof}
    Working with reduced length spectra is the same as choosing a standard generating set $a_1,...,a_{2g}\in \Gamma$ and working with ordinary length spectra which satisfy $l(a_i) = l_i$ for some fixed $l_1,...,l_{2g}\in A$. 

    Suppose for induction that we can express every word in $a_1,...,a_{2g},a_1^{-1},...,a_{2g}^{-1}$ of length $k$ as a sum of $l_i$, and holonomies of taxi cycles in $\mathcal{G}^\circ$. Let $w'$ be a word of length $k+1$. Suppose $w'=a_i w$. By lemma \ref{abc lemma} we have
    \[l(w') = l(a_i) + l(w) + h((a_iw)^-, a_i^-; w^+,a_i\cdot w^+)\]
    Now suppose $w' = a_i^{-1}w$. We can first use lemma $\ref{abc lemma}$
    \[l(w') = l(a_i^{-1}) + l(w) + h((a_i^{-1}w)^-, a_i^+; w^+,a_i^{-1}\cdot w^+)\]
    then choose any $x\in \del\Gamma^\circ$ not fixed by $a_i$ and use lemma \ref{hilbert length box} 
    \[l(w') = -l(a_i) + l(w) + h(x,a_ix;a_i^-,a_i^+) + h((a_i^{-1}w)^-, a_i^+; a_i^{-1}w^+,w^+)\]
\end{proof}

Note that we didn't need positivity of curvature to show that holonomy determines reduced length spectrum. On the other hand, to show that reduced length spectrum determines holonomy, we will make use of positivity. By the dynamics of $\SL_2\!\R$ acting on $\mathbb{RP}^1$ we have
\[\lim_{N\to \infty} (b^Na^N)^- = a^-\]
\[\lim_{N\to \infty} b^N(a^N)^+ = b^+\]
If the points $(a^-,a^+)$ and $(b^-,b^+)$ have zero measure with respect to the curvature of $h$, then $h(x,x';y,y')$ is continuous at $(x,x',y,y')=(a^-,b^-,a^+,a^+)$, so we have
\[\lim_{N\to \infty} l(a^N) + l(b^N) - l(b^Na^N) = h(a^-,b^-;a^+,b^+).\] 
When the taxi-path is the perimeter of a box (i.e. the axes of $\alpha$ and $\beta$ either cross or point the same way), this limit will give the measure of the interior of the box.

Fixed points of group elements are dense in $\mathcal{G}$, so knowing a measure on rectangles of the form $(a^-,b^-)\times (b^+,a^+)$ determines the measure. Curvature is thus determined by reduced length spectrum. 

A positive, equivariant bundle is determined by its curvature up to changing the equivariance, but the period spectrum will clearly fix the equivariance, so we have shown that the period spectrum map 
$\mathcal{A}(S,A)\to A^{[\Gamma]}$ is injective.


\section{Rank $n$ and tropical rank $n$ holonomy functions}
In this section we associate to an $\R^*$ bundle with taxi connection to a Hitchin representation, and we call holonomy functions of such bundles rank $n$ holonomy functions. These are equivalent to Labourie's rank $n$ cross ratios. Next we investigate what kind of holonomy functions one gets as renormalized limits of logarithms of rank $n$ holonomy functions.
These are called tropical rank $n$ holonomy functions, and their curvatures are called tropical rank $n$ currents. 
We demonstrate that tropical rank $n$ currents have no self $n$-intersection. 

\subsection{$\R^*$ bundles with connection from Hitchin representations}
Let $V$ be a real vector space of dimension $n\geq 2$. Let $\rho:\Gamma \to \SL(V)$ be a representation in $\Hit^n(\Gamma)$. In particular, $\rho$ is \textit{projective Anosov} \cite{Labourie_anosov} meaning that we have continuous equivariant limit maps 
\[\xi:\del\Gamma\to \mathbb{P}(V)\]
\[\xi^*:\del\Gamma\to \mathbb{P}(V^*)\]
such that, for each $\gamma\in \Gamma$, $\xi(\gamma^+)$ is the eigenline of top eigenvalue of $\rho(\gamma)$, and $\xi^*(\gamma^+)$ is the eigenline of $\rho^*(\gamma)$ of top eigenvalue. Here, $\rho^*:\Gamma\to \SL(V^*)$ is the dual representation. It is sometimes helpful to think of $\mathbb{P}(V^*)$ as the set of hyperplanes in $\mathbb{P}(V)$. In geometric language, $\xi(\gamma)$ is the attracting fixed point of $\gamma$, and $\xi^*$ is the repelling hyperplane. The limit maps $\xi$, and $\xi^*$ are transverse in the sense that, $\xi^*(x)$ contains $\xi(y)$ if and only if $x=y$. Hitchin representations are precisely the projective Anosov representations such that $V=\xi(x_1) + ... + \xi(x_n)$ for any tuple of distinct points $x_1,...,x_n\in \del\Gamma$ \cite{Guichard08}. This property of $\xi$ is known as hyperconvexity.

Let $\Delta\subset \mathbb{P}(V^*)\times \mathbb{P}(V)$ be the set of pairs consisting of a hyperplane and a point in that hyperplane. The manifold $\mathbb{P}(V^*)\times \mathbb{P}(V) \backslash \Delta$ comes with structure which we can pull back via $\xi^*\times \xi$ to $\mathcal{G}$. Namely, $\mathbb{P}(V^*)\times \mathbb{P}(V) \backslash \Delta$ is the base of the principal $\R^*$ bundle
\[U := \{(\alpha,v)\in V^*\times V : \alpha(v)=1\}\]
which has a natural connection $\nabla$. The $\R^*$ action is $\lambda\cdot(\alpha,v) = (\lambda^{-1}\alpha,\lambda v)$. The connection $\nabla$ can be described by its property that the affine subspaces
\[\{(\alpha,v)\in U : \alpha = \alpha_0\}\;\;\;\alpha_0\in V^*\backslash\{0\}\]
\[\{(\alpha,v)\in U : v=v_0\}\;\;\;v_0\in V\backslash\{0\}\]
are flat sections over the fibers of the projections of $\mathbb{P}(V^*)\times \mathbb{P}(V) \backslash \Delta$ to $\mathbb{P}(V^*)$ and $\mathbb{P}(V)$ respectively. 
The curvature of this connection is a symplectic form for which the projections to $\mathbb{P}(V^*)$ and $\mathbb{P}(V)$ are Lagrangian fibrations. 
If $\gamma$ is a loop in $\mathbb{P}(V^*)\times\mathbb{P}(V) \backslash \Delta$ which visits the 4 points \[(\zeta^*_1,\zeta_1), (\zeta^*_1,\zeta_2), (\zeta^*_2,\zeta_2), (\zeta^*_2,\zeta_1)\in \mathbb{P}(V^*)\times \mathbb{P}(V)\] via four paths in each of which only one coordinate is changing, then the holonomy of $\nabla$ around $\gamma$ is a cross ratio.
\[\textrm{hol}_\nabla(\gamma) = \frac{\langle \zeta^*_1,\zeta_1\rangle\langle \zeta^*_2,\zeta_2\rangle}{\langle \zeta_1^*,\zeta_2\rangle\langle \zeta_2^*,\zeta_1\rangle}\]
Let $P_\rho$ be the pullback of $U$ by $\xi^*\times\xi$:
\[P_\rho := \{(\alpha,v)\in V^*\times V : \alpha(v)=1, [\alpha] \in \im(\xi^*), [v]\in \im(\xi)\}\]
Flat sections of $P_\rho$ over a segment $H$ are defined to be the sections which are contained in a flat section of $U$. 
For each $x\in \del\Gamma$, and $\alpha\in \xi^*(x)\backslash \{0\}$ there is a flat section $\{x\}\times \del\Gamma\backslash\{x\} \to U$ taking $(x,y)$ to $(\alpha, v)$ where $v$ is the element of $\xi(y)$ such that $\alpha(v)=1$. 
Flat sections over horizontal lines $\del\Gamma\backslash\{y\} \times \{y\}$ are similarly indexed by $\xi(y)\backslash \{0\}$. This construction $\rho\mapsto P_\rho$ defines a map $\Hit^n(S)\to \mathcal{A}(S,\R^*)$, which we can compose with to get other structures associated to representations.

\begin{tikzcd}
\Hit^n(S) \arrow[r] & \mathcal{A}(S,\R^*) \arrow[r]\arrow[d,"\log|\cdot|"] & \mathcal{H}(S,\R^*) \arrow[d,"\log|\cdot|"]  \\
 \: & \mathcal{A}(S) \arrow[r] & \mathcal{H}(S) \arrow[r]  & \mathcal{C}(S)
\end{tikzcd}

We refer to the holonomy functions and currents associated with $\rho$ by $H_\rho\in \mathcal{H}(S,\R^*)$, $h_\rho \in \mathcal{H}(S)$, and $\mu_\rho\in \mathcal{C}(S)$.

\begin{definition}
A \textbf{potential} for a holonomy function $H\in \mathcal{H}(S,A)$ is a function $M:\mathcal{G}^\circ \to A$ such that
\[H(x_1,y_1;x_2,y_2) = \frac{M(x_1,y_1)M(x_2,y_2)}{M(x_1,y_2)M(x_2,y_1)}.\]

\end{definition}
To obtain a potential for $H_\rho$, simply choose lifts of the limit maps $\tilde\xi:\del\Gamma\to V$, and $\tilde\xi^*:\del \Gamma\to V^*$, and define $M(x,y) = \langle \tilde{\xi}^*(x),\tilde{\xi}(y)\rangle$. Note that when $n$ is even there won't exist continuous lifts, but it isn't important here that the lifts be continuous. 
In fact, potentials are always of this form:

\begin{lemma}
\label{potentials come from sections}
    If $m$ is a potential for an $A$ bundle with taxi connection on $\mathcal{G}^\circ$, then $m = u - v$ where $v$ is flat along vertical segments and $u$ is flat along horizontal segments. The sections $v,u$ are determined by $m$ up to simultaneously shifting by a constant. 
\end{lemma}
\begin{proof}
    Suppose $m=u-v=u'-v'$ are two expressions of $m$ as a difference of a horizontally flat section and a vertically flat section. Let $C_u = u' - u$ and let $C_v = v' - v$. Clearly $C_u$ is constant along horizontal segments and $C_v$ is constant along vertical segments. Also $C_u = C_v$ by rearranging $u-v=u'-v'$. This implies $C_u = C_v$ is a constant because $\mathcal{G}^\circ$ is taxi-path connected.

    Now for existence. Define a taxi connection $F$ on the trivial bundle $\underline{A}$ on $\mathcal{G}$ as follows. 
    If $s$ is a vertical segment, let $F(s)$ be the set of constant functions. 
    If $s$ is a horizontal segment, let $F(s)$ be functions of the form $m|_s + C$. We see that 
    \[\hol_{F} (x_1,x_2;y_1,y_2) = m(x_1,y_1) + m(x_2,y_2) - m(x_1,y_2) - m(x_2,y_1)\]
    so $m$ is a potential for $(\underline{A},F)$. The bundle $(\underline{A},F)$ comes with a vertically flat section $v = 0$, and a horizontally flat section, namely $u = m$, and the difference $u - m$ is clearly $m$. Any other $(P,H)$ for which $m$ is a potential, will have the same holonomy, thus be isomorphic to $(\underline{A},H_m)$. Choosing an isomorphism gives the desired sections of $P$. 
\end{proof}

If $M:\mathcal{G}^\circ \to \R^*$ is a potential for an $\R^*$ valued holonomy function, it is helpful to define $\bar{M}:\mathcal{G}^\circ\cup \Delta \to \R$ to coincide with $M$ on $\mathcal{G}^\circ$ and be zero on $\Delta$.

\begin{definition}
An $\R^*$ valued holonomy function $H$ is \textbf{rank $n$} if any, thus every, potential $M$ satisfies
\[ \det(\bar{M}(x_i,y_j)) = 0 \]
for all tuples $x_1,...,x_{n+1},y_1,...,y_{n+1}\in \del\Gamma^\circ$, and
\[ \det(\bar{M}(x_i,y_j)) \neq 0 \]
for all tuples $x_1,...,x_{n},y_1,...,y_{n}\in \del\Gamma^\circ$ with $x_i\neq x_j$, and $y_i\neq y_j$ when $i\neq j$. 
\end{definition}
The holonomy $H_\rho$ for $\rho\in \Hit^n(S)$ is clearly rank $n$, as $M(x,y) = \langle \tilde{\xi}^*(x),\tilde{\xi}(y)\rangle$ is a potential. 
Because $\dim(V)=n$,  $(n+1)\times (n+1)$ minors of $M(x,y)$ vanish whereas, by hyperconvexity, $n\times n$ minors do not. 

\begin{remark}
	If $H \in \mathcal{H}(S,\R^*)$ is a rank $n$ holonomy function such that $H(x_1,x_2;y_1,y_2)$ extends to a H\"older function 
$\{(x_1,x_2,y_1,y_2) : x_1\neq y_2, x_2 \neq y_1\} \to \R$, 
 then $H$ is a rank $n$ cross ratio in the sense of \cite{Labourie07}. Labourie showed that $\rho\mapsto H_\rho$ is a bijection from $\Hit^n(S)$ to rank $n$ cross ratios.
\end{remark}

Working with potentials is sometimes convenient, but it is also nice to have a criteria for rank $n$ which doesn't reference a choice of potential. This is the path taken in \cite{Labourie07}.
\begin{lemma}
\label{equivalent rank n condition}
    A holonomy function $H\in \mathcal{H}(S,\R^*)$ is rank $n$ if and only if 
\[ \det(\bar{H}(x_0,x_i;y_0,y_j)) = 0 \]
for all tuples $x_0,...,x_{n+1},y_0,...,y_{n+1}\in \del\Gamma^\circ$ with $x_0\neq y_i$ and $y_0\neq x_i$, and
\[ \det(\bar{H}(x_0,x_i;y_0,y_j)) \neq 0 \]
for all tuples $x_0,...,x_{n},y_0,...,y_{n}\in \del\Gamma^\circ$ with $x_0\neq y_i$ and $y_0\neq x_i$, and with $x_i\neq x_j$, and $y_i\neq y_j$ when $i\neq j$. 
\end{lemma}
\begin{proof}
    Suppose $H$ is an $\R^*$ valued holonomy function and that $M$ is a potential. Fix two distinct points $x_0$ and $y_0$ in $\del\Gamma^\circ$. Every $k\times k$ minor of the function
    \[H(x_0,x;y_0,y) = \frac {M(x_0,y_0)\bar{M}(x,y)}{M(x_0,y)M(x,y_0)}\]
    of $x\in \del\Gamma^\circ\backslash \{y_0\}$ and $y\in \del\Gamma^\circ\backslash \{x_0\}$ will vanish if and only if the corresponding $k\times k$ minor of $\bar{M}(x,y)$ does. 
\end{proof}

\subsection{Tropical rank-$n$ currents}
The notion of tropical rank-$n$ is based on the following lemma, which is just the application of the standard tropicalization of polynomials to the determinant.

\begin{lemma}
\label{tropical determinant}
    If $A^{(k)}$ is a sequence of $n\times n$ matrices with 
    \[\lim_{k\to\infty} \frac {\log|A^{(k)}_{ij}|}{R_k} = a_{ij}\]
    for a sequence of real numbers $R_k\to \infty$, then
    \[\lim_{k\to\infty} \frac{\log|\det(A^{(k)})|}{R_k} = \max_{\sigma\in S_n} \sum_{i=1}^{n} a_{i \sigma(i)}\]
    whenever there is a single permutation $\sigma\in S_n$ attaining the maximum. In particular, if $\det(A_i)=0$ for all $i$, then two permutations must attain the maximum.
\end{lemma}

The right hand side is called the tropical determinant of the matrix $a$. If $m:\mathcal{G}^\circ\to \R$ is a potential for a holonomy function $h\in \mathcal{H}(S,\R)$ then let $\bar{m}:\mathcal{G}^\circ\cup \Delta\to \R\cup \{-\infty\}$ be the extension of $m$ which is $-\infty$ on $\Delta$.
\begin{definition}
    A holonomy function $h$ is \textbf{tropical rank $n$} if for any potential $m(x,y)$ for $h$, there are two distinct permutations realizing the maximum in the tropical determinant
    \[\max_{\sigma\in S_{n+1}} \sum_{i=1}^{n+1} \bar{m}(x_i,y_{\sigma(i)})\]
    for any tuples $x_1,...,x_{n+1},y_1,...,y_{n+1}\in \del\Gamma^\circ$. 
\end{definition}

Just as in lemma \ref{equivalent rank n condition}, there is a more concrete criterion for tropical rank $n$.
\begin{lemma}
    A holonomy function $h\in \mathcal{H}(S,\R)$ is tropical rank $n$ if and only if there are two distinct permutations realizing the maximum in the tropical determinant
    \[\max_{\sigma\in S_n} \sum_{i=1}^{n+1} h(x_0,x_i;y_0,y_{\sigma(i)})\]
    for all tuples $x_0,...,x_{n+1},y_0,...,y_{n+1}\in \del\Gamma^\circ$ with $x_0\neq y_i$ and $y_0\neq x_i$.
\end{lemma}
The proof is the same as for Lemma \ref{equivalent rank n condition}.

\begin{lemma}
\label{equivalent tropical rank n condition}
    If $H_k$ are rank $n$ holonomy functions and
    \[\lim_{k\to \infty} \frac{\log |H_k| }{R_k} = h\]
    for a sequence of real numbers $R_k\to \infty$, then $h$ must be tropical rank $n$. 
\end{lemma}
\begin{proof}
    This follows directly from Lemmas \ref{tropical determinant}, \ref{equivalent rank n condition}, and \ref{equivalent tropical rank n condition}.
\end{proof}
We do not know if every tropical rank $n$ holonomy function arises in the boundary of $\Hit^n(S)$. 
\begin{definition}
\label{tropical rank n current}
    A geodesic current $\mu\in \mathcal{C}(S)$ is tropical rank $n$ if it is the curvature of an equivariant bundle with taxi connection $P\in \mathcal{A}(S)$ which has tropical rank $n$ holonomy. 
\end{definition}
Just as the boundary of Teichmuller space consists of currents with no self intersection, boundary points of $\Hit^n(S)$ have no ``$n$-intersection". 
\begin{figure}[h]
\label{fig:no n intersection}
    \centering
\includegraphics[width=10cm]{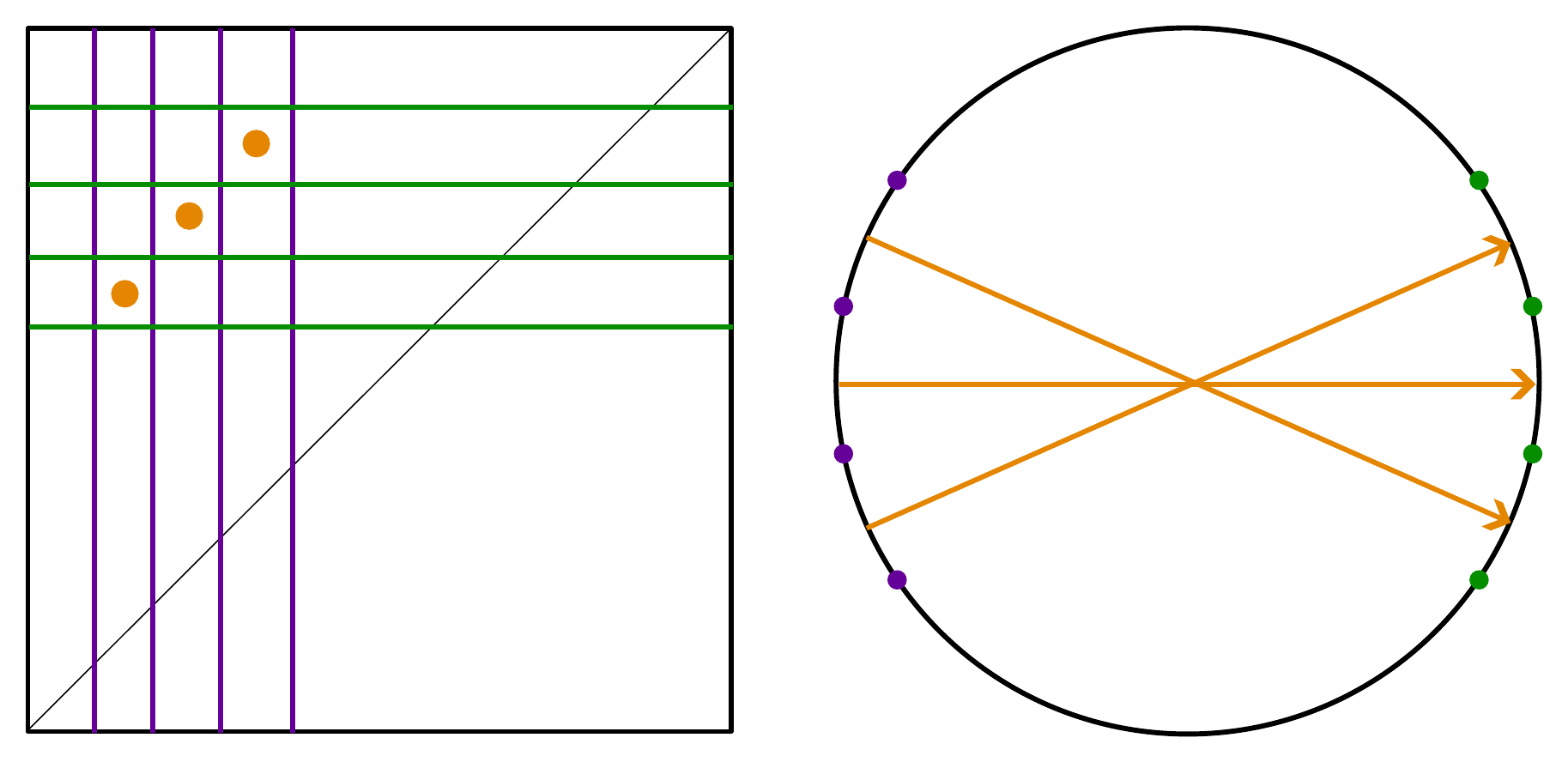}
    \caption{3 geodesics which 3-intersect, and 8 points of $\del\Gamma$ which rule out their simultanious presence in a tropical rank $3$ current}
    
\end{figure}

\begin{lemma}
\label{no n-intersection}
    If $\mu$ is a tropical rank $n$ current, then for any $x_1< ... < x_{n+1} < y_1 < ... < y_{n+1} \in \del\Gamma^\circ$, there must be some $i\in 1,...,n$ such that $\mu([x_i,x_{i+1}]\times [y_i,y_{i+1}]) = 0$.
\end{lemma}
\begin{proof}
Let $h$ denote the holonomy function corresponding to $\mu$ and let $m$ be any potential for $h$. If $\sigma,\sigma'\in S^n$ are two permutations, then the difference
\[\sum_{i=1}^{n+1} m(x_i,y_{\sigma(i)}) - m(x_i,y_{\sigma'(i)})\]
will always the holonomy of a cycle. 
In the special case when the two permutations differ by a transposition, $\sigma' = \sigma(ij)$, the difference of the sums is
\[m(x_i,y_{\sigma(i)}) + m(x_j,y_{\sigma(j)}) - m(x_i,y_{\sigma(j)}) - m(x_j,y_{\sigma(i)}) = 
h(x_i ,x_j,y_{\sigma(i)},y_{\sigma(j)})\]
and in the case when $i<j$ and $\sigma(i)<\sigma(j)$ we recognize this as the measure of the box $\mu([x_i,x_j]\times[y_{\sigma(i)},y_{\sigma(j)}])$. Thus, whenever $\sigma' = \sigma(ij)$, with $i<j$ and $\sigma(i) <\sigma(j)$, 
\[\sum_{i=1}^{n+1} m(x_i,y_{\sigma(i)}) \geq m(x_i,y_{\sigma'(i)}).\]
In other words, the map
\[\sigma \mapsto \sum_{i=1}^{n+1} m(x_i,y_{\sigma(i)})\]
is weakly increasing for the reverse Bruhat order on the symmetric group. Let $X(\sigma)$ denote the number of crossings of $\sigma$, i.e. the number of pairs $i<j$ with $\sigma(i)<\sigma(j)$. The reverse Bruhat order is defined by the property that $\sigma$ covers $\sigma'$ iff $\sigma'=\sigma(ij)$, and $X(\sigma) = X(\sigma')+1$. Recall that an element of a poset $a$ covers another element $b$ if $a>b$ and there is no $c$ such that $a>c>b$. There is a unique maximal element for the reverse Bruhat order, namely the identity permutation.

Since $m$ is tropical rank $n$, we know that at least two permutations must maximize $\sum_{i=1}^{n+1} m(x_i,y_{\sigma(i)})$. One of those must be the identity permutation, and another must be covered by identity, i.e. must be a transposition of the form $(i(i+1))$ for some $i\in \{1,...,n\}$. The difference, $\mu([x_i,x_{i+1}]\times [y_i,y_{i+1}])$, must then be zero.
\end{proof}

\section{From a current to a metric space}
In this section, for any nullhomologous geodesic current $\mu$, we construct a metric space $X_\mu$ with $\Gamma$ action. In the case when $\mu$ is the curvature of an equivariant bundle $P$ with period function $l$, the translation length of $a\in \Gamma$ acting on $X_\mu$ is $l(a) + l(a^{-1})$. The non-symmetrized periods are encoded in a more abstract structure on $X_\mu$ which we call a relative metric. If $\mu$ is a continuous measure, then $X_\mu$ is infinite dimensional, but when $\mu$ is tropical rank $n$, $X_\mu$ is at most $n-1$ dimensional. 

\subsection{Holonomy zero lower submeasures}
A ``lower submeasure" is similar to an \textbf{order ideal}: a subset $I$ of a poset $A$ such that if $a$ in $I$, every element less than $a$ is in $I$. 
The relevant poset for us is $\mathcal{G}$ with $(x',y') < (x,y)$ if $(x<x'<y'<y)$. Order ideals of $\mathcal{G}$ arise in a natural way. Choosing a hyperbolic metric on $S$, $\mathcal{G}$ becomes identified with the set of geodesic half-spaces in $\tilde{S}$ which are naturally ordered by inclusion. For each $x\in \tilde{S}$, the set of half spaces not containing $x$ is an order ideal of $\mathcal{G}$. 

\begin{definition}
If $\mu$ is a measure on $\mathcal{G}$, a \textbf{lower submeasure} of $\mu$ is a measure $\nu$, with $\nu(U)<\mu(U)$ for all measurable sets $U\subset \mathcal{G}$, such that if $(x,y)\in \supp(\nu)$ then $\nu$ and $\mu$ coincide on $\mathcal{G}_{<(x,y)}$. 
\end{definition}
An order ideal $I\subset \mathcal{G}$ gives a lower submeasure $\nu(U) := \mu(U\cap I)$, though different order ideals can give the same submeasure and not all lower submeasures come from order ideals. 

If $l\subset \mathcal{G}$ is a monotonic path in $\mathcal{G}^\circ$ which wraps around once, then define $\nu_l$ to be the measure which coincides with $\mu$ on or below $l$ and is zero above $l$. Note that for any two loops $l$ and $l'$, the difference $\nu_l - \nu_{l'}$ will be a finite signed measure because $\mu$ is locally finite.

If $\nu$ is a lower submeasure, then $\bar{\nu} := \mu - \nu$ is an upper submeasure. Instead of using lower submeasures, it is more natural to use ``monotone partitions" of $\mu$: 
\begin{definition} A monotone partition of a geodesic current $\mu$ is a pair of measures $(\nu,\bar{\nu})$ on $\mathcal{G}$ such that $\mu = \nu + \bar{\nu}$ and every point of $\supp(\nu)$ is less than or equal to every point of $\supp(\bar{\nu})$. 
\end{definition}
It is not hard to check that $(\nu,\bar{\nu})$ is a monotone partition if and only if $\nu$ is a lower submeasure. We will switch back and forth between these two notions.

\begin{definition}
    A lower submeasure $\nu$ is \textbf{admissible} if $\nu - \nu_l$ is a finite measure for a monotonic loop $l\subset \mathcal{G}^\circ$. 
\end{definition}

\begin{definition}
Let $\mu\in \mathcal{C}_0(S)$ be a nullhomologous geodesic current, and let $h\in \mathcal{H}(S)$ be the unique holonomy function which can be lifted to $\mathcal{A}(S)$. Holonomy of admissible lower submeasures of $\mu$ is defined by the following two conditions.
	\begin{itemize}
		\item If $l$ is a monotonic taxi loop in $\mathcal{G}^\circ$ wrapping once around, then $h(\nu_l) = h(l)$.
		\item If $\nu' = \nu + \epsilon$ where $\epsilon$ is a finite positive measure, then $h(\nu')=h(\nu)+|\epsilon|$.
	\end{itemize}
\end{definition}
Now we can give the main definition of this section: a geometric incarnation of geodesic currents.
\begin{definition}
\label{X definition}
If $\mu$ is a nullhomologous geodesic current, its \textbf{universal asymmetric dual space} $X_\mu$ is the set of holonomy zero lower submeasures of $\mu$. 
\end{definition}
The weak topology makes $X_\mu$ into a topological space, though soon we will endow it with an explicit metric. As mentioned before, a lower submeasure gives a monotone partition $\mu = \nu + \bar{\nu}$. We call $\supp(\nu)\cap \supp(\bar{\nu})$ the set of shared points.

\begin{lemma}
If the support of $\mu$ is discrete, then the set of lower submeasures with finitely many shared points is a cube complex. 
\end{lemma}
\begin{proof}
    By evaluating a submeasure on each support point of $\mu$, the set of submeasures is identified with a cube in $\R^{\supp(\mu)}$. For every partition $\supp(\mu) = L\sqcup F \sqcup U$ such that $F$ is finite, let $C(L,F,U)$ denote the set of lower submeasures $\nu < \mu$ such that $\nu(p) = \mu(p)$ for $p\in L$, and $\nu(p) = 0$ for $p$ in $U$. The set $C(L,F,U)$ is either empty, or a closed finite dimensional face of the cube of submeasures. A lower submeasure $\nu$ gives a partition $\supp(\mu) = L\sqcup F \sqcup U$, such that $C(L,F,U)$ is the smallest face containing $\nu$. Every point in $C(L,F,U)$ is also a lower submeasure. The set of lower submeasures with finitely many shared points is thus a union of closed finite dimensional faces of a cube in $\R^{\supp(\mu)}$.
\end{proof}

\begin{lemma} If $\mu$ is the curvature of a holonomy function $h$, and $\mu = \nu + \bar{\nu}$ is a monotone, holomony zero partition, then the function 
\[m_\nu(x,y) := - \nu(\mathcal{G}_{\geq(x,y)}) - \bar{\nu}(\mathcal{G}_{\leq(x,y)})\]
is a potential for $h$. 
\end{lemma}
\begin{proof}
    Admissibility implies that $m_\nu(x,y)$ is well defined. 
    
    Let $x_1,x_2,y_1,y_2\in \del \Gamma$, be four points of $\del \Gamma$ so that the taxi-loop $l$ given by
    \[(x_1,y_1)\to(x_1,y_2)\to (x_2,y_2)\to (x_2,y_1)\to (x_1,y_1)\] is in $\mathcal{G}^\circ$.
    Suppose $x_1<x_2<y_1<y_2$ so that $l$ is the boundary of the rectangle $r = [x_1,x_2]\times [y_1,y_2]$. For a subset $U\subset \mathcal{G}$, let $\mathbf{1}[U]$ denote the indicator function on $U$. There are inclusion-exclusion type identities of indicator functions.
    \[- \mathbf{1}[\mathcal{G}_{\geq (x_1,y_1)}] - \mathbf{1}[\mathcal{G}_{\geq (x_2,y_2)}] + \mathbf{1}[\mathcal{G}_{\geq (x_1,y_2)}] + \mathbf{1}[\mathcal{G}_{\geq (x_2,y_1)}] = \mathbf{1}[(x_1,x_2]\times [y_1,y_2)]\]
    \[- \mathbf{1}[\mathcal{G}_{\leq (x_1,y_1)}] - \mathbf{1}[\mathcal{G}_{\leq (x_2,y_2)}] + \mathbf{1}[\mathcal{G}_{\leq (x_1,y_2)}] + \mathbf{1}[\mathcal{G}_{\leq (x_2,y_1)}] = \mathbf{1}[[x_1,x_2)\times (y_1,y_2]]\]
    These identities imply $m_\nu$ correctly computes the area of $r$.
    \[m_\nu(x_1,y_1) + m_\nu(x_2,y_2) - m_\nu(x_1,y_2) - m_\nu(x_2,y_1) = \nu(r) + \bar{\nu}(r) = \mu(r).\]
    On the other hand, suppose $x_1 < y_1 < x_2 < y_2$. Now $l$ is a taxi loop which winds once, positively, around $\mathcal{G}$. We have
    \[- \mathbf{1}[\mathcal{G}_{\geq (x_1,y_1)}] - \mathbf{1}[\mathcal{G}_{\geq (x_2,y_2)}] + \mathbf{1}[\mathcal{G}_{\geq (x_1,y_2)}] + \mathbf{1}[\mathcal{G}_{\geq (x_2,y_1)}] = -\mathbf{1}[\mathcal{G}_{\geq l}]\]
    \[- \mathbf{1}[\mathcal{G}_{\leq (x_1,y_1)}] - \mathbf{1}[\mathcal{G}_{\leq (x_2,y_2)}] + \mathbf{1}[\mathcal{G}_{\leq (x_1,y_2)}] + \mathbf{1}[\mathcal{G}_{\leq (x_2,y_1)}] = \mathbf{1}[\mathcal{G}_{\leq l}]\] 
    Now we have
    \[m_\nu(x_1,y_1) + m_\nu(x_2,y_2) - m_\nu(x_1,y_2) - m_\nu(x_2,y_1) = - \nu(\mathbf{1}[\mathcal{G}_{\geq l}]) + \bar{\nu}(-\mathbf{1}[\mathcal{G}_{\leq l}]).\]
    Since $\nu$ was assumed to be holonomy zero, this is preciesly the holonomy of $l$.
\end{proof}
This next lemma is in a way the main point of this paper.
\begin{lemma}
\label{Xmu is dimension n-1}
Suppose $\mu$ is tropical rank $n$, then holonomy zero lower submeasures can have at most $n$ shared points. Consequently, if $\mu$ is also discrete, $X_\mu$ is a polyhedral complex of dimension at most $n-1$. 
\end{lemma}
\begin{proof}
    Let $\nu$ be a holonomy zero lower submeasure of $\mu$. 
    Recall that $m_\nu$ is zero precisely on the points which are neither below $\nu$ nor above $\bar{\nu}$. Suppose $(u_1,v_1), ... ,(u_n,v_n)\in \mathcal{G}$ are the shared points of $\nu$. 
    Choose points $(x_1,y_1), \dots ,$ $(x_n,y_n)$ in the zero set of $m$ with $u_i < x_i < u_{i+1}$ and $v_i < y_i < v_{i+1}$. 
    This implies that if $i\neq j$, $(x_i,y_j)$ is above or below some shared point, thus $m(x_i,y_j) < 0$. This means that the term $\sum m(x_i,y_i)$ uniquely maximizes the tropical determinant, thus $\mu$ cannot be tropical rank $n-1$ or less.
\end{proof}
\begin{figure}[h]
\label{fig:n shared points}
    \centering
    \includegraphics[width=9cm]{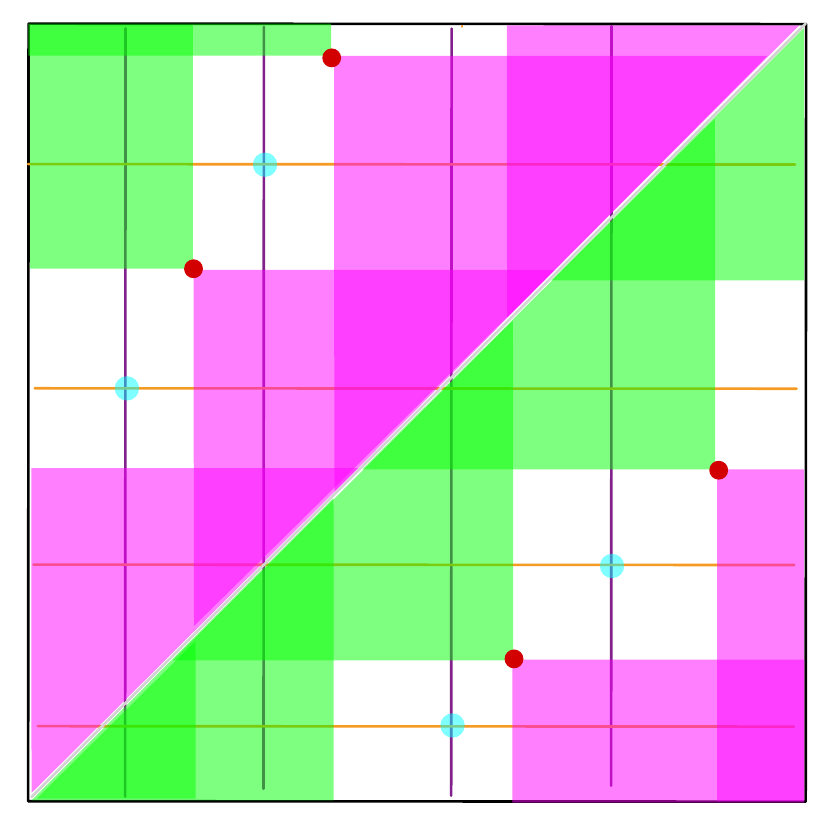}
    \caption{A partition with 4 shared points}
    
\end{figure}

\subsection{The metric}
Now we put a metric on $X_\mu$.
Let $\nu_1,\nu_2\in X_\mu$. The difference $\nu_2 - \nu_1$ is a signed measure of total measure zero. The metric on $X_\mu$ which is simplest to define is $d'(\nu_1,\nu_2) = |(\nu_2 - \nu_1)^+|$, the total measure of the positive part of the difference. Instead we opt for a different metric, the advantages of which will start to emerge in the next subsection. 
We can push forward $\nu_2 - \nu_1$ to $\del \Gamma$ and integrate to get a function on $\del \Gamma$.
\[f_{\nu_1,\nu_2}(x) := \int_{(x_0,x]} (\pi_1)_* (\nu_2 - \nu_1)\]
Here $x_0\in \del \Gamma$ is a basepoint, and changing the basepoint only changes $f_{\nu_1,\nu_2}$ by a constant.
\begin{definition}
\[d(\nu_1,\nu_2) = \sup(f_{\nu_1,\nu_2}) - \inf(f_{\nu_1,\nu_2})\]
\end{definition}

\begin{remark}
This metric is inspired by a metric in symplectic geometry called Lagrangian Hofer distance. If $L\subset M$ is a connected Lagrangian in a symplectic manifold, the space of infinitesimal Hamiltonian deformations of $L$ is $C^\infty(L)/\R$. The norm $\sup(f) - \inf(f)$ defines a Finsler metric on the Hamiltonian isotopy class of $L$. In our context, if $\mu$ is an absolutely continuous geodesic current, we can think of it as a symplectic form, and holonomy-zero, monotone partitions correspond to monotonic loops in $\mathcal{G}$ which we can think of as Lagrangians. See the appendix for more explanation.
\end{remark}


\begin{lemma}
    $(X_\mu,d)$ is a metric space.
\end{lemma}
\begin{proof}
Symmetry and triangle inequality are straight forward to check. Non-degeneracy requires more effort. Suppose $d(\nu_1,\nu_2) = 0$. This means $f_{\nu_2,\nu_1}$ is constant, therefore $(\pi_1)_*(\nu_2 - \nu_1) = 0$. Let $\epsilon = \nu_2-\nu_1$, and let $\epsilon = \epsilon_+ - \epsilon_-$ be the decomposition into positive measures coming from the Hahn decomposition theorem. Since $\nu_2 - \nu_1 = \bar{\nu}_1 - \bar{\nu}_2$, we have
\[\supp(\epsilon_+) \subset \supp(\nu_2) \cap \supp(\bar{\nu}_1)\]
\[\supp(\epsilon_-) \subset \supp(\nu_1) \cap \supp(\bar{\nu}_2)\]
Since $\mu = \nu_1 + \bar{\nu}_1$, and $\mu = \nu_2 + \bar{\nu}_2$ are monotone partitions, we see
\[p \in \supp(\epsilon_+) \;\;\; \text{and} \;\;\; q\in \supp(\epsilon_-) \;\;\; \Rightarrow \;\;\; q \nless p \;\;\; \text{and} \;\;\; p \nless q \]
Suppose $(x,y)\in \supp(\epsilon_+)$. Since $\pi_*\epsilon_- = \pi_*\epsilon_+$, and both $\epsilon_+$ and $\epsilon_-$ have support in some compact $C\in \mathcal{G}$, there must be some $y'$ with $(x,y')\in \supp(\epsilon_-)$. By the observation above, $y=y'$, and actually $(x,y)$ is the only support point of $\epsilon_-$ and $\epsilon_+$. We conclude that $\supp(\epsilon_-) = \supp(\epsilon_+)$, and this subset projects homeomorphically to its image in $\del\Gamma$. Then, $\pi_*\epsilon_- = \pi_*\epsilon_+$ implies $\epsilon_- = \epsilon_+$ and thus $\nu_1 = \nu_2$.

\end{proof}

We now recall the notion of translation length. Suppose $\phi$ is an isometry of a metric space $X$. The translation length of $\phi$ is 
\[\lim_{n\to \infty} \frac{d(x,\phi^n(x))}{n}\]
for any choice of $x\in X$. If $y\in X$ is another point, we have 
\[\lim_{n\to \infty} \frac{d(y,\phi^n(y))}{n} \leq \lim_{n\to \infty} \frac{2 d(y,x) + d(x,\phi^n(x))}{n} = \lim_{n\to \infty} \frac{d(x,\phi^n(x))}{n}\]
so the definition doesn't depend on the choice of point.
If $x\in X$ satisfies $d(x,\phi^n(x)) = n d(x,\phi(x))$ for all $n\in \mathbb{N}$, then the translation length is simply $d(x,\phi(x))$. 

\begin{lemma}
\label{translation length and area}
If $\mu$ is a geodesic current, and $\gamma\in \Gamma$ then a submeasure $\nu\in X_\mu$ such that $\nu = \mu$ on both $\mathcal{G}_{<(\gamma^+,\gamma^-)}$ and $\mathcal{G}_{<(\gamma^-,\gamma^+)}$ will satisfy $d(\nu,\gamma^n\nu) = n d(\nu,\gamma \nu)$ thus the translation length of $\gamma$ is $d(\nu,\gamma \nu)$. 
\end{lemma}
\begin{proof}
    The difference $\gamma^n\nu - \nu$ will be zero on the set $C\subset \mathcal{G}$ of geodesics not intersecting $(\gamma^-,\gamma^+)$.
    \[C = \mathcal{G}_{\leq (\gamma^-,\gamma^+)} \cup \mathcal{G}_{\geq (\gamma^-,\gamma^+)} \cup \mathcal{G}_{\leq (\gamma^+,\gamma^-)} \cup \mathcal{G}_{\geq (\gamma^+,\gamma^-)}\]
    Both $\gamma^n\nu$ and $\nu$ are zero above the fixed points $(\gamma^-,\gamma^+)$ and $(\gamma^+,\gamma^-)$, agree with $\mu$ below the fixed points, and must agree with eachother at the fixed points. The set of geodesics intersecting $(\gamma^-,\gamma^+)$ has two components
    \[D_+ = \{(x,y) : x < \gamma^- < y < \gamma^+\}\]
    \[D_- = \{(x,y) : y < \gamma^- < x < \gamma^+\}\]
    and $\gamma$ translates $D_+$ upward while it translates $D_-$ downward. This means that $\gamma^n\nu - \nu$ is always positive on $D_+$ while it is always negative on $D_-$. This means that the distance from $\nu$ to $\gamma^n \nu$ is just the integral of $\gamma^n \nu - \nu$ over $D^+$. 
    
    \[d(\nu,\gamma^n\nu) = \int_{D^+} \gamma^n\nu - \nu = \sum_{k=1}^n \int_{D^+} \gamma^k\nu - \gamma^{k-1}\nu = n d(\nu,\gamma \nu)\]

\end{proof}
The distance $d(\nu,\gamma \nu)$ is the integral of $\gamma \nu - \nu$ over $D^+$ which is equal to the integral of $\nu$ over the box $[x,\gamma x] \times [\gamma^-,\gamma^+]$ for $x\in \del\Gamma^\circ$ with $\gamma^+ < x <\gamma^-$. Together with Lemma \ref{hilbert length box} this shows that the geometry of $(X_\mu,d)$ captures symmetrized periods.

\begin{lemma} If $\mu$ is the curvature of a bundle $P\in\mathcal{A}(S)$, the translation length of $\gamma$ acting on $X_\mu$ is $l_P(\gamma) + l_P(\gamma^{-1})$. 
\end{lemma}

\subsection{Universal symmetric dual space}
If $\mu$ is symmetric, than we can define another smaller universal dual space. Let $\tau:\mathcal{G}\to \mathcal{G}$ be the involution $(x,y)\mapsto (y,x)$. 
\begin{definition}
Let $\mu\in \mathcal{C}(S)$, and suppose $\tau_*\mu = \mu$. A monotone partition $\mu = \nu + \bar{\nu}$ is \textbf{symmetric} if $\bar\nu = \tau_*\nu$. We also call $\nu$ a symmetric lower submeasure.
\end{definition}

\begin{lemma}
\label{symmetric implies hol zero}
Symmetric lower submeasures are holonomy-zero.
\end{lemma}
\begin{proof}
    Since $\mu$ is symmetric, the corresponding holonomy function $h$ is symmetric, meaning $h(z) := -h(\tau_* z)$ for any cycle $z$ in $\mathcal{G}^\circ$. This means that $h(\tau_* \bar{\nu}) = - h(\nu)$ for all admissible lower submeasures $\nu$. If $\nu = \tau_*\bar{\nu}$, then $\nu$ must be holonomy-zero.
\end{proof}

The space $X_\mu^{sym}$ of admissible symmetric lower submeasures is thus a subset of $X_\mu$. We call it the universal symmetric dual space of $\mu$. When $\mu$ is discrete, $X_\mu^{sym}$ is a cube complex. If a symmetric lower submeasure $\nu$ has $n$ shared points, than the dimension of the cube in $X_\mu^{sym}$ containing $\nu$ is $n/2$, while the dimension of the polyhedron in $X_\mu$ containing $\nu$ is $n-1$.

\subsection{A relative metric which knows periods}
We would really like to have an asymmetric metric space with $\Gamma$ action whose translation lengths are periods of a given equivariant $\R$ bundle $P\in \mathcal{A}(S)$. Changing the $\Gamma$ action by a homomorphism $\Gamma\to \R$ changes the periods, but doesn't change the curvature, so this metric space certainly must involve $P$, not just it's curvature $\mu$. In this section we find something slightly different: a $\Gamma$-space living over $X_\mu$ with an asymmetric two argument function that is not quite a metric, but which nonetheless has translation lengths which are periods of $P$.

\begin{definition}
\label{definition relative metric}
    A \textbf{relative metric} on a principal $\R$ bundle $L$ over a set $X$ is a function $d:L\times L\to \R$ which satisfies the following three properties.
    \begin{enumerate}
        \item Homogeneity: $d(x,y+r) = r+d(x,y) = d(x-r,y)$ for all $x,y\in L$ and $r\in \R$,
        \item Triangle inequality: $d(x,y) + d(y,z) \geq d(x,z)$ for all $x,y,z\in L$, and
        \item Non-degeneracy: $d(x,y) + d(y,x) = 0$ if and only if $x$ and $y$ are in the same fiber of $L$.
    \end{enumerate}
\end{definition}

If $d$ is a relative metric on a principal $\R$ bundle $L\to X$, then the symmetrization $d(x,y)+d(y,x)$ descends to a metric on $X$. If we choose a section $s:X\to L$ such that $d(s(x),s(y)) > 0$ for all $x\neq y$, then $d(s(x),s(y))$ is an asymmetric metric on $X$. Two different sections will give metrics which differ by a function of the form $f(y) - f(x)$. 

We now construct a relative metric space over $X_\mu$. The rough idea is that in defining the symmetric metric on $X_\mu$ we had to take $\sup(f_{\nu_1,\nu_2}) - \inf(f_{\nu_1,\nu_2})$ because $f_{\nu_1,\nu_2}$ is really only naturally defined up to adding constants, so we will find a way to fix this constant, then use only the supremum. Let $P$ be an equivariant principal bundle on $\mathcal{G}^\circ$ with curvature $\mu$. Recall that Lemma \ref{potentials come from sections} says that for every $\nu\in X_\mu$, $m_\nu = u - v$ where $v$ and $u$ are sections of $P$ such that $v$ is flat along vertical segments and $u$ is flat along horizontal segments. Furthermore, if $m_\nu = u' - v'$ is another such decomposition, then there is $C\in \R$ such that $v' = v + C$ and $u' = u + C$. 
\begin{definition}
    Let $L_P$ denote the principal $\R$-bundle on $X_\mu$ consisting of triples $(\nu,v,u)$ where $\nu\in X_\mu$, $v$ is a vertically flat section of $P$, $u$ is a horizontally flat section of $P$, and $m_\nu = u - v$.
\end{definition}
If $\mathcal{L} = (\nu,v,u)$ and $\mathcal{L}' = (\nu',v',u')$ are two points in $L_P$, let 
\[f_{\mathcal{L},\mathcal{L'}}(x) = \lim_{y\to x^+} v'(x,y) - v(x,y).\]
The function $v'(x,y) - v(x,y)$ is constant along vertical segments, though since it is only defined on $\mathcal{G}^\circ$, this doesn't quite imply that it is independant of $y$. The limit deals with this technicality. Really, $v'(x,y) - v(x,y)$ will be constant in $y$ for $(x,y)$ not above the supports of $\bar{\nu}$ and $\bar{\nu}'$. 
Now we define the relative metric.
\[d(\mathcal{L},\mathcal{L}') := \sup_{x\in \del\Gamma} f_{\mathcal{L},\mathcal{L'}}(x)\]
The triangle inequality follows from
\[f_{\mathcal{L},\mathcal{L}''} = f_{\mathcal{L},\mathcal{L}'} + f_{\mathcal{L}',\mathcal{L}''}\]
because supremums are subadditive. 
This relative metric symmetrizes to the ordinary metric:

\begin{lemma}
    If $\mathcal{L} = (\nu, v, u)$ and $\mathcal{L}' = (\nu', v', u')$ are two points in $L_P$, then
    \[f_{\mathcal{L},\mathcal{L}'} = f_{\nu,\nu'} + C\]
    where $C$ is some constant, consequently
    \[d(\mathcal{L},\mathcal{L}')
    + d(\mathcal{L}',\mathcal{L}) = d(\nu,\nu')\]
\end{lemma}
\begin{proof}
    Let $x_1 < x_2 < y \in \del \Gamma$ be points such that $(x_1,y)$ and $(x_2,y)$ are both in $\mathcal{G}^\circ$, and neither are above any support points of $\bar{\nu}$ or $\bar{\nu}'$. It will suffice to show
    \[f_{\nu,\nu'}(x_2) - f_{\nu,\nu'}(x_1) = f_{\mathcal{L},\mathcal{L}'}(x_2) - f_{\mathcal{L},\mathcal{L}'}(x_1).\]
    The left hand side is equal to $(\nu' - \nu)(\pi_1^{-1}((x_1,x_2]))$. This is the same as the evaluation of $\nu' - \nu$ on $\mathcal{G}_{\geq (x_2,y)} \backslash \mathcal{G}_{\geq (x_1,y)}$ which can be written in terms of $m_\nu$ and $m_{\nu'}$.
    \[\nu'(\pi_1^{-1}((x_1,x_2])) - \nu(\pi_1^{-1}((x_1,x_2])) = m_{\nu'} (x_1,y) - m_{\nu'}(x_2,y) - m_\nu (x_1,y) + m_\nu(x_2,y)\]
    Using the decompositions of the potentials into horizontally and vertically flat sections, this becomes
    \[v'(x_2,y) - v'(x_1,y) - v(x_2,y) + v(x_1,y),\]
    where the contributions from $u$ and $u'$ have all cancelled. This is equal to the right hand side $f_{\mathcal{L},\mathcal{L}'}(x_2) - f_{\mathcal{L},\mathcal{L}'}(x_1)$. 
    
    It follows that $f_{\mathcal{L},\mathcal{L}'} = f_{\nu,\nu'} + C$ for some constant $C$, because every interval in $\del\Gamma$ can be subdivided into intervals $[x_1,x_2]$ such that there exists $y$ satisfying our hypothesis on $x_1,x_2,y$. Finally,
    \[\sup f_{\nu,\nu'} - \inf f_{\nu,\nu'} = \sup f_{\nu,\nu'} + \sup f_{\nu',\nu} = \sup f_{\mathcal{L},\mathcal{L}'} + \sup f_{\mathcal{L},\mathcal{L}'}\]
    therefore $d(\nu,\nu') = d(\mathcal{L},\mathcal{L}') + d(\mathcal{L}',\mathcal{L})$.

\end{proof}

Translation length for relative metrics is defined in exactly the same way as for ordinary metrics. If $d$ is a relative metric on an $\R$ bundle $L$ over a space $X$, and $\phi: L\to L$ is a bundle map preserving $d$, then we choose a point $\tilde{x}\in L$ and take the limit of $d(\tilde{x},\phi^n(\tilde{x}))/n$ as $n$ goes to $\infty$. Again, if $d(\tilde{x},\phi^n(\tilde{x})) = n d(\tilde{x},\phi(\tilde{x}))$ then the translation length of $\phi$ is $d(\tilde{x},\phi(\tilde{x}))$. 
\begin{lemma}
    Let $P\in \mathcal{A}(S)$, and let $\gamma\in \Gamma$ be a non-trivial group element. The translation length of $\gamma$ acting on $L_P$ is the period $l_P(\gamma)$. 
\end{lemma}
\begin{proof}
Let $\mu$ be the curvature of $P$, and let $\nu\in X_\mu$ be such that $(\gamma^-,\gamma^+)$ is not below any support points of $\nu$ or above any support points of $\bar{\nu}$. The supremum of $f_{\nu,\gamma\nu}$ is attained at $\gamma^-$. Hence, if $\mathcal{L} = (\nu,v,u)$ is a lift of $\nu$ to $L_P$, the supremum of $f_{\mathcal{L},\gamma\mathcal{L}}$ is also attained at $\gamma^-$. The value $f_{\mathcal{L},\gamma\mathcal{L}}(\gamma^-)$ is $\gamma v(\gamma^-,y) - v(\gamma^-,y)$ for any $\gamma^- < y < \gamma^+$, which is the period $l_P(\gamma)$. 
\end{proof}

\section{Tropical rank $2$ currents}
In this section we show that tropical rank $2$ currents are measured laminations, and that the space of holonomy zero lower submeasures of a measured lamination is an $\R$-tree.
\subsection{Symmetry}
In Bonahon's original work \cite{Bonahon88}, geodesic currents were defined to be invariant under the involution $(x,y) \mapsto (y,x)$ of $\mathcal{G}$. Here we call such geodesic currents \textbf{symmetric}. To show symmetry of rank $2$ currents, we first show that holonomies of certain paths vanish. 

\begin{definition}
    For any three distinct points $x,y,z\in \del\Gamma$, let $[x,y,z] \in Z_1(\mathcal{G})$ denote the following taxi path.
    \[(x,y)\to (x,z)\to (y,z)\to (y,x) \to (z,x) \to (z,y) \to (x,y)\]
    If $h$ is a holonomy function, $h([x,y,z])$ is referred to as a \textbf{triple ratio} of $h$.
\end{definition}

\begin{lemma} All triple ratios of rank $2$ holonomy functions are $-1$ and all triple ratios of tropical rank $2$ holonomy functions are $0$.
\end{lemma}
\begin{proof}
    Let $M$ be a potential for a rank $2$ holonomy function $H\in \mathcal{H}(\R^*)$. By definition,
    \[\det 
    \begin{bmatrix}
        0 & M(x,y) & M(x,z) \\
        M(y,x) & 0 & M(y,z) \\
        M(z,x) & M(z,y) & 0 
    \end{bmatrix}=0\]
    for any distinct $x,y,z\in \del\Gamma$. 
    This implies the triple ratio is $-1$.
    \[H([x,y,z]) = \frac{M(x,y)M(y,z)M(z,x)}{M(x,z)M(y,x)M(z,y)} = -1\]
    Now let $m$ be a potential for a tropical rank $2$ holonomy function $h\in \mathcal{H}(\R)$. Since there are only two sums in the tropical determinant which are not $-\infty$, they must coincide, so we have $m(x,y)+m(y,z)+m(z,x) = m(x,z)+m(y,x)+m(z,y)$, implying that the triple ratio is zero.
\end{proof}

\begin{lemma}
\label{vanishing triple ratios implies symmetry}
    A holonomy function $h\in \mathcal{H}(S,\R)$ is symmetric if and only if it has trivial triple ratios.
\end{lemma}
\begin{proof}
Suppose $\tau^*h = -h$. Then
\[h([x,y,z]) = h(\tau([x,y,z])) = -h([x,y,z])\] because
the cycle $[x,y,z]$ is $\tau$-invariant. Hence, $h([x,y,z]) = 0$.

Now we show that vanishing of triple ratios implies symmetry. For any distinct $x, x', y, y'\in \del \Gamma^\circ$, the identity 
\[[x,x',y'] - [x,x',y] \simeq [x,x';y,y'] - [y,y';x,x']\]
of cycles holds, as depicted in Figure \ref{triple ratio difference}. 
\begin{figure}[h]
    \centering
    \includegraphics[width=7cm]{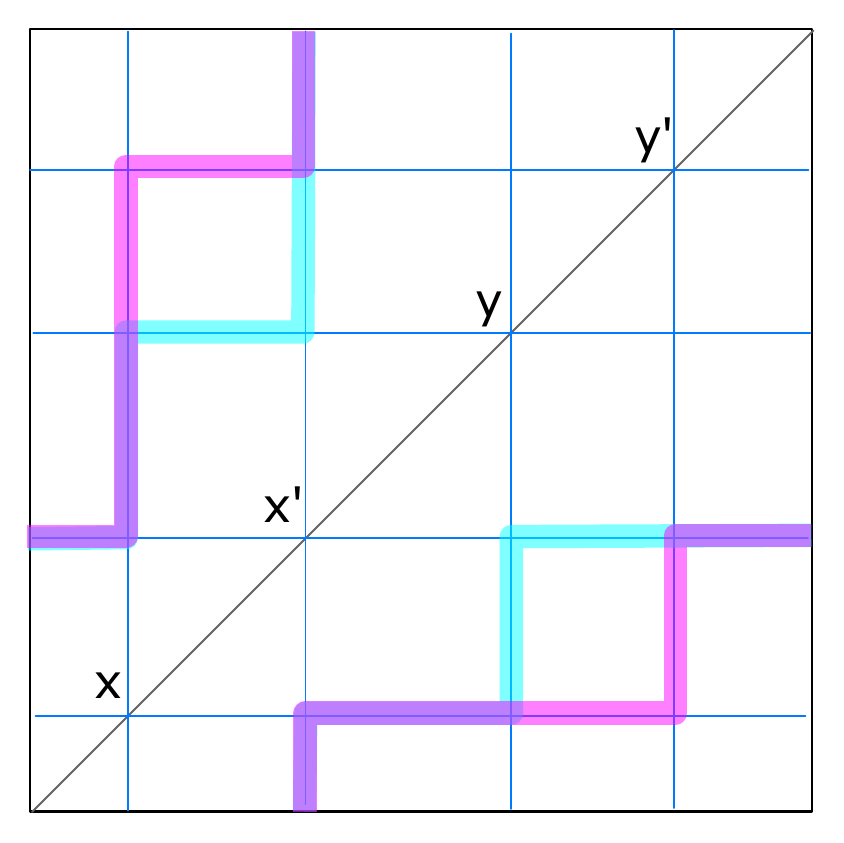}
    \caption{A relation between taxi cycles}
    \label{triple ratio difference}
\end{figure}
If all triple ratios of a holonomy function $h$ vanish, then $h([x,x';y,y']) = h([y,y';x,x'])$. Since $h$ is determined by its cross ratios, $h$ must be symmetric.

    
\end{proof}

A measured lamination is a geodesic current which is symmetric and has no self-intersection. By lemma \ref{no n-intersection} a tropical rank $2$ cross ratio has no self intersection, so must be a measured lamination. 

The original construction of an $\R$-tree from a measured lamination in \cite{MorganShalen91} is very intuitive, and gives the same result as $X_\mu$. Choose a hyperbolic structure on $S$, and identify $\mathcal{G}$ with the space of geodesics in $\tilde{S}$.  Let $\tilde{L}\subset \tilde{S}$ be the union of geodesics parameterized by $\supp(\mu)$. Let $V$ denote the set of components of $\tilde{S}\backslash \tilde{L}$. We define the distance between $v_1,v_2\in V$ to be the measure of the set of geodesics in $L$ which separate $v_1$ from $v_2$. By convention we divide by $2$ to compensate for counting each geodesic with two orientations. Morgan and Shalen show that there is a unique minimal $\R$-tree into which $V$ isometrically embeds. If $\mu$ is discrete, then $V$ is the set of vertices, each leaf of $\tilde{L}$ corresponds to an edge, and the measure $\mu$ assigns to that leaf is the length of the edge.

\begin{figure}[h]
    \centering
    \includegraphics[width=4cm]{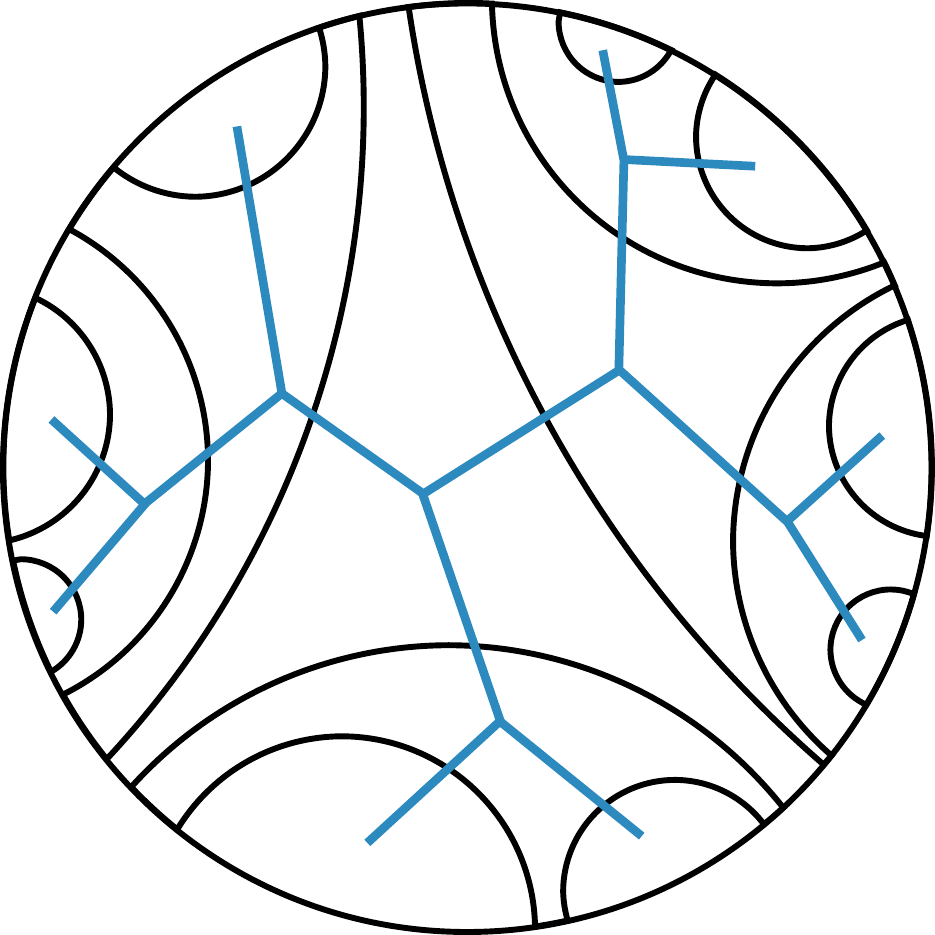}
    \caption{The dual tree to a measured lamination}
    
\end{figure}

\subsection{Holonomy zero lower submeasures}
A submeasure $\nu$ of a geodesic current $\mu$ is called symmetric if $\tau_*\nu = \bar{\nu}$, where $\bar{\nu} := \mu - \nu$.

\begin{lemma}
    A lower submeasure $\nu$ of a measured lamination $\mu$ is holonomy zero if and only if it is symmetric.
\end{lemma}
\begin{proof}
We showed in lemma \ref{symmetric implies hol zero} that symmetric lower submeasures are holonomy zero. Conversely, suppose $\nu$ is holonomy zero. Let $\epsilon = \nu - \tau_* \bar{\nu}$. Expanding this definition,
    \[\epsilon = \nu - \tau_*(\mu - \nu) = \nu + \tau_*\nu - \mu\]
    we see that $\epsilon$ is $\tau$-invariant.
Decompose $\epsilon$ as $\epsilon^+ - \epsilon^-$, where $\epsilon^\pm$ are positive measures, using the Hahn decomposition theorem. Suppose $\epsilon^+$ is non-zero, and let $(x,y)\in \supp(\epsilon^+)$. By symmetry, $(y,x)\in \supp(\epsilon^+)$. By monotonicity, $\epsilon^-$ is zero on the set of geodesics greater than or less than or equal to $(x,y)$ and $(y,x)$, that is, all geodesics not intersecting $(x,y)$. Since $\mu$ has no self-intersection, $\epsilon^-$ is also zero on the set of geodesics intersecting $(x,y)$, so $\epsilon^-$ is zero on all of $\mathcal{G}$. 
The total measure of $\epsilon$ must be zero because both $\nu$ and $\tau_*\nu$ are holonomy-zero, so if $\epsilon^+$ is non-zero then $\epsilon^-$ is also non-zero. In conclusion, $\epsilon = 0$.
\end{proof}

\subsection{$X_\mu$ is a tree}

Firstly, the metric on $X_\mu$ has a simpler form in the rank $2$ case.
\begin{lemma}
    If $\nu_1$ and $\nu_2$ are symmetric lower submeasures of a geodesic lamination $\mu$, then there is a unique geodesic from $\nu_1$ to $\nu_2$, and $d(\nu_1,\nu_2) = |(\nu_2 - \nu_1)^+|$.
\end{lemma}
\begin{proof}
    The difference $\epsilon = \nu_2 - \nu_1$ is anti-invariant with respect to $\tau$:
    \[\tau_*(\nu_2 - \nu_1) = (\mu - \nu_2) - (\mu - \nu_1) = \nu_1 - \nu_2\]
    This means that $\epsilon^- = \tau_* \epsilon^+$. If $(x,y)\in \supp(\epsilon^+)$, then $\epsilon^+$ is zero on the set of all geodesics intersecting $(x,y)$
    because $\mu$ is a measured lamination. It is also zero on $\mathcal{G}_{\leq (y,x)}\cup \mathcal{G}_{\geq (y,x)}$
    because $(y,x)\in \supp(\epsilon^-)$.
    It follows that $\epsilon^+$ is supported within $\mathcal{G}_{\leq (x,y)}\cup \mathcal{G}_{\geq (x,y)}$. We conclude that $\supp(\epsilon^+)$ is a totally ordered subset of $\mathcal{G}$. 

    We can find $z,z'\in \del\Gamma$ such that for any $(x,y)\in \supp(\epsilon^+)$ we have $z < x \leq z' < y \leq z$. If we integrate $(\pi_+)_*\epsilon$ to get a function $f_{\nu_1,\nu_2}:\del\Gamma\to \R$, its maximum will be attained at $z'$, its minimum will be attained at $z$, and the difference $d(\nu_1,\nu_2) = f_{\nu_1,\nu_2}(z')-f_{\nu_1,\nu_2}(z)$ will be the integral of $(\pi_+)_*\epsilon$ over $(z,z']$ which is $|\epsilon^+|$. 

    The union of all geodesics in $X_\mu$ connectiong $\nu_1$ to $\nu_2$ is the set of lower submeasures $\eta$ such that $d(\nu_1,\eta) + d(\eta,\nu_2) = d(\nu_1,\nu_2)$. In other words,
    \[|(\eta - \nu_1)^+| + |(\nu_2 - \eta)^+| = |\epsilon^+|.\]
    The only way this equality of masses can hold is if it holds at the level of measures:
    \[(\eta - \nu_1)^+ + (\nu_2 - \eta)^+ = \epsilon^+.\]
    The difference $(\eta - \nu_1)^+$ must be a lower submeasure of $\epsilon^+$ because $\eta$ is a lower submeasure. A point on a geodesic from $\nu_1$ to $\nu_2$ thus corresponds with a lower submeasure of $\epsilon^+$. The space of lower submeasures of $\epsilon^+$ is the interval $[0,d(\nu_1,\nu_2)]$ because $\supp(\epsilon^+)$ is totally ordered.
\end{proof}

\begin{lemma}
\label{measured lamination implies R tree}
If $\mu$ is a measured lamination, $X_\mu$ is an $\R$-tree.
\end{lemma}
\begin{proof}
One way to show that $X_\mu$ is an $\R$-tree is to show that it is a $0$-hyperbolic metric space. Let $\nu_1,\nu_2,\nu_3\in X_\mu$. The three partitions of $\mu$ give rise (using the Hahn decomposition theorem) to a partition of $\mu$ into eight pieces
\[\mu = \nu_1\cap \nu_2 \cap \nu_3\; +\; \bar\nu_1\cap \nu_2 \cap \nu_3 \;+ \cdots\]
where for two positive measures $\alpha$ and $\beta$, $\alpha\cap\beta$ denotes the biggest measure less than $\alpha$ and $\beta$. We will show that the geodesics connecting any two of the $\nu_i$ pass through 
\[\nu_0 = \nu_1\cap \nu_2 \cap \nu_3 + \bar\nu_1\cap \nu_2 \cap \nu_3 + \nu_1\cap \bar\nu_2 \cap \nu_3 + \nu_1\cap \nu_2 \cap \bar\nu_3\]
Let's check the difference between $\nu_0$ and $\nu_1$, and between $\nu_0$ and $\nu_2$.
\[\nu_0 - \nu_1 = \bar\nu_1\cap \nu_2\cap\nu_3 - \nu_1\cap\bar\nu_2\cap\bar\nu_3\]
\[\nu_0 - \nu_2 = \bar\nu_2\cap \nu_1\cap\nu_3 - \nu_2\cap\bar\nu_1\cap\bar\nu_3\]
Since these two differences are totally disjoint, there will be no canceling when we add, so the triangle inequality will be an equality for $\nu_1,\nu_0,\nu_2$. This means that $\nu_0$ lives on the geodesic from $\nu_1$ to $\nu_2$. The same is true of $\nu_1$, $\nu_3$ and $\nu_2$, $\nu_3$. We have shown that $X_\mu$ is $0$-hyperbolic.
\end{proof}

\section{Currents from Finsler metrics}
\label{sec:currents from metrics}
In this section we show how to extract a geodesic current from a Finsler metric on $S$ which is not quite negatively curved, and not necessarily symmetric. In contrast to the negatively curved case, the current may be singular. This current will be the curvature of a bundle with connection on $\mathcal{G}^\circ$ whose periods are lengths of curves in $S$.

In Section \ref{rank 3} we will apply the theory to Finsler metrics whose length spectra arise in $\del_{\lambda_1}\! \Hit^3(S)$, namely triangular Finsler metrics. Triangular Finsler metrics exhibit the the eccentricities that we will have to deal with in this section: asymmetry, and non-uniqueness of geodesics, so we define them now as an example to keep in mind.
\begin{definition}
Let $\mu$ be a cubic differential on a Riemann surface $C$; that is, a holomorphic section of $(T^*C)^{\otimes 3}$. The Finsler metric $F^\Delta_\mu$ is defined by
\[F^\Delta_\mu(v) := \max_{\{\alpha\in T_x^*C:\alpha^3=\mu_x\}} 2 Re(\alpha(v))\]
where $x\in C$ is a point, and $v\in T_x C$ is a tangent vector.
\end{definition}

\subsection{Horofunction Boundaries}
In this subsection we recall Gromov's notion of horofunction boundary \cite{MR0624814} but in the case of asymmetric metrics.
Let $X$ be a proper, geodesic, asymmetric metric space. Let $C(X)$ denote the space of continuous real valued functions with the topology of uniform convergence on compact subsets. There are natural embeddings $D_+,D_-:X\to C(S)$ given by 
\[D_+(x) = d(-,x)\]
\[D_-(x) = d(x,-)\]
which are both isometric embeddings from $X$, with the metric $\max(d(x,y), d(y,x))$, to $C(X)$ equipped with the supremum norm.
\begin{definition}
The \textbf{plus horofunction compactification}, $\del_+^{h}X$, is the closure of $D_+(X)$ in the quotient $C(X)/\R$ of continuous functions by constant functions. Similarly, the \textbf{minus horofunction compactification} $\del_-^h X$ is the closure of $D_-(X)$ in $C(X)/\R$.  The plus and minus horofunction boundaries $\del^h_+X$, and $\del_-^{h}X$ are the complements of $X$ in these two compactifications. Plus and minus \textbf{horofunctions} are functions on $X$ which represent points in $\del^h_-X$, and $\del^h_+X$  respectively.
\end{definition}
 Since $\del_-^h X$ is just $\del_+^h X$ for the reversed metric, we will sometimes make statements only for $\del_+^h X$. 

As an example, if $X = \C$ with the triangular Finsler metric $F^\Delta_{dz^3}$, then the horofunction boundary $\del_-^h X$ is a circle with a natural cell decomposition into three copies of $\R$, and three points. Linear functions \[h(z) = 2 Re(\zeta z) + C\] with $\zeta^3=1$
will give three points in $\del_-^h(\C,F^\Delta_{dz^3})$. 
There will also be horofunctions of the form \[h(y) = \max (2 Re(\zeta z)+C, 2 Re(\zeta' z)+C')\] for two distinct third roots of unity $\xi,\xi'$ which will descend to three copies of $\R$ in $\del_-^h$.

In general, the horofunction boundary can be quite different from the visual boundary, but in the Gromov hyperbolic case there is a close relationship. 
\begin{lemma} [\cite{CoornaertPapadopoulos2001}]
    For a Gromov hyperbolic geodesic metric space $X$, the visual boundary $\del\! X$ is the quotient of the horofunction boundary $\del^{h}\!X$, where two horofunctions are identified if their difference is bounded.
\end{lemma}
The map is defined as follows: For any horofunction $h$, and any $p\in X$ with $h(p)=0$, we can find a geodesic ray $\gamma$ starting at $p$ satisfying $h(\gamma(t)) = -t$. We find $\gamma$ by taking geodesics from $p$ to $q_i$ for a sequence $q_i\in X$ converging to $[h]$. Simply take $\gamma$ to be the limit of a convergent subsequence for the topology of convergence on compact subsets. Gromov hyperbolicity forces any two geodesic rays $\gamma,\gamma'$ which satisfy this property, with respect to two horofunctions $h,h'$ with $h-h'$ bounded, to be bounded distance from eachother, thus represent the same point in $\del X$. Note however, that $\gamma$ may not converge to $[h]$ in the horofunction compactification. We will denote the projections to the visual boundary by $v_-:\del_-^h X\to \del X$ and $v_+:\del_+^h X \to \del X$. 

In the generality we need, there is no natural map from the visual boundary to the horofunction boundary, but geodesic rays do always converge to horofunctions. Endpoints of geodesic rays in the horofunction boundary are called Bussmann points. 

\begin{lemma}
    Geodesic rays converge in the horofunction compactification. If two rays $\gamma_1,\gamma_2:[0,\infty)\to X$ are asymptotic, i.e.
    \[\lim_{t\to\infty} d(\gamma_1(t),\gamma_2(t+T)) = 0\]
    for some $T$, then they converge to the same point.
\end{lemma}
\begin{proof}
    The second claim is clear from the triangle inequality, but the first statement is a bit less immediate.
    Let $\gamma:[0,\infty)\to X$ be a geodesic, meaning $d(\gamma(t),\gamma(t'))=t'-t$ for all $0\leq t\leq t'$. Let $x\in X$. We would like to show that the path of functions
    \[h^t_{\gamma}(x) := d(x,\gamma(t)) - t\]
    converges on all compact subsets of $X$ as $t$ goes to infinity. First we show that $h^t_{\gamma}(x)$ is decreasing in $t$. Let $t' \geq t$, and use the triangle inequality.
    \[h^t_{\gamma}(x) - h^{t'}_{\gamma}(x) = d(x,\gamma(t)) + (t'-t) - d(x,\gamma(t')) \geq 0\]
    The function $h^t_{\gamma}(x)$ is also bounded below as a function of $t$:
    \[d(\gamma(0),x) + d(x,\gamma(t)) \geq t\]
    \[d(x,\gamma(t)) - t \geq  - d(\gamma(0),x)\]
    Since the path of functions $h^t_\gamma(x)$ is $1$-lipshitz in $x$, bounded below on compact subsets of $X$, and monotonically decreasing in $t$, it must converge uniformly on compact subsets of $X$.
\end{proof}

\subsection{A bundle with taxi connection}
We will define a pairing between minus horofunctions and plus horofunctions, from which we construct a cross ratio, and even an $\R$-bundle with taxi connection on a subset of $\del_-^hX\times \del_+^h X$. 

\begin{definition}
    Let $X$ be a proper asymmetric geodesic metric space. The \textbf{pairing} of a minus horofunction $g$, and a plus horofunction $h$ is the infimum of their sum.
    \[\langle g, h \rangle := \inf_{X} (g + h) \in [-\infty,\infty)\]
    The \textbf{cross ratio} of $[g_1],[g_2]\in \del_-^h X$ and $[h_1],[h_2]\in \del_+^h X$ is the following combination of pairings.
    \[b([g_1],[g_2];[h_1],[h_2]) = \langle g_1,h_1\rangle + \langle g_2,h_2\rangle - \langle g_1,h_2\rangle - \langle g_2,h_1\rangle\]
    
\end{definition}

The pairing between horofunctions is a sort of renormalized limit of distance between points. It is closely related to the Gromov product.
\begin{lemma}
\label{gromov product and pairing}
    Let $X$ be a Gromov hyperbolic,  proper, geodesic, asymmetric metric space. Let $p\in X$ be a basepoint. Suppose $x_i$ converges to $[g]\in\del_-^h X$ and $y_i$ converges to $[h]\in \del_+^h X$ where $g$ and $h$ are normalized to vanish on $p$. Then
    \[\lim_{i\to \infty} d(x_i,y_i) - d(x_i,p) - d(p,y_i) = \langle g,h\rangle\]
\end{lemma}
\begin{proof}
    First note that $d(x_i,y_i)$ is the minimum of the function $d(x_i,z) + d(z,y_i)$. The set of minima is the union of all geodesics connecting $x_i$ to $y_i$. Re-write the left hand side:
    \[\lim_{i,j\to \infty} \inf_{z\in X} d(x_i,z) - d(x_i,p) + d(z, y_i) - d(p,y_i)\]
    Gromov hyperbolicity implies that there is a compact region of $X$ that all geodesics from $x_i$ to $y_i$ pass through for all $i$. These infimums can thus be all attained for $z$ constrained to this compact region. On such region, $d(x_i,z) - d(x_i,p)$ is converging uniformly to $g(z)$, and $d(z, y_i) - d(p,y_i)$ is converging uniformly to to $h(z)$.  The expression becomes
    \[\inf_{z\in X} h(z) + g(z)  = \langle h, g \rangle\]
    
\end{proof}
Note that $d(x_i,y_i) - d(x_i,p) - d(p,y_i)$ is just $-2$ times the Gromov product of $x_i$ and $y_i$. If $X$ is Gromov hyperbolic, and $x_i$ and $y_i$ converge in the visual boundary, then the Gromov product diverges if and only if they converge to the same point. Therefore, if $X$ is Gromov hyperbolic, then $\langle g,h\rangle = -\infty$ if and only if $v([g]) = v([h])$.

Lemma \ref{gromov product and pairing} implies a more geometric formula for the cross ratio of four horofunctions as a limit of ``cross distances".
\begin{lemma}
\label{cross distance becomes cross ratio}
    If $x_{1,i}, x_{2,i}$ limit to $[g_1],[g_2]\in \del_-^h X$, and $y_{1,i}, y_{2,i}$ limit to $[h_1],[h_2]\in \del_+^h X$, then 
    \[b([g_1],[g_2];[h_1],[h_2]) = \lim_{i\to\infty}[d(x_{1,i}, y_{1,i}) + d(x_{2,i}, y_{2,i}) - d(x_{1,i}, y_{2,i}) - d(x_{2,i}, y_{1,i})]\]
\end{lemma}

In much the same way as for Anosov representations, we construct a bundle with connection for which these cross ratios are holonomies.
\begin{definition}
    Let $\mathcal{G}^h_X$ be the subset of $\del_-^h X \times \del_+^h X$ consisting of pairs which map to two distinct visual boundary points. 
    Let $U^h_X$ denote the set of pairs $(g,h)$ with $\langle g, h\rangle = 0$. Endow $U^h_X$ with the $\R$ action $r\cdot(g,h) = (g-r,h+r)$. It is a principal $\R$ bundle on $\mathcal{G}^h_X$. Endow $U^h_X$ with the taxi connection whose horizontal and vertical flat sections are sections for which one coordinate is constant.
\end{definition}

We can usually understand this connection using geodesics. A geodesic $\gamma:\R\to X$, gives a pair of horofunctions.
\[\gamma(-\infty) := \lim_{t\to-\infty} [d(\gamma(t),p) + t]\]
\[\gamma(\infty) := \lim_{t\to\infty} [d(p,\gamma(t)) - t]\]
These satisfy $\langle \gamma(-\infty),\gamma(\infty)\rangle = 0$, thus $[\gamma]:=(\gamma(-\infty),\gamma(\infty))$ is a point in $U^X_h$. Suppose $\gamma_1$ and $\gamma_2$ are parametrized geodesics which are asymptotic in a strong sense: 
\[\lim_{t\to\infty} d(\gamma_1(t),\gamma_2(t)) = 0\]
The triangle inequality implies $\gamma_1(\infty) = \gamma_2(\infty)$, so the corresponding points $[\gamma_1]$ and $[\gamma_2]$ in $U^h_X$ will lie on a flat section. 

\subsection{Cyclic order on horofunction boundary}
Let $d$ be a $\Gamma$-invariant metric on $\tilde{S}$. We would like to push forward the curvature of $U^h_d$ from $\mathcal{G}^h_d$ to $\mathcal{G}$ to define a geodesic current. To do this, we need the curvature of $U^h_d$ to be a positive measure on $\mathcal{G}^h_d$. To define positivity, we construct cyclic orders on $\del_-^h \tilde{S}$ and $\del_+^h \tilde{S}$, which refine the cyclic order on $\del \Gamma$, such that cross ratios are positive when expected.

From now on, let $X = \tilde{S}$ be the universal cover of $S$ equipped with a $\Gamma$-invariant asymmetric metric $d$ such that
\begin{enumerate}
    \item every point in $\del^h_-\tilde{S}$ and $\del^h_+\tilde{S}$ is the limit of a continuous path $[0,\infty)\to X$,
    \item connecting any two visual boundary points $x,y\in \del \Gamma$, there is a bi-geodesic, i.e. a path $\gamma$ such that both $\gamma$ and $\gamma^{-1}$ are geodesic for $d$,
    \item and every point between two parallel bi-geodesics $\gamma$, and $\gamma'$ is itself on a bi-geodesic between $\gamma$ and $\gamma'$.
\end{enumerate}
We call such metrics ``good".

\begin{definition}
Three horofunction boundary points $[h_1],[h_2],[h_3]\in \del_h^+ X$ are cyclically ordered if we can find three counterclockwise ordered rays $\gamma_1,\gamma_2,\gamma_3:[0,\infty)\to X$ which only intersect at $\gamma_1(0)=\gamma_2(0)=\gamma_3(0) = p$, and converge to $[h_1],[h_2],[h_3]$. 
\begin{figure}[h]
    \centering
    \includegraphics[width=0.5\linewidth]{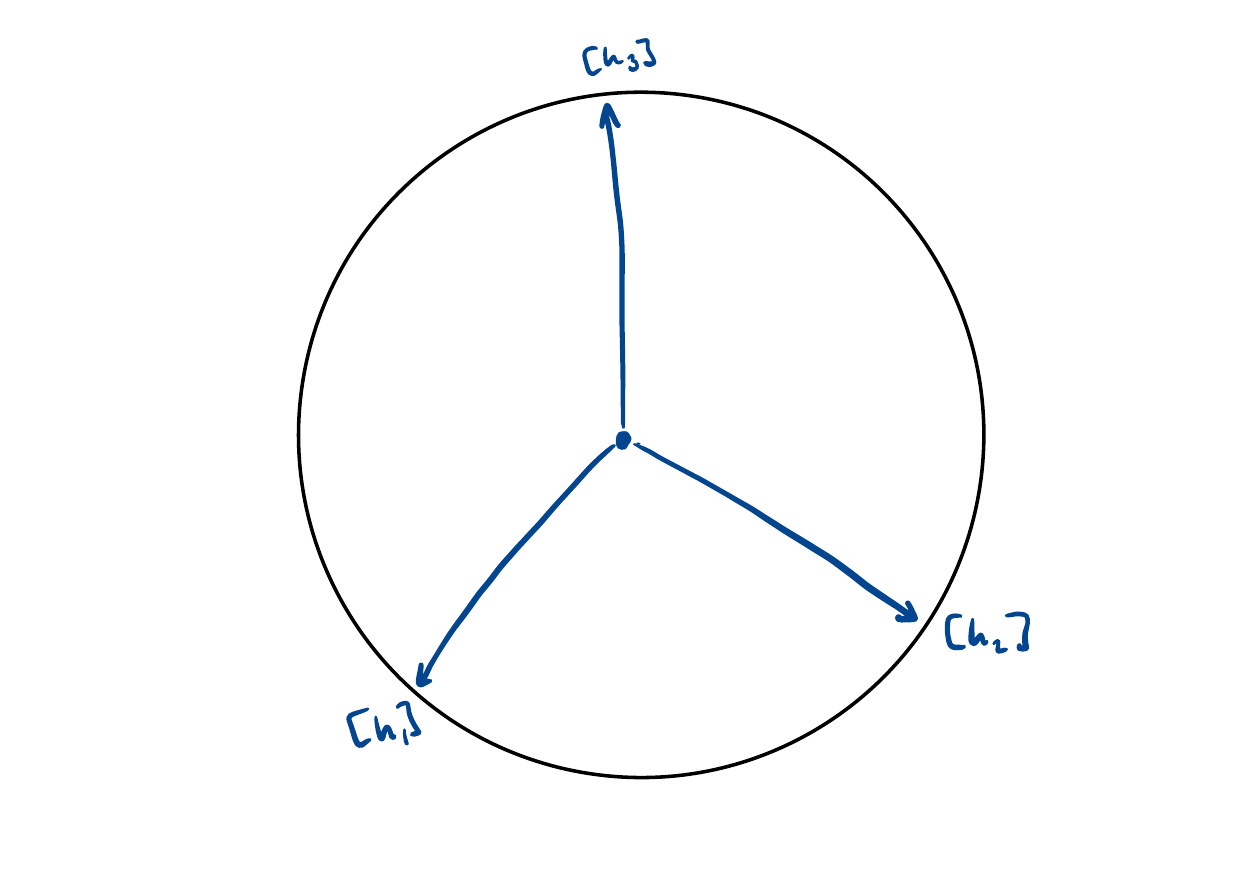}
    \caption{Cyclic order on the horofunction boundary}
    
\end{figure}
\end{definition}
This is compatible with the cyclic ordering on $\del X$. Note that paths converging to distinct horofunctions must eventually be distinct, so it is not hard to find such rays for any triple of distinct boundary points. The next lemma shows this cyclic order is well defined.

\begin{lemma}
\label{cyclic well defined}
    Let $h_1,h_2,h_3$ be horofunctions on $X$, and suppose $v(h_1) = v(h_2) \neq v(h_3)$. If $h_1\leq h_2 < h_3$ and $h_1 \leq h_2 < h_3$, then $[h_1] = [h_2]$.
\end{lemma}
\begin{proof}
    $h_1\leq h_2 < h_3$ implies that $h_1 - h_2$ is an increasing function on any bi-geodesic which starts between $[h_3]$ and $[h_1]$ and ends between $[h_2]$ and $[h_3]$. To see this, choose paths $\gamma_i$ converging to $[h_i]$ which exhibit their cyclic order. Chose a nested sequence of halfspaces $H^k$ bounded by bi-geodesics which converge to $v(h_1)=v(h_2)$. Let $p_1^k,p_2^k$ be the last points on $\gamma_1$ and $\gamma_2$ to hit the boundary of $H_k$. The function $d(-,p_1^k) - d(-,p_2^k)$ is increasing on any bigeodesic from $([h_3],[h_1])$ to $([h_2],[h_3])$ because crossing paths are longer than non crossing paths. This implies the same of $h_1 - h_2$. 
    
    If also $h_2\leq h_1 \leq h_3$, then $h_1 - h_2$ is also increasing on every such two way geodesic. A function which is constant on all such two way geodesics must be constant (because our metric is good), so $[h_1] = [h_2]$.
\end{proof}
\begin{figure}[h]
    \centering
    \includegraphics[width=0.5\linewidth]{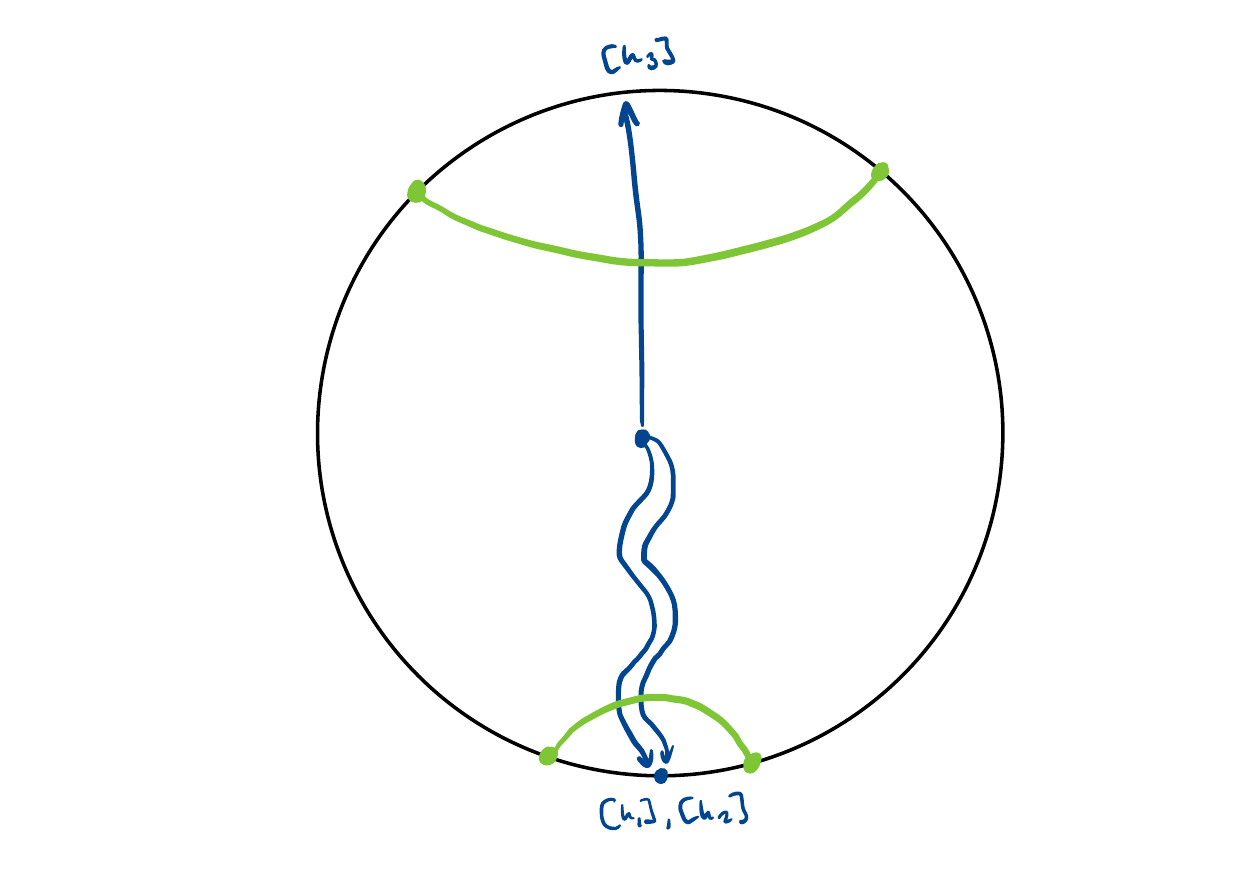}
    \caption{$h_1 - h_2$ must be increasing on the upper bi-geodesic}
    
\end{figure}
A tuple of $n$ horofunctions is cyclically ordered if they can be reached by $n$ cyclically ordered rays which only intersect at their starting points.

\begin{lemma}
    If $[h_1],[h_2],[g_1],[g_2]$ are cyclically ordered, then 
    \[b([h_1],[h_2];[g_1],[g_2])\geq 0.\]
\end{lemma}
\begin{proof}
    Represent all four horofunctions as endpoints of cyclically ordered rays \linebreak $\gamma_1^-,\gamma_2^-,\gamma_1^+,\gamma_2^+$ eminating from $p\in X$. For $r>0$ let $x_1,x_2,y_1,y_2$ be the last points on each ray which are distance $r$ from $p$. These points will be cyclically ordered on the boundary of the the $r$ ball centered at $p$. The cross-distance
    \[d(x_1,y_1) + d(x_2,y_2) - d(x_1,y_2) - d(x_2,y_1)\] 
    is positive because we can find geodesics within the ball $B(p,r)$ connecting each pair of points, and the sum of the crossing geodesics' lengths is always greater than the sum of non-crossing geodesics' lengths. By Lemma \ref{cross distance becomes cross ratio}, the cross ratio is the limit as $r$ goes to infinity which thus must also be positive.
\end{proof}
\begin{figure}[h]
    \centering
    \includegraphics[width=0.5\linewidth]{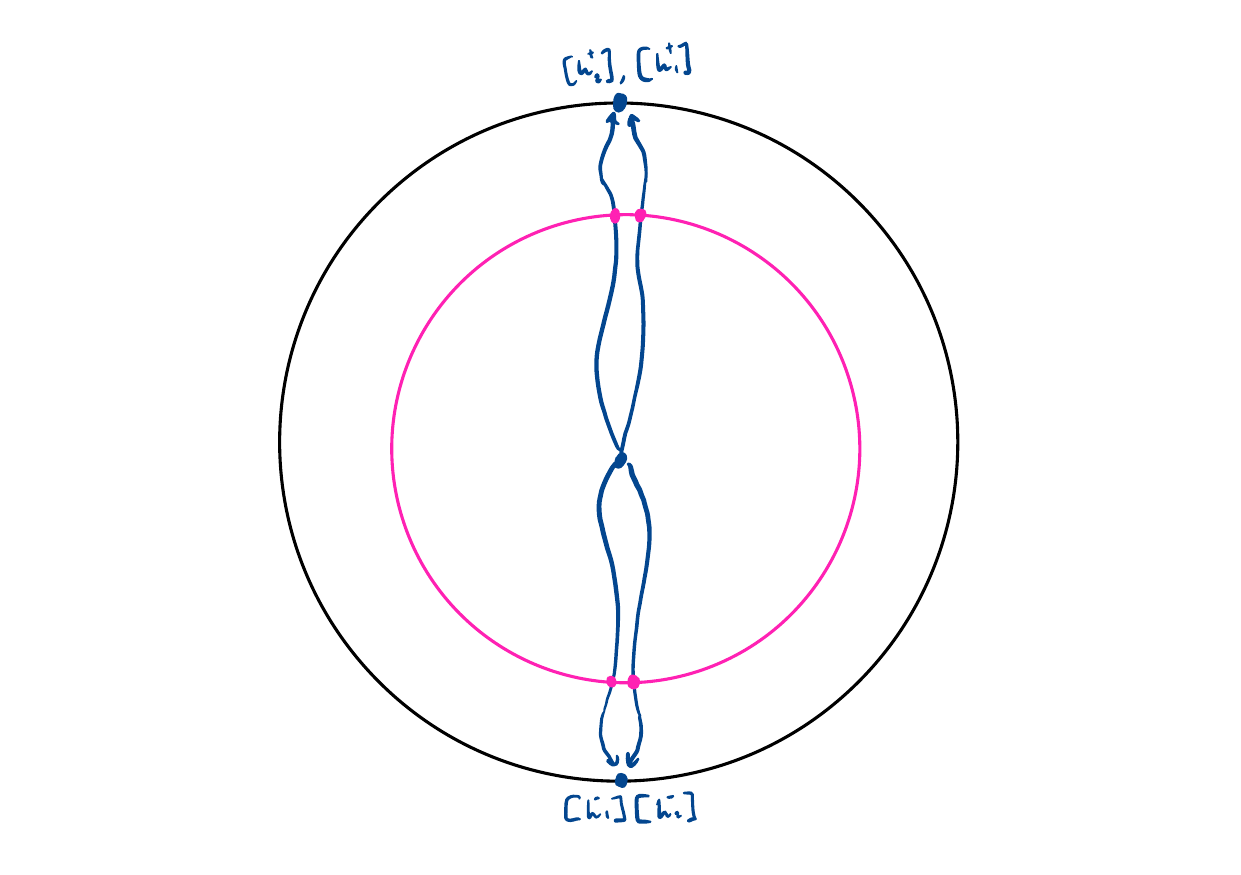}
    \caption{Four cyclically ordered horofunction boundary points}
    
\end{figure}

This cross ratio $b$ is a finitely additive, positive function on rectangles in $\mathcal{G}^h_d$. It is also continuous in its four arguments, which is enough to show it is also $\sigma$-additive, so the curvature of $U^h_d$ is a positive measure $\mu^h_d$ on $\mathcal{G}^h_d$. Pushing forward $\mu^h_d$ to $\mathcal{G}$ gives a geodesic current $\mu_d$ which we call the \textbf{Lioville current} of $d$. 

Let $v:\mathcal{G}^h_d\to \mathcal{G}$ be the projection given by restricting $v_- \times v_+$. For every $p=(x,y)\in \mathcal{G}^\circ$, $\mu_d(p) = 0$, so $\mu^h_d(v^{-1}(p))=0$. This means that the connection on $U^h_d$ must be flat on $v^{-1}(p)$. The preimage $v^{-1}(p)$ is a little box $v_-^{-1}(x)\times \pi_+^{-1}(y)$. We can define $U_p$ to be the set of flat sections over $\pi^{-1}(p)$, and this will be non-empty because the connection is flat over this box. We call the bundle with taxi connection $U_d\in \mathcal{A}(S)$ the generalized geodesic flow bundle of $d$.

\subsection{Structure of horofunction boundaries}
So far we have no idea how many horofunctions map to each Gromov boundary point. 
The following non-degeneracy fact gives some control.
\begin{lemma}
\label{horofunction non-degen}
    If $h_1$ and $h_2$ are two plus horofunctions such that $\langle g, h_1 \rangle = \langle g, h_2 \rangle$ for every minus horofunction $g$, then $h_1 = h_2$.
\end{lemma}

\begin{proof}
    First we show that $h_1-h_2$ vanishes at infinity. Suppose $y_i$ converges to $[g]\in \del_+^h \tilde{S}$. 
    \[\lim_{i\to \infty} h_1(y_i) - h_2(y_i) = \lim_{i\to\infty} (h_1(y_i) - d(p,y_i)) - (h_2(y_i) - d(p,y_i))\]
    \[ = \langle h_1, g\rangle - \langle h_2, g\rangle = 0\]
    Here, $g$ is the representative of $[g]$ which vanishes on $p$, and we have used lemma \ref{gromov product and pairing}. Clearly we must have $v_+(h_1) = v_+(h_2)$. By lemma \ref{cyclic well defined}, $h_1 - h_2$ must be either increasing or decreasing on every bi-geodesic, meaning that it must be zero on every bi-geodesic with neither endpoint at $v_+(h_1)$. We have assumed that there is such a bi-geodesic through every point, so $h_1-h_2 = 0$.
\end{proof}

By definition, $\mu_d$ is the pushforward to $\mathcal{G}$ of $\mu_d^h$, which is defined using cross ratios of horofunctions. The constraints that $\mu_d$ satisfies, derived simply from local finiteness and $\Gamma$ invariance, imply that many cross ratios of horofunctions have to vanish.
\begin{lemma}
\label{horofunction cross ratio vanish}
    Let $g_1,g_2$ be minus horofunctions, and $h_1, h_2$ be plus horofunctions such that $g_1 < g_2 < h_1 < h_2$ and $v(g_1) = v(g_2)$. The cross ratio $b(g_1,g_2;h_1,h_2)$ is zero unless $v(g_1) = v(g_2) = a^-$ and $h_1^+ \leq a^+ \leq h_2^+$ for some $a\in \Gamma$. 
\end{lemma}

Lemmas \ref{horofunction non-degen} and \ref{horofunction cross ratio vanish} immediately tell us something interesting about horofunctions:
\begin{lemma}
\label{unique horofunction}
    There are unique plus and minus horofunction boundary points mapping to each visual boundary point which is not fixed by any group element. 
\end{lemma}

If we start with a good, $\Gamma$-invariant metric $d$ on $\tilde{S}$, we get a Lioville current $\mu$, from which we can construct a metric space $X_\mu$. It is natural to wonder if $(\tilde{S},d)$ is related to $X_\mu$. For negatively curved metrics and triangular Finsler metrics, (and probably some nice class of metrics containing these two examples) there is an embedding $\tilde{S}\to X_\mu$ which is an isometric embedding for the symmetrized metric on $\tilde{S}$. There is also a lift of this embedding to $L_U$, where $U_d\in \mathcal{A}(S)$ is the generalized geodesic flow bundle, which is an isometric embedding for the relative metric on $L_U$. These constructions probably go through for metrics which are negatively curved in a very weak sense (maybe "good" is enough) but for the sake of staying on topic we will just do these constructions for triangular Finsler metrics in the next section.

\section{Tropical rank $3$ currents}
\label{rank 3}
In this section we show that certain paths in $\Hit^3(S)$, called cubic differential rays, converge to particularly nice geodesic currents, namely currents of descending real trajectories of cubic differentials. Then we show that for such a geodesic current $\mu$, the space $X_\mu$ is simply $\tilde{S}$. We don't know of any other geodesic currents for which this is the case.

We use the main theorem 
of \cite{Reid2023} where it was shown that along the cubic differntial ray corresponding to $\alpha$, the $\lambda_1$ spectrum approaches the length spectrum of a Finsler metric $F_\alpha^\Delta$. Here we show that there is a natural trivialization of the relative metric recovering the $\Delta$-Finsler metric $F^\Delta_\alpha$. In \cite{Reid2023} it was conjectured that $\Delta$-Finsler metrics are determined by their length spectra, and that these length spectra form a dense subset of $\del_{\lambda_1}\Hit^3(S)$. The fact that we can directly construct the Finsler surface $(S,F^\Delta_\alpha)$ from its Lioville current, which is in turn determined by its length spectrum affirms that $\Delta$-Finsler surfaces are determined by length spectrum. It is still unknown whether length spectra of $\Delta$-Finsler metrics are open and dense in $\del_{\lambda_1}\Hit^3(S)$. It is also still unknown what $X_\mu$ looks like for the rest of $\del_{\lambda_1}\Hit^3(S)$, though surely it is some combination of $\R$-tree behavior and cubic differential behavior, as observed in a related compactification \cite{OuyangTamburelli21}. 

\subsection{Cubic differential rays}
Let $C$ be a closed Riemann surface. Quadratic and cubic differentials on $C$ are holomorphic sections of $K^2$ and $K^3$ respectively, where $K$ is the cannonical bundle, which for Riemann surfaces is just the holomorphic cotangent bundle. Hitchin \cite{Hitchin92} defined the following family of Higgs bundles $(E,\phi_{\alpha_2,\alpha_3})$ parametrized by $(\alpha_2,\alpha_3)\in H^0(C,K^2)\times H^0(C,K^3)$.
\[E = K\oplus \underline{\C} \oplus K^{-1}  \;\;\;\; \phi_\alpha = \begin{bmatrix}
    0& \alpha_2& \alpha_3 \\
    1& 0& \alpha_2 \\
    0& 1& 0 \\
\end{bmatrix}\] 
Solving the Hitchin equation gives a cannonical flat connection on $E$ which preserves a real structure, giving a diffeomorphism $H^0(C,K^2)\times H^0(C,K^3)\to \Hit^3(C)$. Hitchin's construction works for general split real lie groups, but something special happens in the case of $\Hit^3(C)$. Labourie \cite{labourie_flat_2006} and Loftin \cite{Loftin01} showed that we can set $\alpha_2=0$, and instead range over all complex structures on a smooth surfaces $S$, up to isotopy, and get a parametrization of $\Hit^3(S)$ by the bundle over Teichmuller space whose fibers are cubic differentials. 

The Labourie-Loftin parametrization suggests that to understand the boundary of $\Hit^3(S)$, a good place to start is understanding how holonomies grow along paths in $\Hit^3(S)$ parametrized by a fixed complex structure and a ray of cubic differentials $R\alpha$, where $\alpha$ is a non-zero cubic differential and $R\in \R_{\geq 0}$. This was first investigated in \cite{Loftin07} for the case of loops which are straight lines in the $1/3$ translation structure avoiding zeros, then in \cite{loftin_limits_2022} for general loops, and rephrased and reproved in terms of Finsler metrics in \cite{Reid2023}.

\begin{theorem}
Let $(J_i,\alpha_i)$ be a sequence of pairs of complex structure with cubic differential on a smooth oriented surface $S$ of genus at least $2$, such that $J_i$ converges uniformally to some $J$, and $\alpha_i/R_i^3$ converges uniformally to $\alpha$ for some sequence of positive real numbers $R_i$ tending to $\infty$. Let $\rho_i\in \Hit^3(S)$ be the corresponding sequence of representations. Let $[a]\in [\pi_1(S)]$. Let $F^{\Delta}_\alpha(a)$ denote the infimal length of loops in the free homotopy class $[a]$ with respect to the triangular Finsler metric $F_\alpha^{\Delta}$.
\[\lim_{i \to \infty} \frac{\log|\lambda_1(\rho_i(a))|}{R_i} = F^{\Delta}_\alpha(a)\]
\end{theorem}

Let $\rho_i$ be such a sequence of representations. Let $l_i$, and $\mu_i$ be the $\lambda_1$-spectrum, and cross ratio current of $\rho_i$. Let 
 $l = \lim  (l_i/R_i)$
be the limiting length spectrum. The measures $\mu_i/R_i$ will converge to a limiting current $\mu$ which will be nullhomologous, thus the curvature of an equivariant taxi bundle on $\mathcal{G}^\circ$ with period spectrum $l$. On the other hand $l$ is the length spectrum of $F^\Delta_\alpha$, so it is the period spectrum of the generalized geodesic flow bundle $U_{F^\Delta_\alpha}$. Since equivariant taxi bundles are determined by their periods, it follows that $\mu$ is the curvature of $U_{F^\Delta_\alpha}$. We will show that $\mu$ is proportional to the current of descending real trajectories of $\alpha$. 

\subsection{Cubic differential currents}
To understand the geodesic current associated with the Finsler metric $F^\Delta_{\alpha}$ we need to understand its geodesics. We start by describing geodesics in $\C$ for the metric $F^\Delta_{dz^3}$.

\[F^\Delta_{dz^3} = \max_{\zeta^3 = 1} [2 Re(\zeta dz)]\]
A path $\gamma:\R\to \C$ is a geodesic for $F^\Delta_{dz^3}$ if one of these three one forms is maximal on $\gamma'(t)$ for all $t$. Note that there are infinitely many geodesics from $0$ to $1$, whereas there is a unique geodesic from $0$ to $-1$. More generally, geodesics of the form $\gamma(t) = a - \zeta t$ where $\zeta^3=1$ are ``rigid" in the sense that for any $t_1<t_2$, $\gamma|_{[t_1,t_2]}$ is the unique geodesic segment from $\gamma(t_1)$ to $\gamma(t_1)$. 
\begin{figure}[h]
    \centering
    \includegraphics[width=5cm]{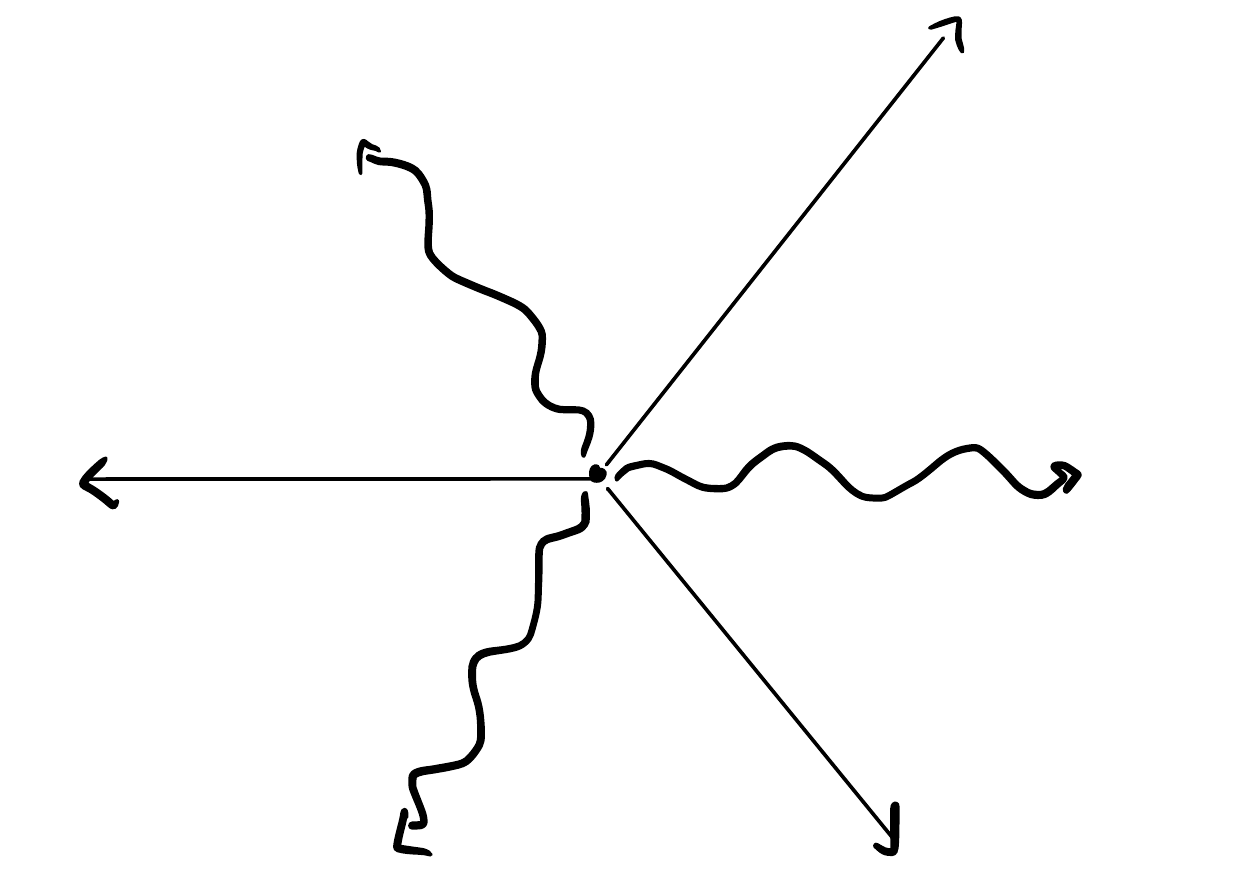}
    \caption{Some $F^\Delta_{dz^3}$ geodesics}
    
\end{figure}

We can understand $F^{\Delta}_{dz^3}$ as the taxicab metric for a city that has three  directions of one way streets. If you can get to some place by taking one of these streets then it is the fastest way to get there, otherwise there are multiple optimal routs. On a Riemann surface, these streets will be called descending real trajectories.

\begin{definition}
    A \textbf{descending real trajectory} of a cubic differential $\alpha$ on a Riemann surface $X$ is a smooth map $\gamma:(-\infty,\infty)\to X$ such that $\alpha(\gamma')=-1$.  A \textbf{generalized descending trajectory} is a non-constant continuous map $\gamma:(-\infty,\infty)\to X$ such that if $\alpha(\gamma(t))\neq 0$ then $\gamma$ is differentiable at $t$ and $\alpha(\gamma'(t))=-1$, and also has angle at least $\pi$ on both sides (in the flat metric) at zeros.
\end{definition}

\begin{lemma}
\label{trajectories are geodesic}
    Generalized descending trajectories are geodesics. They are also rigid: if $\gamma$ is a generalized descending trajectory then for all $t_1 < t_2$ in $\R$, $\gamma|_{[t_1,t_2]}$ is the unique shortest path from $\gamma(t_1)$ to $\gamma(t_2)$.
\end{lemma}
\begin{proof}
     Since $\gamma$ is the unique Euclidean geodesic connecting $\gamma(0)$ to $\gamma(T)$, and the triangle metric is bounded below by the Euclidean metric, any other path from $\gamma(0)$ to $\gamma(T)$ must be longer than $\gamma$ with respect to $F^{\Delta}$. 
\end{proof}

A corollary of Lemma \ref{trajectories are geodesic} is that an $F^\Delta_{\alpha}$-geodesic in $X$ cannot span a bigon with a descending trajectory with both edges oriented the same way. 

We would like a local way of telling whether a path in $\tilde{S}$ is a geodesic. One might expect a result like this because $\tilde{S}$ is Gromov hyperbolic. On the other hand, there is an obstruction to such a result which is the same as for the taxi-cab metric: taking one right turn can be geodesic but two consecutive right turns is not, and the turns can be arbitrarily far apart.

Let $X$ be a Riemann surface with cubic differential $\alpha$. Denote by $\Sigma\subset T^*X$ the triple branched covering of $X$ whose points are cube roots of $\alpha$. This is the spectral curve of the Higgs bundle determined by $\alpha$. Say a path $\gamma:[0,T]\to \tilde{S}$ is liftable if there is a continuous lift $\beta:[0,T] \to \Sigma$ such that $2Re(\beta(\gamma'))$ is always maximal amongst the three square roots. In $(\C,dz^3)$, geodesics are precisely the liftable paths, and descending trajectories are the paths that admit two lifts.

If $\alpha$ has zeros, then a liftable path can use a zero of $\alpha$ to turn straight around, but after ruling this out, we get our desired characterization of geodesics. 

\begin{lemma}
\label{geodesic characterization}
    Let $S$ be equipped with a complex structure and cubic differential $\alpha$. A path $\gamma:[0,T]\to \tilde{S}$ is a geodesic for $F^\Delta_\alpha$ if and only if it is liftable, and is geodesic in some neighborhood of each zero.
\end{lemma}
\begin{proof}
Suppose $\gamma$ is a geodesic. It must be geodesic on the complement of its zeros, thus it is liftable on the complement of the zeros. At the zeros the three cube roots coincide, so $\gamma$ is liftable. Since $\gamma$ is geodesic, it must be geodesic in a neighborhood of each zero.

There must be some geodesic $\gamma_0$ connecting any two points $p$ and $q$ because $\tilde{S}$ is a complete Finsler space. Suppose $\gamma_1$ is another path which is liftable, and geodesic near zeros. We will show that $\gamma_0$ and $\gamma_1$ have the same length. 

We can replace $\gamma_0$ and $\gamma_1$ with piecewise smooth paths which have the same lengths, so we can assume that they are piecewise smooth.  We can further assume that $\gamma_0$ and $\gamma_1$ have disjoint interiors because otherwise we just apply the argument multiple times. Let $D\subset \tilde{S}$ be the disk bounded by $\gamma_0$ and $\gamma_1$. 

The cubic differential $\alpha$ induces a singular Euclidean metric $|\alpha|^{2/3}$ which is flat everywhere except at zeros of $\alpha$ where it has cone points of angles $2\pi + 2\pi k/3$. We will apply the Gauss Bonnet theorem to show that there are no zeros of $\alpha$ in $D$. The Gauss Bonnet formula says
\[\int_D K + \int_{\del D} \kappa  = 2\pi \]
where $K$ is the Gauss curvature, and $\kappa$ is the geodesic curvature of the boundary. In our setting, the integral of Gauss curvature means the sum of cone angles which is $-2\pi/3$ times the number of zeros in $D$ counted with multiplicity, and the integral of geodesic curvature means the total turning angle of the boundary.
\begin{figure}[h]
    \centering
    \includegraphics[width=0.6\linewidth]{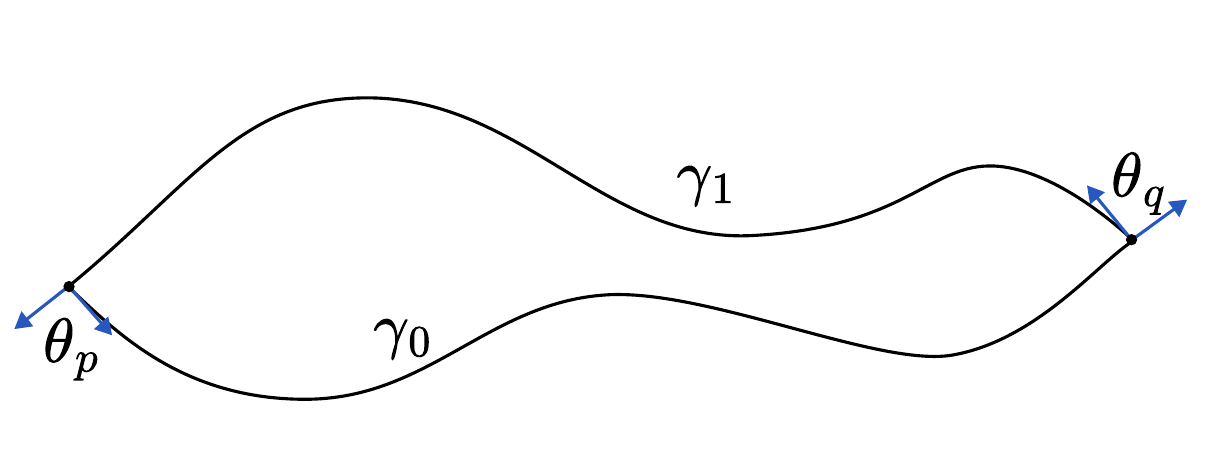}
    \caption{Two liftable paths bounding a disk}
    
\end{figure}

Let $\beta_0$ and $\beta_1$ be lifts of $\gamma_0$ and $\gamma_1$. Let $\theta_i(t) = Arg[\beta_i(t)(\gamma'_i(t))]$. Since $\beta_i$ is a maximal lift, $-\pi/3 \leq \theta_i(t) \leq \pi/3$. The turning contributions from $\gamma_i$ can be expressed using $\theta_i$.
\[T(\gamma_0) = \theta_0(1) - \theta_0(0) - 2\pi k_0/3\]
\[T(\gamma_1) = \theta_1(1) - \theta_1(0) + 2\pi k_1/3\]
Here $k_i$ are non-negative integers counting the extra turning contribution from where $\gamma_i$ passes through zeros. If $\gamma_i(t)$ is a zero of $\alpha$, then sometimes the pair $(\gamma_i,\beta_i)$ cannot be isotoped into the interior of $D$ in such a way that $\beta_i$ is still a maximal lift. In this situation, $\theta_i(t+\epsilon) - \theta_i(t-\epsilon)$ differes from the turning angle of $\gamma_i$ in $[t-\epsilon,t+\epsilon]$ by a multiple of $2\pi/3$. The fact that $\gamma_i$ are geodesic near zeros means that the turning angle corrections from zeros can only be negative for $\gamma_0$ and positive for $\gamma_1$.

The turning angles at $p$ and $q$  can also be expressed using $\beta_i$:
\[\theta_p = \pi - [\theta_0(1)-\theta_1(0)] - 2\pi k_p/3\]
\[\theta_q = \pi - [\theta_0(0)-\theta_1(1)] - 2\pi k_q/3\]
Here the extra factors of $2\pi/3$ come from the possibility that $\beta_0(0)\neq \beta_1(0)$ or $\beta_0(1)\neq \beta_1(1)$. This possibility can only give a negative contribution to $\theta_p$ or $\theta_q$. 

The total turning around $\del D$ is 
\[T(\gamma_0) + \theta_q - T(\gamma_1) + \theta_p = 2\pi - (k_0 + k_1 + k_p + k_q) 2\pi/3.\] The curvature of $D$ is non-positive, so from the Gauss Bonnet formula we conclude that the curvature of $D$ is zero, and $k_0 = k_1 = k_p = k_q = 0$. This means that there are no zeros in $D$, that $\beta_0 = \beta_1$ at $0$ and $1$, and that $\beta_i$ can both be perturbed into the interior of $D$, so there is a continuous lift $\bar{D}\to \Sigma$ which restricts to $\beta_i$ on the boundary. Since $\beta$ is a closed $1$-form, we conclude that
\[\int_{\gamma_0}2Re(\beta_0) = \int_{\gamma_1}2Re(\beta_1)\]
so $\gamma_0$ and $\gamma_1$ have the same length.

\end{proof}

\begin{lemma}
\label{cubic current corridor}
Let $\eta:[0,1]\to \tilde{S}$ be an arc in $\tilde{S}$ on which $\alpha$ is purely imaginary. Assume $\eta$ avoids zeros. Let $\gamma_0$ and $\gamma_1$ be (generalized) descending trajectories through $\eta(0)$, and $\eta(1)$ respectively which are perpendicular to $\eta$. Let $\gamma_i(-\infty)$ and $\gamma_i(\infty)$ be the endpoints in $\del_-^h\tilde{S}$ and $\del_+^h \tilde{S}$.
\[b(\gamma_1(-\infty),\gamma_2(-\infty);\gamma_2(\infty),\gamma_1(\infty)) = 2F^\Delta_\alpha (\eta)\]

\end{lemma}
\begin{proof}
    Make a path $\gamma_{12}$ by concatenating $\gamma_1|_{(-\infty,0]}$, $\eta$, and $\gamma_2|_{[0,\infty)}$, and make a path $\gamma_{21}$ by concatenating $\gamma_2|_{(-\infty,0]}$, $\eta^{-1}$, and $\gamma_1|_{[0,\infty)}$. 
    The paths $\gamma_{12}$ and $\gamma_{21}$ are geodesics by Lemma \ref{geodesic characterization}. 
    Holonomy in $U_h$ of the sequence of geodesics $\gamma_1$, $\gamma_{12}$, $\gamma_2$, $\gamma_{21}$ is $2F^\Delta_\alpha (\eta)$. More immediately it is $F^\Delta_\alpha (\eta) + F^\Delta_\alpha (\eta^{-1})$ but $F^\Delta_\alpha$ is symmetric on $\eta$. 
\end{proof}

\begin{figure}[h]
    \centering
    \includegraphics[width=5cm]{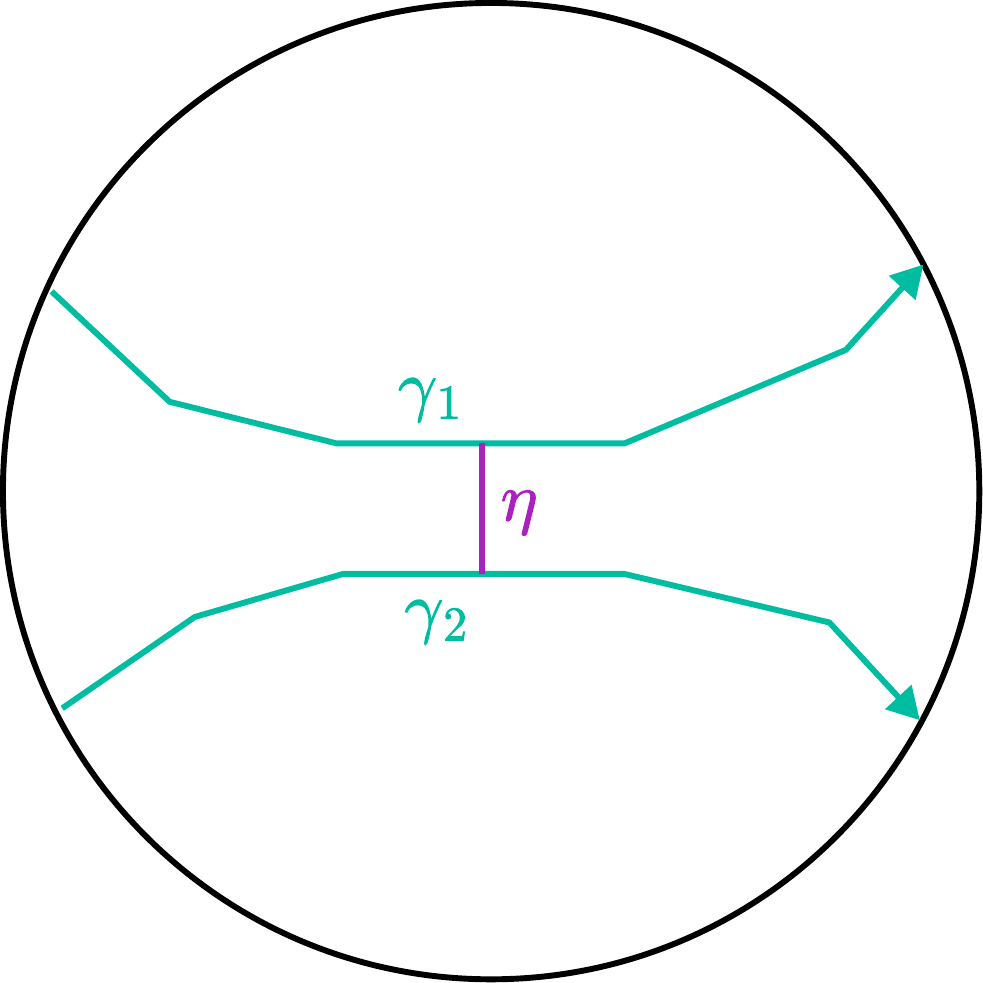}
    \caption{Two real trajectories perpendicular to a segment of imaginary trajectory}
    
\end{figure}

Let $\mathcal{T}'(\alpha)$ be the space of parameterized descending trajectories of $\alpha$ with the topology of uniform convergence on compact sets. Let $\mathcal{T}$ be the quotient by reparameterization. Note that $\mathcal{T}$ maps continuously to $\mathcal{G}^h \subset \del^h_- \tilde{S} \times \del^h_+ \tilde{S}$. Define $\bar{\mathcal{T}}'$ to be the closure of $\mathcal{T}'$ with respect to uniform convergence on compact sets, and $\bar{\mathcal{T}}$ to be its quotient by reparametrization. The closure $\bar{\mathcal{T}}$ will contain trajectories with zeros that either always turn left or always turn right.  A trajectory is determined by its horofunction endpoints $([g],[h])$ as the minimum set of $g+h$, so we may view $\mathcal{T}$ and $\bar{\mathcal{T}}$ as subsets of $\mathcal{G}^h$. 

\begin{lemma}
\label{cubic current support}
The geodesic current associated with $F^\Delta_\alpha$ has support $\bar{\mathcal{T}}(\alpha)$.
\end{lemma}
\begin{proof} 
    Suppose $([g],[h])\in \mathcal{T}(\alpha)$ are the endpoints of a trajectory $\gamma\in \mathcal{T}'(\alpha)$. Choosing a perpendicular segment $\eta$ to a point $p$ on $\gamma$ will determine arbitrarily small boxes containing $([g],[h])$ which have positive measure by Lemma \ref{cubic current corridor}, thus showing $([g],[h])$ is in $\supp(\mu_h)$. Since the support is closed by definition, it contains $\bar{\mathcal{T}}(\alpha)$

    Suppose $([g],[h])$ is not the endpoints of a trajectory. Choose a geodesic $\eta$ connecting $v([g])$ to $v([h])$. Choose any point $p$ on $\eta$ which is not a zero. There are three trajectories going through $p$. If these trajectories run into zeros, meaning that there are choices to make, choose consistant turns so that all three trajectories are in $\bar{\mathcal{T}}(\alpha)$. Since $\eta$ is not a trajectory, or in the closure of trajectories, it cannot coincide with any of these three trajectories. However, it is possible that $\eta$ is asymptotic to (or coincides with) one of the trajectories in the forward or backward direction. If this is the case, shift $p$ slightly to the left or right so that all three trajectories through $p$ cross $\eta$. There will be two trajectories $\gamma_1$ and $\gamma_2$ such that the tip of $\eta$ is between their tips and the tail of $\eta$ is between their tails. The cross ratio $b(\gamma_1(-\infty),\gamma_2(-\infty);\gamma_1(\infty),\gamma_2(\infty))$ vanishes. To see this note that the paths 
    \[\gamma_{12}:= (\gamma_1|_{(-\infty,0]})\circ (\gamma_2|_{[0,\infty)})\]
    \[\gamma_{21}:= (\gamma_2|_{(-\infty,0]})\circ (\gamma_1|_{[0,\infty)})\]
    are geodesics, and also pass through $p$. We have constructed a box of zero measure containing $([g],[h])$.
\end{proof}
\begin{remark}
Lemmas \ref{cubic current corridor} and \ref{cubic current support} completely specify the ``horofunctional" geodesic current $\mu^h$ because they give the measure of arbitrarily small boxes around any point in $\mathcal{G}^h$. Consequently, they uniquely specify the geodesic current $\mu = v_*(\mu^h)$. One could imagine an alternative construction of $\mu$ by first defining a measure on the space of trajectories $\mathcal{T}(\alpha)$, then pushing forward to $\mathcal{G}$. 
\end{remark}


\subsection{Lower submeasures of cubic differential currents}
Let $S$ be equipped with complex structure and cubic differential $\alpha$. In the previous subsection we investigated the Lioville current of the triangular Finsler metric $F^\Delta_\alpha$ and found that it is the current of real trajectories of $\alpha$. Let $\mu^h$ be this Lioville current on $\mathcal{G}^h(X)$, and let $\mu$ be its pushforward to $\mathcal{G}$. 

\begin{lemma}
\label{horofunctional lower submeasures are same}
    Let $v:\mathcal{G}^h\to \mathcal{G}$ be the projection. For every lower submeasure $\nu$ of $\mu$ there is a unique lower submeasure $\nu^h$ of $\mu^h$ such that $v_* (\nu^h) = \nu$.
\end{lemma}
\begin{proof}
    If two trajectories of $\alpha$ have the same visual endpoints $(x,y)\in \mathcal{G}$ then they cannot cross. This means that the support of $\mu^h$ is totally ordered on the preimage of each $(x,y)\in \mathcal{G}$. 
\end{proof}
Lemma \ref{horofunctional lower submeasures are same} allows us pass freely back and forth between lower submeasures of $\mu$ and $\mu^h$. If $\nu<\mu$ is a lower submeasure, by ``maximal trajectories of $\nu$" we mean the generalized trajectories specified by maximal support points of $\nu^h$.




\begin{lemma}
\label{lower submeasure of cubic differential currents}
    If $\nu$ is an admissible lower submeasure of a cubic differential current $\mu$, then it has finitely many maximal trajectories $T_1,...,T_k$ which can be ordered so that either the boundaries of the open right half spaces bounded by $T_i$ are an open cover of $\del\Gamma$ such that only adjacent intervals intersect, or this is true of the open left half spaces. 
\end{lemma}

\begin{center}
    \includegraphics[width=\linewidth]{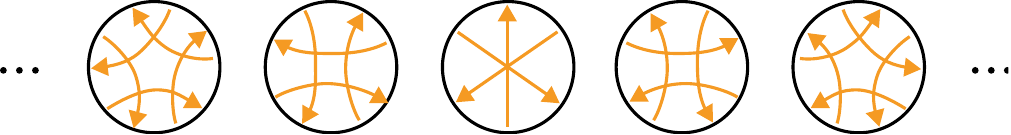}
\end{center}
\begin{proof}
The support of $\nu^h$ is a closed order ideal of $\bar{\mathcal{T}}(\alpha)$. 
Let $\max(\nu^h)$ denote the set of maximal support points of $\nu^h$. Since $\supp(\nu^h)$ is closed, it is generated as an order ideal by $\max(\nu^h)$. 

First we show that each trajectory in $\max(\nu^h)$ crosses at most two others. Note that $\max(\nu^h)$ has no $3$-intersection because this is true of $\mu^h$ by Lemma \ref{no n-intersection}, and has no pairs of parallel trajectories as that would render one of the trajectories non-maximal. 
If three trajectories $T_1,T_2,T_3\in \max(\nu^h)$ all cross $T_0\in \max(\nu^h)$ then at least two out of the three must cross in the same direction. Without loss of generality suppose $T_1$ and $T_2$ cross in the same direction. If $T_1,T_2$ cross, then they give $3$-intersection with $T_0$, but if they don't cross they are parallel, giving a contradiction. 

Next, assuming that $\max(\nu^h)$ has at least 4 elements, we choose a half-space bounded by each $T\in \max(\nu^h)$. There is at least one element of $\max(\nu^h)$ which doesn't intersect $T$, and all such elements must be on the same side of $T$. Let $H_T$ denote the half-space bounded by $T$ not containing any other trajectories of $\max(\nu^h)$ which don't intersect $T$.

Call a maximal trajectory $T$ positive if it is oriented such that $H_T$ is its right half space and negative otherwise. If $T$ and $T'$ are maximal trajectories which don't intersect then they must have the same sign, otherwise they are parallel. This means that if $\max(\nu^h)$ has at least five trajectories, they all have the same sign.

Suppose all trajectories have positive sign. Then the boundaries of the right half spaces $H_T$ are a collection of intervals in $\del\Gamma$ such that only adjacent intervals intersect. If the interiors of these intervals failed to cover $\del\Gamma$, then the complement of the half spaces would be unbounded. Choosing $p$ in the complement of the interiors of these half spaces, there would be an infinite measure of trajectories in $\nu(p)$ but not in $\nu$, so $\nu$ would not be admissible. 
The intervals thus must cover. 
Since $\del\Gamma$ is compact, the number of intervals must be finite. 

Suppose all trajectories have a negative sign. Then left half spaces are a collection of intervals in $\del\Gamma$ such that only adjacent intervals intersect. If these failed to cover, then, choosing $p$ in the complement of the interiors of these half spaces, there would be infinite measure of trajectories not in $\nu(p)$ but in $\nu$. 
This again would contradict admissibility of $\nu$.

Suppose there are only four maximal trajectories of $\nu^h$, and they have alternating signs. Then they bound a quadrilateral which has alternating oriented edges, and whose interior angles are all $2\pi/3$ with respect to the Euclidean cone metric. This is forbidden by Gauss-Bonnet.
\end{proof}

A point $p\in \tilde{S}$ determines a partition of the set of trajectories $\mathcal{T}$ into those which go counter-clockwise around $p$ and those which go clockwise, (and the measure zero set of trajectories going through $p$).
The set of trajectories going counter-clockwise is an order ideal, and thus gives a lower submeasure $\nu^h(p)$ of $\mu^h$. Let $\nu(p)$ denote the pushforward $v_*(\nu^h(p))$. 

\begin{lemma}
\label{holonomy zero means comes from point}
    An admissible partition $\mu = \nu + \bar{\nu}$ of a cubic differential current $\mu$ is holonomy zero if and only if its maximal trajectories all intersect at a common point $p$, in which case $\nu = \nu(p)$.
\end{lemma}

\begin{proof}
Recall that $\nu_h(p)$ is defined as the lower submeasure of $\mu_h$ consisting of the closure of all trajectories which pass to the right of $p$. The maximal support points of $\nu_h(p)$ must pass through $p$ because if a trajectory doesn't pass through $p$ it can be shifted to the left and still go to the right of $p$. Conversely, suppose we have another lower submeasure $\nu'_h$, all of whose maximal trajectories pass through $p$. $\nu'_h$ must be a submeasure of $\nu_h(p)$, otherwise it would have a trajectory, thus a maximal trajectory, passing to the left of $p$. If $\nu'_h$ is also holonomy zero, then it must coincide with $\nu_h(p)$. Now all we have to show is that if $\nu_h$ is holonomy zero then all of its maximal trajectories pass through a common point.

Denote by $\gamma_1,...,\gamma_k$ the unit speed parametrizations of the maximal trajectories of $\nu$, ordered as in Lemma \ref{lower submeasure of cubic differential currents} such that $\gamma_i(0)$ is the intersection point of $\gamma_i$ with $\gamma_{i-1}$.
Let $l_i\in \R$ be defined by $\gamma_i(l_i) = \gamma_{i+1}(0)$. The segments $\gamma_i|_{[0,l_i]}$ fit together in to a closed loop. We will show that the length of this loop, $l_1 + ... + l_k$, is the holonomy of $\nu$. Let $\eta_i$ be as follows.
\[
    \eta_i(t) = 
    \begin{cases}
        \gamma_i(t) & t\leq l_i \\
        \gamma_{i+1}(t-l_{i}) & t\geq l_i
    \end{cases}
\]
The paths $\eta_i$ are geodesic by lemma \ref{geodesic characterization}. Since $\gamma_i$ parameterize the maximal trajectories of $\nu^h$, the holonomy of $\nu^h$ is the holonomy of any monotonic taxi-path in $\mathcal{G}^h$ passing through each $\gamma_i$, so in particular the taxi-path given by the sequence of geodesics $\gamma_1,\eta_1,...,\gamma_n,\eta_n$. The holonomy of this loop is easily seen to be $l_1 + ... + l_k$.

By Lemma \ref{lower submeasure of cubic differential currents}, we know that $l_i$ together bound a $k$-gon in $\tilde{S}$ with consistently oriented sides (some of which may have zero length).  We conclude that $l_i$ are either all non-positive or non-negative. Consequently, if $\nu$ is holonomy zero, then all the $l_i$ vanish, and all the $\gamma_i$ must pass through a common point $p$. 
\end{proof}

The map $p\mapsto \nu(p)$ is thus a $\Gamma$-equivariant bijection from $\tilde{S}$ to $X_\mu$. 
We will show this map is an isometry for the symmetrized metric on $\tilde{S}$ by first lifting it to the relative metric space $L_U$ over $X_\mu$. 
First we define sections of $U^h_{F^{\Delta}_\alpha}$ associated with $p\in \tilde{S}$. 
Let $u^h(p)$ denote the horizontally flat section of $U^h_d$ which takes a pair $([g],[h])\in U^h_d$ to the pair $(g,h)$ such that $g(p)=0$, and let $v^h(p)$ denote the vertically flat section of $U^h_d$ which takes $([g],[h])\in U^h_d$ to the pair $(g,h)$ such that $h(p) = 0$.

\begin{lemma}
\label{horofunctional legendrian from point}
\[u^h(p) - v^h(p) = - \nu^h(\mathcal{G}^h_{\geq ([g],[h])}) - \bar{\nu}^h(\mathcal{G}^h_{\leq ([g],[h])})\]
\end{lemma}
\begin{proof}
The left hand side is the holonomy of a loop in $\mathcal{G}^h$ which starts at $([g],[h])$ moves vertically to a pair $([g],[h'])$ which passes through $p$, then rotates through pairs passing through $p$ to a pair $([g'],[h])$ passing through $p$, then moves horizontally back to $([g],[h])$. This loop will precisely enclose the collection of trajectories which are parallel to $([g],[h])$ and pass between $([g],[h])$ and $p$. The holonomy of a loop must be $\mu^h$ evaluated on its interior, with a sign depending on the orientation of the loop. This is the right hand side of the formula.
\end{proof}

To turn the sections $u^h(p)$ and $v^h(p)$ into sections of $U$, we pull back via a particular section of $v:\mathcal{G}^h\to \mathcal{G}^\circ$. 

By lemma \ref{unique horofunction} the maps from the horofunction boundaries $\del^h_+ \tilde{S}$ and $\del^h_-\tilde{S}$ to the Gromov boundary $\del \Gamma$ are bijective except over fixed points in $\del\Gamma$ of group elements. Preimages over fixed points are possibly non-trivial closed intervals in $\del^h_+ \tilde{S}$ and $\del^h_-\tilde{S}$. The cyclic orders on horofunction boundaries induce orders on these intervals. 
If $\gamma\in \Gamma$, let $\gamma^+_r$ and $\gamma^+_l$ denote the greatest and least lifts of $\gamma^+$ to $\del^h_+\tilde{S}$, and let $\gamma^-_r$ and $\gamma^-_l$ denote the greatest and least lifts of $\gamma^-$ to $\del^h_-\tilde{S}$. 

Let $\eta: \mathcal{G}^\circ \to \mathcal{G}^h$ be the section of which sends $(x,y)$ to $([g],[h])$ according to the rule that for each $\gamma\in \Gamma$:
\begin{itemize}
    \item If $x = \gamma^-$ and $\gamma^- < y < \gamma^+$ then $[g] = \gamma^-_r$.
    \item If $x = \gamma^-$ and $\gamma^+ < y < \gamma^-$, then $[g] = \gamma^-_l$.
    \item If $y = \gamma^+$ and $\gamma^- < x < \gamma^+$ then $[h] = \gamma^+_l$.
    \item If $y = \gamma^+$ and $\gamma^+ < x < \gamma^-$, then $[h] = \gamma^+_r$.
\end{itemize}

\begin{lemma}
\label{section order preserving}
    $(x,y) < (x',y')$ if and only if $\eta(x,y) < \eta(x',y')$.
\end{lemma}
Let $m_\nu^h = - \nu^h(\mathcal{G}^h_{\geq ([g],[h])}) - \bar{\nu}^h(\mathcal{G}^h_{\leq ([g],[h])})$ be the right hand side of lemma \ref{horofunctional legendrian from point}. Lemma \ref{section order preserving} implies that the potential $m_\nu$ is the same as the pullback $m_\nu^h\circ \eta$. Let $v(p) = v^h(p) \circ \eta$, and let $u(p) = u^h(p) \circ \eta$. These are vertically and horizontally flat sections of $U_{F^\Delta_\alpha}$. By Lemma \ref{horofunctional legendrian from point}, $m_\nu = u - v$. Define $\mathcal{L}(p) \in L_U$ to be the triple $(\nu(p),v(p),u(p))$. 

\begin{lemma}
    The map $p$ is an isometric embedding $\tilde{S}\to L_U$. 
\end{lemma}
\begin{proof}
    Let $p,q\in \tilde{S}$. The distance from $\mathcal{L}(p)$ to $\mathcal{L}(q)$ is defined to be the supremum of $v(q) - v(p)$. If we instead took the supremum of $v^h(q) - v^h(p)$ this would mean the supremum over all minus horofunctions $g$.
    \[\sup_{[g]\in \del^h_- \tilde{S}} g(q) - g(p)\]
    This is bounded above by $d(p,q)$ by the triangle inequality. 
    Letting $\gamma:(-\infty,0]$ be any geodesic ray which passes through $p$ and ends at $q$, (which exists, take for instance a geodesic for the Euclidean metric $|\alpha|^{2/3}$,) and letting $[g] = \gamma(-\infty)$ shows that indeed the supremum
    is equal to $d(p,q)$.

    Let $f_{p,q}:\del^h_- \tilde{S} \to \R$ denote the function $[g]\mapsto g(q) - g(p)$. It will suffice to show that $f_{p,q}$ is monotonic on $v_-^{-1}(\gamma^-)$ for all $\gamma\in \Gamma$ because then it suffices to take the supremum over the points of $\del^h_- \tilde{S}$ not fixed by group elements, which is the same as the supremum of $v(q) - u(p)$, which is $d(\mathcal{L}(p),\mathcal{L}(q))$. 

    Suppose for contradiction that $f_{p,q}$ is not monotonic on $v_-^{-1}(\gamma^-)$. For any $[g],[g']\in \del^h_-\tilde{S}$, $f_{p,q}([g']) - f_{p,q}([g])$ is the measure of set of trajectories which start between $[g]$ and $[g']$ and pass between $p$ and $q$, minus the measure of the set of trajectories which start between $[g]$ and $[g']$ and pass the other way between $p$ and $q$. It is impossible to have two trajectories starting at the same Gromov boundary point $\gamma^-$ and passing opposite directions between $p$ and $q$.
\end{proof}

Now we have lifted the bijection $\tilde{S}\to X_\mu$ to an isometric embedding $\tilde{S}\to L_U$ for the Finsler metric $F^\Delta_\alpha$. It follows that the bijection $\tilde{S}\to X_\mu$ is an isometry for the symmetrized metric $F^\Delta_\alpha + F^\Delta_{-\alpha}$.

\appendix
\section{Appendix: A symplectic perspective on negative curvature}
Many of the constructions of this paper, in particular the definition of $X_\mu$ are motivated by the symplectic perspective on negative curvature pioneered by Otal \cite{Otal92}, and with roots going back to Arnold and Hilbert. In much of this paper, issues of regularity, and the fact that we worked in two dimensions might have obscured the symplectic geometry, so we describe the picture here. The passage from geometry to symplectic geometry is the same as always: instead of looking at $X$ we look at $T^*X$, but the story plays out in a particular way when $X$ is a negatively curved manifold.

Let $X$ be a Hadamard manifold: a simply connected complete Riemannian manifold with sectional curvature bounded above by $\epsilon < 0$. The exponential map is a diffeomorphism $T_x X\to X$ for any $x\in X$. The visual boundary $\del\! X$ is the set of geodesic rays $\gamma:[0,\infty)\to X$ modulo the equivalence relation $\gamma_1 \sim \gamma_2$ if $d(\gamma_1(t),\gamma_2(t))$ is bounded. The visual boundary is naturally identified with the unit tangent sphere at any point. 

Let $\mathcal{G}$ be the space of oriented, unparametrized geodesics in $X$. Since $X$ is Hadamard, a geodesic is encoded by its visual endpoints.
\[\mathcal{G} = \del X\times \del X \backslash \Delta\]
We will see that $\mathcal{G}$ has a natural symplectic structure. The cotangent bundle $T^*\!X$ has a cannonical symplectic structure. Under the identification of $TX$ with $T^*\!X$ by the metric, geodesic flow becomes Hamiltonian flow of the inverse metric. The space of geodesics $\mathcal{G}$ is thus symplectic reduction of $T^*X$. 

Let $U$ be the unit cotangent bundle of $X$. It is useful to view $U$ as a principal $\R$ bundle over $\mathcal{G}$ where the $\R$ action is geodesic flow. The tautological $1$-form $\lambda$ on $T^*X$ restricts to a contact form on $U$, which can also be viewed as a connection for this $\R$ bundle. This connection has a simple geometric origin back on $X$: a path of unit cotangent vectors $(x_s,\alpha_s)$ for $s\in \R$ is a flat section of $U$ if $\alpha_s(x'_s) = 0$. It is also useful to view $U$ as the space of parametrized geodesics. The connection $\alpha$ declares a path of parametrized geodesics $\gamma_s(t)$ to be flat if $\del_s (\gamma_s(t))$ is perpendicular to $\del_t(\gamma_s(t))$. 
\begin{figure}[h]
    \centering
    \includegraphics[width=0.3\linewidth]{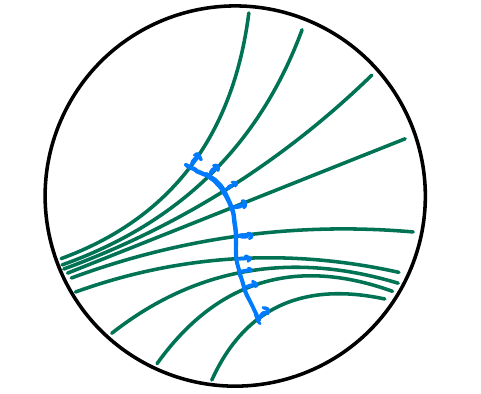}
    \caption{A path of geodesics, with a flat section to the unit tangent bundle}
    
\end{figure}
In particular, flat sections over a path of geodesics which all start (or end) at the same point are outward (inward) unit normal bundles to horosoheres.

Points in $X$ correspond to Legendrian spheres in $U$, namely a point corresponds to its unit cotangent sphere. 
More generally, the unit conormal bundle of a submanifold $Y\subset X$ is a Legendrian submanifold of $U$. 
Legendrians in $U$ project to (possibly singular) Lagrangians in $\mathcal{G}$. Lagrangians in $\mathcal{G}$ which admit Legendrian lifts to $U$ are called exact. The Lagrangian corresponding to a point $p\in X$ is the sphere of geodesics passing through $p$, and has the special property that it projects homeomorphically to both factors of $\del X$. We call any such Lagrangian sphere monotonic, and call a Legendrian sphere in $U$ monotonic if it projects to a monotonic Lagrangian sphere. 

Suppose we only have the boundary $\del X$, the symplectic form on $\mathcal{G} = \del X\times \del X$, and the principal $\R$ bundle $U\to \mathcal{G}$ whose curvature is the symplectic form, and we want to reconstruct $X$. This situation first arose in relation to questions of marked length spectrum rigidity \cite{Otal90}, but for us the motivation was finding a geometric incarnation of geodesic currents. 
It seems difficult to reconstruct $X$, but as a replacement we could consider the space $L_U$ of all monotonic Legendrian spheres in $U$. Perhaps better, we can quotient $L_U$ by the $\R$ action of geodesic flow and get the space $X_U$ of exact monotonic Lagrangian spheres. The idea behind definition \ref{X definition} is to, in the case $\dim(X)=2$, make a definition of exact monotonic Lagrangian sphere which is robust enough that it makes sense even when the ``symplectic structure" on $\mathcal{G}$ is very singular. 

In the special case when $\dim(X) = 2$, the symplectic structure on $\mathcal{G}$ is just a measure. A monotonic Lagrangian in $\mathcal{G}$ is just a path which is the graph of a fixed-point free monotonic function $\del X\to \del X$. Such a path determines a partition of $\mathcal{G}$, and thus the measure. This partition will be holonomy zero if and only if the Lagrangian was exact.

Everything in this section works more generally for Finsler manifolds of negative curvature. The identification of unit tangent bundle with unit cotangent bundle is achieved by the Legendre transform which identifies $v$ with $\alpha$ if $\alpha(v)=1$. An important subtlety to mention is that for asymmetric Finsler metrics there are two types of horospheres, one for backward endpoints of geodesics and one for forward endpoints of geodesics. In section \ref{sec:currents from metrics}, we generalized this picture, in the two dimensional case, to a class of Finsler metrics whose unit balls are not necessarily strictly convex. In this case the symplectic structure on $\mathcal{G}$ can become quite degenerate, and even concentrate onto a discrete or cantor subset of $\mathcal{G}$. 

\printbibliography

\end{document}